\documentclass[leqno]{amsart}
\usepackage{amssymb}
\usepackage{mathrsfs}
\usepackage{amsmath, amsfonts, vmargin, enumerate}
\usepackage{graphics}
\usepackage{color}
\usepackage{verbatim}

\usepackage{amsthm}
\usepackage{latexsym, bm}
\usepackage{euscript}
\usepackage{dsfont}

\input xypic
\xyoption {all}

 \makeatletter

\newtheorem{thm}{Theorem}[section]
\newtheorem{prop}[thm]{Proposition}
\newtheorem{cor}[thm]{Corollary}
\newtheorem{lem}[thm]{Lemma}
\newtheorem{defn}[thm]{Definition}
\newtheorem{rem}[thm]{Remark}
\newtheorem{example}[thm]{Example}
\newtheorem{pb}[thm]{Problem}
\newtheorem{conj}[thm]{Conjecture}

\newenvironment{rmk}{\begin{rem}\rm}{\end{rem}}

\numberwithin{equation}{section}

\newcommand{\supp}{{\rm supp\,}}
\newcommand{\bmo}{{\rm bmo}}
\newcommand{\BMO}{{\rm BMO}}
\newcommand{\h}{{\rm h}}
\newcommand{\M}{\mathcal M}
\newcommand{\N}{\mathcal N}
\newcommand{\F}{\mathcal F}

\newcommand{\R}{\mathbb{R}^d}

\newcommand{\e}{\varepsilon}
\newcommand{\T}{\mathbb{T}^d}
\newcommand{\I}{\mathbb{I}^d}
\newcommand{\Z}{\mathbb{Z}^d}

\begin{document}

\title{Mapping properties of operator-valued pseudo-differential
operators}

\author{Runlian  XIA}

\address{Laboratoire de Math{\'e}matiques, Universit{\'e} de Franche-Comt{\'e},
25030 Besan\c{c}on Cedex, France, and Instituto de Ciencias Matem{\'a}ticas, 28049 Madrid, Spain}
\email{runlian91@gmail.com}



\author{Xiao XIONG}

\address{Department of Mathematics and Statistics, University of Saskatchewan, Saskatoon, Saskatchewan, S7N 5E6, Canada}
\email{ufcxx56@gmail.com}

\maketitle

\markboth{R. Xia and X. Xiong}%
{Pseudo-differential operators}

\begin{abstract}
In this paper, we investigate the mapping properties of pseudo-differential operators with operator-valued symbols. Thanks to the smooth atomic decomposition of the operator-valued Triebel-Lizorkin spaces $F_1^{\alpha,c}(\R,\M)$ obtained in our previous paper, we establish the $F_1^{\alpha,c}$-regularity of regular symbols for every $\alpha\in \mathbb{R}$, and the $F_1^{\alpha,c}$-regularity of forbidden symbols for $\alpha>0$. As applications, we obtain the same results on the usual and quantum tori.
\end{abstract}

\tableofcontents

 \setcounter{section}{-1}

\section{Introduction}

Pseudo-differential operators were first explicitly defined by Kohn-Nirenberg  \cite{Kohn-Nirenberg} and H\"{o}rmander \cite{Hormander-65} to connect singular integrals and differential operators. They can be viewed as generalizations of Fourier multipliers, i.e., those operators acting on functions on $s$, formally determined by 
$$T(e^{2\pi {\rm i} s\cdot \xi} ) =   \sigma(\xi) e^{2\pi {\rm i} s\cdot \xi}, \quad \forall\, \xi \in \R. $$
In this sense, $\sigma(\xi)$ is called the symbol of the operator $T$. If $T$ is one of those more general operators, it is characterized by its symbol $\sigma(s,\xi)$, which is now a function of $s$ as well as $\xi$, i.e., 
$$
T(e^{2\pi {\rm i} s\cdot \xi} ) =   \sigma(s, \xi) e^{2\pi {\rm i} s\cdot \xi}  . $$
 Using the inverse Fourier transform, this characterization looks like
\begin{equation}\label{charact-iF}
Tf(s)=\int_{\mathbb{R}^{d}}\sigma(s,\xi)\widehat{f}(\xi)e^{2\pi {\rm i}s\cdot \xi}d\xi.
\end{equation}  
To emphasize the role of the symbol $\sigma$, we often write $T$ as $T_\sigma$. And we call this $T_\sigma$ a pseudo-differential operator.

Here are some examples of pseudo-differential operators. If $\sigma$ is independent of the variable $s$, then we go back to the Fourier multiplier mentioned above. On the other hand, if $\sigma$ is independent of the variable $\xi$, then by \eqref{charact-iF}, we get $Tf(s) = \sigma(s) \cdot f(s)$, the pointwise multiplication. To give an example of pseudo-differential operator whose symbol is a function of both $s$ and $\xi$, we consider the partial differential operator $L= \sum_{|m|_1\leq k  }  a_m(s) D^m _s $, where $m\in \mathbb{N}_0^d$ and $|m|_1 = m_1+\cdots + m_d$. This time, by \eqref{charact-iF} again, we know that the symbol of $L$ is 
$$\sigma(s,\xi)= \sum_{|m|_1\leq k  }  a_m(s) (2\pi {\rm i} \xi)^m .$$
For a general symbol $\sigma$, $T_\sigma$ may be thought as a limit of linear combinations of operators composed by pointwise multipliers and Fourier multipliers.

The study of pseudo-differential operators connects that of partial differential operators with harmonic analysis. More precisely, the regularity of solutions of a PDE corresponds to the boundedness of the related pseudo-differential operator on some function spaces. This amounts to one of the most important problems in pseudo-differential operator theory: the mapping properties of these operators on various  function spaces. Given $n\in \mathbb{R}$ and $0\leq \delta, \rho \leq 1$, denote by $S_{\rho,\delta}^n$ the H\"{o}rmander class of symbols, consisting of all infinitely differentiable functions $\sigma: \R\times \R \rightarrow \mathbb{C}$ such that 
\begin{equation}\label{def-scaler-Hormander}
| D_{s}^{\gamma}D_{\xi}^{\beta}\sigma(s,\xi)|    \leq C_{\gamma,\beta}(1+\vert\xi\vert)^{n+\delta |\gamma |_1-\rho |\beta |_1}
\end{equation}
for all $s, \xi \in \R$. One may ask which kind of symbol classes give pseudo-differential operators that are bounded on $L_p$-spaces, Sobolev, Besov, Hardy or Triebel-Lizorkin spaces. In general, it is known that pseudo-differential operators are not necessarily bounded on the classical Hardy space $\mathcal{H}_1(\mathbb{R}^d)$, or homoegeous Besov and Triebel-Lizorkin spaces. As a result, when studying the mapping properties of pseudo-differential operators, one usually focuses on inhomogeneous function spaces, such as the local Hardy spaces $\h_p(\R)$ defined by Goldberg \cite{Goldberg1979}, or inhomogeneous Besov and Triebel-Lizorkin spaces (see Triebel \cite{Tri} and \cite{Tri1992} for the definitions). For details on these results in the classical setting, we refer to \cite{Cal-Val-1972, CM1978, Runst-1985, Bourdaud-1988, Torres-1990, Tri1992, Stein1993}.

\medskip

In the noncommutative setting, this line of research started with Connes' work \cite{Connes1980} on pseudo-differential calculus for $C^*$-dynamical systems, which intended to extend the Atiyah-Singer index theorem \cite{Atiyah-Singer} for Lie group actions on $C^*$-algebras. At that time, due to the fact that very little had been done about the analytic aspect, the work of Connes and his collaborators did not include $L_p$-estimates for parametrices and error terms. Recently, inspired by the development on noncommutative harmonic analysis, a lot of progress has been made on Fourier multiplier theory and Calder\'on-Zygmund theory on noncommutative $L_p$ spaces, thanks the efforts of many researchers \cite{Mei2007, Parcet, N-R, CXY2013, HLMP2014, JMP2014, XXX17, XXY17}.
But so far, the mapping properties of pseudo-differential operators are rarely studied.

In this paper, we consider the boundedness of noncommutative pseudo-differential operators.
Our definition of symbol classes is modelled on the classical definition by H\"{o}rmander; the idea is to consider those operator-valued functions $\sigma: \R\times \R \rightarrow \M$ satisfying \eqref{def-scaler-Hormander} with operator norms in place of absolute values of the derivatives of $\sigma$. Here $\M$ is a von Neumann algebra. If $f :\R \rightarrow L_1(\M)+\M$ is a good enough function, we can consider the action of a pseudo-differential operator with symbol $\sigma$ on this $f$. Because of the noncommutativity, we have two different actions:
\begin{equation}\label{intro-pdo-c}
T_\sigma ^cf(s)=\int_{\mathbb{R}^{d}}\sigma(s,\xi)\widehat{f}(\xi)e^{2\pi {\rm i}s\cdot \xi}d\xi
\end{equation}
and 
\begin{equation*}
T_\sigma^r f(s)=\int_{\mathbb{R}^{d}}\widehat{f}(\xi)\sigma(s,\xi) e^{2\pi {\rm i}s\cdot \xi}d\xi.
\end{equation*}
We will mainly work on the column operators $T_\sigma ^c$ and establish their mapping properties on local Hardy spaces $\h_p^c(\R, \M)$ (introduced in \cite{XX17}) and inhomogeneous Triebel-Lizorkin spaces $F_p^{\alpha,c} (\R, \M)$ (introduced in \cite{XX18}). The main part of the proof concerns the case $p = 1$ for both kinds of spaces. Compared to the standard proof of the boundedness on Hardy spaces of a usual Calder\'on-Zygmund operator with a commutative or noncommutative convolution kernel, the present proof is much subtler and more technical. We need a careful analysis of a pseudo-differential operator on smooth (sub)atoms given in \cite{XX18}.

Now we state the main results of this paper. We concentrate on the pseudo-differential operators with operator-valued symbols in $S_{1,\delta}^n$, the class of infinitely differentiable functions $\sigma: \R\times \R \rightarrow \M$ such that 
\begin{equation*}
\| D_{s}^{\gamma}D_{\xi}^{\beta}\sigma(s,\xi) \|_\M  \leq C_{\gamma,\beta}(1+\vert\xi\vert)^{n+\delta |\gamma |_1- |\beta |_1}, \quad  \forall \, s, \xi \in \R.
\end{equation*}
If $0\leq \delta < 1$ and $\sigma \in S_{1,\delta}^0$, we will prove that the operator $T_\sigma^c$ in \eqref{intro-pdo-c} is bounded on $F_1^{\alpha,c} (\R,\M) $ for every $\alpha\in \mathbb{R}$, in particular on $\h_1^c(\R,\M)= F_1^{0,c} (\R,\M) $. Moreover, since the adjoint of $\sigma \in S_{1,\delta}^0$ still belongs to $S_{1,\delta}^0$ when $\delta<1$, we will deduce the boundedness on $F_p^{\alpha,c} (\R,\M) $  for $1<p \leq \infty$ from duality and interpolation. And for $\sigma\in S^n_{1,\delta}$, we can use the Bessel potential of order $n$ to connect $T_\sigma^c$ with $T_{\sigma'}^c $ for $\sigma'\in S_{1,\sigma}^0$, i.e. $T_\sigma^c = T_{\sigma'}^c \circ J^n$, so as to get the boundedness of $T_\sigma^c$ from $F_p^{\alpha,c} (\R,\M) $ to $F_p^{\alpha-n,c} (\R,\M) $. For the case $\delta=1$, we will also show that $T_\sigma^c$ is bounded on $F_1^{\alpha,c} (\R,\M) $ for every $\alpha>0$. But to get the boundedness on $F_p^{\alpha,c} (\R,\M) $ for $p>1$, we need more assumption on the symbol.

We then apply the outcome to the usual and quantum tori, and obtain parallel results in both cases. 

\medskip

Let us mention that independently  and at the same time, Gonz\'alez-P\'erez, Junge and Parcet developed in  \cite{P-J-P}  the pseudo-differential theory in quantum Euclidean spaces that are the non compact analogues of quantum tori. Although the two papers overlap in some ways, they are very different in nature in regard to both results and arguments. Their results concern the boundedness of a pseudo-differential  operator on the $L_p$-spaces with $1<p<\infty$, while the ours deal with this boundedness on a column Triebel-Lizorkin spaces $F_p^{\alpha,c}(\R,\M)$ with $\alpha\in \mathbb{R}$ and $1\leq p\leq \infty$. Note that the mixture Triebel-Lizorkin space $F_p^{\alpha}(\R,\M)$ coincides with $L_p(\N)$ when $\alpha=0$ and $1<p<\infty$. On the other hand, the arguments of \cite{P-J-P} are based on a careful analysis of the $L_2$ and $\BMO$ cases, while our proof in the case $p=1$ (the main case) relies entirely on the atomic decomposition of $F_1^{\alpha,c}(\R,\M)$ obtained in \cite{XX18}. 
 
\medskip

The paper is organized as follows. In the next section, we introduce some elementary notation and knowledge on noncommutative $L_p$-spaces, and the definitions of local Hardy spaces in \cite{XX17} and inhomogeneous Triebel-Lizorkin spaces in \cite{XX18}.  Then we present the smooth atomic decompositions of these spaces obtained in \cite{XX18}. In section \ref{section-pdo-def}, we give the concrete definitions and some easily deduced useful facts on operator-valued pseudo-differential operators. Section \ref{section-pdo-atom} is devoted to the study of the local mapping properties of pseudo-differential operators, i.e. their action on atoms. In sections \ref{section-regular} and \ref{section-forbidden}, we prove the mapping properties of pseudo-differential operators with regular and forbidden symbols respectively. The last section presents applications to the usual and quantum tori.

\medskip

We close this introduction section by the following convention on notation of inequalities. 
Throughout, we will use the notation $A\lesssim B$,
which is an inequality up to a constant: $A\leq cB$ for some constant
$c>0$. The relevant constants in all such inequalities may depend
on the dimension $d$, the test function $\Phi$ or $p$, etc, but
never on the function $f$ in consideration. The equivalence $A\approx B$
will mean $A\lesssim B$ and $B\lesssim A$ simultaneously.

 \section{Preliminaries on noncommutative analysis}\label{prelimi}

We begin with an introduction of notation and basic knowledge on vector-valued Fourier analysis, i.e., Fourier analysis on functions with values in a Banach spaces $X$.
Let $\mathcal{S}(\R; X)$ be the space of  $X$-valued rapidly decreasing  functions on $\R$ with the standard Fr\'{e}chet topology. In particular,  $\mathcal{S}(\mathbb R^d; \mathbb C)$ is simply denoted as $\mathcal{S}(\mathbb R^d)$.  Let $\mathcal{S}'(\R; X)$  be the space of continuous linear maps from $\mathcal{S}(\R)$ to $X$; the elements of $\mathcal{S}'(\R; X)$ are the so-called $X$-valued tempered distributions. All operations on $\mathcal{S}(\R)$ such as derivation, convolution and Fourier transform transfer to $\mathcal{S}'(\R; X)$ in the usual way. On the other hand, $L_p(\R; X)$ naturally embeds into $\mathcal{S}'(\R; X)$ for $1\leq p \leq \infty$, where $L_p(\R; X)$ stands for the space of strongly $p$-integrable functions from $\R$ to $X$. By this definition, Fourier transform and Fourier multipliers on $\R$ extend to vector-valued tempered distributions in a natural way.

We give some typical Fourier multipliers that will be frequently used in the sequel. For a real number $\alpha$, the Bessel potential is the operator $J^\alpha =(1-(2\pi)^{-2}\Delta  )^{\frac \alpha  2}$ defined on $\mathcal{S}'(\R; X)$, where $\Delta$ denotes the Laplacian on $\R$. If $\alpha=1$, we will abbreviate $J^1$ as $J$. We denote also $J_\alpha (\xi)=(1+|\xi|^2 )^{\frac \alpha  2}$ on $\mathbb{R}^d$. It is the symbol of the Fourier multiplier $J^\alpha $. Recall also the symbols of Littlewood-Paley decomposition on $\R$, which are used to define the Triebel-Lizorkin spaces.
Fix a Schwartz function $\varphi$ on $\mathbb{R}^d$ satisfying:
\begin{equation}\label{condition-phi}
\begin{cases}
\supp \varphi \subset \lbrace \xi:\frac{1}{2}\leq |  \xi |  \leq 2\rbrace.\\
\varphi >0 \mbox{ on } \lbrace \xi:\frac{1}{2}< |  \xi |  < 2\rbrace,\\
\sum_{k\in \mathbb{Z}}\varphi (2^{-k}\xi)=1, \forall \, \xi \neq 0.
\end{cases}
\end{equation}
For each $k \in \mathbb{N}$ let $\varphi_k$ be the function whose Fourier transform is equal to $\varphi(2^{-k}\cdot)$, and let $\varphi_0$ be the function whose Fourier transform is equal to $1-\sum_{k>0}\varphi (2^{-k}\cdot)$. Then $\{\varphi_k\}_{k\geq 0}$ gives a Littlewood-Paley decomposition on $\mathbb{R}^d$ such that
\begin{equation}
{\supp} \widehat\varphi_k\subset\{\xi\in\mathbb{R}^{d}:2^{k-1}\leq  |\xi |\leq2^{k+1}\},\quad \forall \, k\in\mathbb{N},\;\;\;\supp \widehat\varphi_0\subset\{\xi\in\mathbb{R}^{d}:  |\xi |\leq2\}\label{eq: supp varphi_j}
\end{equation}
and that
\begin{equation}
\sum_{k=0}^{\infty}\widehat\varphi_k(\xi)=1,\quad \forall \,\xi\in\mathbb{R}^{d}.\label{eq:resolution of unity}
\end{equation}

Other than the above Littlewood-Paley decomposition, we will need another kind of resolution of the unit on $\R$. Let $\mathcal{X}_0$ be a nonnegative infinitely differentiable function on $\R$ such that $\supp \mathcal{X}_0\subset 2Q_{0,0}$, and $\sum_{k\in\mathbb{Z}^d} \mathcal{X}_0(s-k)=1$ for every $s\in \R$. Here $Q_{0,0}$ is the unit cube centered at the origin, and $2Q_{0,0}$ is the cube with the same center, but twice the side length; see the end of this section for notation of general cubes. Set $\mathcal{X}_k = \mathcal{X}_0(\cdot-k)$. Then each $\mathcal{X}_k$ is supported in the cube $2Q_{0,k}=k+2Q_{0,0}$, and all $\mathcal{X}_k$'s form a smooth resolution of the unit:
\begin{equation}\label{eq: unit resolution}
1=\sum_{k\in\mathbb{Z}^d}\mathcal{X}_k(s), \quad \forall \, s\in \mathbb{R}^d. 
\end{equation}
This smooth resolution of the unit will often be used to divide functions or distributions into small pieces, which have the same smoothness as before, but have compact supports additionally.

\subsection{Noncommutative $L_{p}$-spaces}
Let us turn to the setting of operator-valued analysis, where the above involved Banach spaces $X$ are required to have some operator space structure now. In this paper, all function spaces in consideration are based on the noncommutative $L_p$-spaces associated to $(\M,\tau)$, where $\M$ is a von Neumann algebra $\tau$ is a normal semifinite faithful trace, and $1\leq p\leq\infty$.
The norm of  $L_p(\M)$ will be often denoted simply by $\|\cdot \|_p$. But if different  $L_p$-spaces  appear in a same context, we will sometimes precise the respective $L_p$-norms in order to avoid possible ambiguity. The reader is referred to \cite{PX2003} and  \cite{Xu2007} for more information on noncommutative $L_p$-spaces. We will also need Hilbert space-valued noncommutative $L_p$-spaces (see \cite{JLX2006} for more details).
Let $H$ be a Hilbert space and  $v \in H$ with $\|v\|=1$. Let  $p_v$ be the orthogonal projection onto
the one-dimensional subspace generated by $v$. Define
 $$L_p(\M; H^{r})=(p_v\otimes 1_{\M}) L_p(B(H)\overline\otimes \M)\;\textrm{ and }\;
 L_p(\M; H^{c})= L_p(B(H)\overline\otimes \M)(p_v\otimes 1_{\M}).$$
These are the row and column noncommutative $L_p$-spaces. They are 1-complemented subspaces of $L_p(B(H)\overline\otimes \M)$.

In most part of this paper, we are interested in operator-valued functions. The involved von Neumann algebra is the semi-commutative algebra $L_\infty(\R) \overline\otimes \M$ with tensor trace, denoted by $\N$ in the sequel.
We will frequently use the following Cauchy-Schwarz type inequality,
\begin{equation}\label{op-CS}
  \big|\int_{\mathbb{R}^d}\phi (s)f(s)ds \big|^{2}\leq\int_{\mathbb{R}^d}  |\phi(s) |^{2}ds\int_{\mathbb{R}^d}  |f(s) |^{2}ds,
\end{equation}
where $\phi:\mathbb{R}^d\rightarrow\mathbb{C}$ and $f:\mathbb{R}^d\rightarrow L_{1}(\mathcal{M} )+\mathcal{M}$
are functions such that all integrations of the above inequality make sense.
We will also require the operator-valued version of the Plancherel formula. For sufficiently nice functions $f: \mathbb{R}^{d}\rightarrow L_{1}  (\mathcal{M} )+\M$,
for example, for $f \in L_{2}  (\mathbb{R}^{d} )\otimes L_{2}  (\mathcal{M} )$,
we have
\begin{equation}\label{op-Plancherel}
\int_{\mathbb{R}^d} |f  (s )|^2 ds=\int_{\mathbb{R}^d}  |  \widehat{f}  (\xi )|^2  d\xi.
\end{equation}

\medskip

\subsection{Inhomogeneous Triebel-Lizorkin spaces}

We follow the presentation in \cite{XX18}, to give the definition of inhomogeneous Triebel-Lizorkin spaces.
Let $1\leq p<\infty$ and $\alpha \in \mathbb{R}$, and $\varphi$ be the Schwartz function determined by \eqref{condition-phi}.
 The column Triebel-Lizorkin space $F_p^{\alpha,c}(\R,\M)$ is defined by
$$
F_p^{\alpha,c}(\R,\M)=\lbrace f\in \mathcal{S}'(\R; L_1(\M)+\M): \|  f\| _{F_p^{\alpha,c}}<\infty\rbrace,
$$
where
$$
\|  f\| _{F_p^{\alpha,c}}= \big\|  (\sum_{j\geq 0}2^{2j\alpha}|  {\varphi}_j*f | ^2)^\frac{1}{2}\big\| _{L_p(\N)}.
$$
The row space $F_p^{\alpha,r}(\R,\M)$ consists of all $f$ such that $f^*\in F_p^{\alpha,c}(\R,\M)$, equipped with the norm $\|  f\| _{F_p^{\alpha,r}}=\|  f^*\| _{F_p^{\alpha,c}}$.
The mixture space $F_p^{\alpha}(\R,\M)$ is defined to be
\begin{equation*}
F_p^{\alpha}(\R,\M)=
\begin{cases}
F_p^{\alpha,c}(\R,\M)+F_p^{\alpha,r}(\R,\M) \quad \mbox{if}\quad  1\leq p\leq 2\\
F_p^{\alpha,c}(\R,\M)\cap F_p^{\alpha,r}(\R,\M) \quad \mbox{if}\quad  2<p<\infty,
\end{cases}
\end{equation*}
equipped with
\begin{equation*}
\|  f\| _{F_p^{\alpha}}=
\begin{cases}
\inf \lbrace \|  g\| _{F_p^{\alpha,c}}+\|  h\| _{F_p^{\alpha,r}}: f=g+h \rbrace  \quad \mbox{if}\quad  1\leq p\leq 2\\
\max \lbrace \|  f\| _{F_p^{\alpha,c}},\|  f\| _{F_p^{\alpha,r}} \rbrace  \quad \mbox{if}\quad  2< p<\infty.\\
\end{cases}
\end{equation*}
If $p=\infty$, define $F^{\alpha,c}_\infty(\R,\M)$ as the space of all $f\in \mathcal{S}'(\R; \M)$  such that 
$$
\|  f\|  _{F^{\alpha,c}_\infty}=\| \varphi_0*f\|_\N+\sup_{|Q|<1} \Big\|  \frac{1}{|  Q| }\int_Q\sum_{j\geq -\log_2(l(Q))}2^{2j\alpha}|  \varphi_j*f(s) | ^2ds \Big\| _\M^\frac{1}{2}<\infty,
$$
where $Q$ denotes a cube in $\R$, and $l(Q)$ its side length.
Let $1\leq p <\infty$, $\alpha\in \mathbb{R}$ and $q$ be the conjugate index of $p$. Then the dual space of $F_p^{\alpha,c}(\R,\M)$ coincides isomorphically with $F_q^{-\alpha,c}(\R,\M)$.

The Triebel-Lizorkin spaces form an interpolation scale with respect to the complex interpolation method: For $\alpha_0,\alpha_1\in \mathbb{R}$ and $1<p<\infty$, we have
$$
\big( F_\infty^{\alpha_0,c}(\R,\M), F_1^{\alpha_1,c}(\R,\M)\big)_{\frac{1}{p}}=F_p^{\alpha,c}(\R,\M), \quad \alpha=(1-\frac{1}{p})\alpha_0+\frac{\alpha_1}{p}.
$$
See \cite{XX18} for the proof of this interpolation.

When $\alpha=0$ and $1\leq  p<\infty$, it is proved in \cite{XX18} that  $ F_p^{0,c}(\R,\M) =\h_p^c(\R,\M)$ with equivalent norms, where $ \h_p^c(\R,\M)$ are the local Hardy spaces studied in \cite{XX17}. The lifting property of Triebel-Lizorkin spaces states that, for any $\beta\in \mathbb{R}$, $J^\beta$ is an isomorphism between $F_p^{\alpha,c}(\R, \M)$ and $F_p^{\alpha-\beta,c}(\R, \M)$. In particular, $J^\alpha$ is an isomorphism between $F_p^{\alpha,c}(\R, \M)$ and $\h_p^{c}(\R, \M)$. In this sense, these Triebel-Lizorkin spaces can be viewed as an extension of local Hardy spaces.
Moreover, when $1<p<\infty$, we have, with equivalent norms,
\begin{equation}\label{equi-Lp-hp}
L_p(\N) = \h_p(\R,\M) = F_p^{0} (\R,\M).
\end{equation}

\subsection{Atomic decompositions} We begin with the case $\alpha=0$, i.e., the atomic decomposition of local Hardy space $\h_1^c(\R,\M)$. Much as in the classical case, the atomic decomposition of $\h_1^c(\R,\M)$ can be deduced from the $\h_1$-$\bmo$ duality. The following definition of atoms is given in \cite{XX17}.

\begin{defn}\label{def: atom h1}
Let $Q$ be a cube in $\R$ with $| Q|\leq 1$. If $| Q|=1$, an $\h _1^c$-atom associated with $Q$ is a function $a\in L_{1}  (\mathcal{M};L_{2}^{c}  (\mathbb{R}^{d} ) )$ such that
\begin{itemize}%
\item $\supp a\subset Q$;
\item $\tau\big(\int_{Q}  |a  (s ) |^{2}ds\big)^{\frac{1}{2}}\leq  |Q |^{-\frac{1}{2}}$.
\end{itemize}
If $|Q|<1$, we assume additionally: 
\begin{itemize}
\item $\int_{Q}a(s)ds=0.$
\end{itemize}
\end{defn}

Let $\h_{1,{\rm at}}^{c}  (\mathbb{R}^{d},\mathcal{M} )$ be the
space of all $f$ admitting a representation of the form
\[
f=\sum_{j=1}^\infty \lambda_ja_j,
\]
where the $a_{j}$'s are $\h_1^{c}$-atoms and $\lambda_{j}\in\mathbb{C}$ such that $\sum_{j=1}^\infty  |\lambda_{j} |<\infty$. The above
series  converges in the sense of distribution.
We equip $\h_{1,{\rm at}}^{c} (\mathbb{R}^{d},\mathcal{M} )$ with
the following norm:
\[
  \|  f \|  _{\h_{1,{\rm at}}^{c}}=\inf  \{ \sum_{j=1}^\infty  |\lambda_{j} |: f=\sum_{j=1}^\infty\lambda_{j}a_{j};\,\mbox{\ensuremath{a_{j}}'s are \ensuremath{\h_1^{c}} -atoms, }\mbox{\ensuremath{\lambda}}_{j}\in\mathbb{C} \} .
\]
It is proved in \cite{XX17} that
\begin{equation}\label{thm:atomic h1}
\h_{1,{\rm at}}^{c} (\mathbb{R}^{d},\mathcal{M} )=\h_{1}^{c}  (\mathbb{R}^{d},\mathcal{M} )
\end{equation}
with equivalent norms. It is also evident in the proof of \eqref{thm:atomic h1} given in \cite{XX17} that, in Definition \ref{def: atom h1}, we can replace the support $Q$ of atoms  by any bounded multiple of $Q$.

\medskip

Let us introduce the smooth atomic decomposition of $F_1^{\alpha ,c }(\R,\M)$, which will be a key ingredient to obtain the boundedness of pseudo-differential operators on $F_1^{\alpha ,c }(\R,\M)$. This decomposition is an extension as well as an improvement of the atomic
decomposition of $\h^c_1(\R, \M)$ in \eqref{thm:atomic h1}. Compared to \eqref{thm:atomic h1},  smoothness of the atoms is improved and subatoms enter in the game.

For every $l=  (l_{1},\cdots,l_{d} )\in\mathbb{Z}^{d}$, $\mu \in \mathbb{N}_0$, we define  $Q_{\mu, l}$ in $\R$ to be the cubes centered at $2^{-\mu}l$, and with side length $2^{-\mu}$. For instance, $Q_{0,0} = [-\frac 1 2 , \frac 1 2 )^d$ is the unit cube centered at the origin. Let $\mathbb{D}_d$ be the collection of all the cubes $Q_{\mu, l}$ defined above. We write $(\mu , l)\leq (\mu ', l')$ if
 \[
 \mu \geq \mu ' \quad \mbox{and}\quad Q_{\mu, l}\subset 2Q_{\mu' ,l'}.
 \]
 For $a \in \mathbb{R}$, let $a_+=\max\{a,0\}$ and $[a]$ the largest integer less than or equal to $a$. Denote $|\gamma|_1 =\gamma_1+\cdots+\gamma_d$ and $D^\gamma = \partial_1^{\gamma_1}\cdots \partial_d^{\gamma_d}$ for $\gamma\in \mathbb{N}_0^d$, and $s^{\beta}=s_1^{\beta_1}\cdots s_d^{\beta_d}$ for $s\in \R$, $\beta \in \mathbb{N}_0^d$. Recall that $J^\alpha$ is the Bessel potential of order $\alpha$.

 \begin{defn}\label{def:smooth T atom}
Let $\alpha\in\mathbb{R}$, and let $K$ and $L$ be two integers such that
\[
K\geq ([\alpha]+1)_+ \quad \mbox{and} \quad L\geq \max{\{[-\alpha],-1\}}.
\]
\begin{enumerate}[\rm(1)]
\item A function $b\in L_1(\mathcal{M};L_2^c(\R))$  is called an $(\alpha, 1)$-atom if
\begin{itemize}
\item $\supp b\subset 2 Q_{0,k}$;
\item  $\tau  (\int_{\R}  |D^{\gamma}b(s) |^{2}ds )^{\frac{1}{2}}\leq1$, $\,\,\forall\,\gamma \in  \mathbb{N}_0^d\,, \,\, |\gamma |_1\leq K$.
\end{itemize}

\item Let $Q=Q_{\mu, l}\in \mathbb{D}_d$, a function $a\in L_1(\mathcal{M};L_2^c(\R))$  is called an $(\alpha, Q)$-subatom if
\begin{itemize}
\item $\supp a\subset 2Q$;
\item $\tau  (\int_{\R}  |D^{\gamma}a(s) |^{2}ds )^{\frac{1}{2}}\leq  |Q |^{\frac{\alpha}{d}-\frac{  |\gamma |_1}{d}}$,
$\,\,\forall\,\gamma \in  \mathbb{N}_0^d\,, \,\, |\gamma |_1\leq K$;
\item $\int_{\mathbb{R}^{d}}s^{\beta}a  (s )ds=0$, $\,\,\forall\,\beta \in  \mathbb{N}_0^d\,, \,\, |\beta |_1\leq L$.
\end{itemize}

\item A function $g\in L_1(\mathcal{M};L_2^c(\R))$ is called an $(\alpha, Q_{k,m})$-atom if
\begin{equation}\label{eq: Q atom-sc}
\tau  (\int_{\R}  | J^\alpha g(s) |^{2}ds )^{\frac{1}{2}}\leq  |Q_{k,m} |^{-\frac 1 2 }\quad \text{and} \quad g=\sum_{  (\mu,l )\leq (k,m )}d_{\mu, l}a_{\mu, l},
\end{equation}
for some $k\in\mathbb{N}_0$ and $m\in\mathbb{Z}^{d}$, where the $a_{\mu ,l}$'s
are $(\alpha,Q_{\mu ,l})$-subatoms and the $d_{\mu ,l}$'s are complex numbers
such that $$  (\sum_{  (\mu,l )\leq  (k,m )}  |d_{\mu ,l} |^{2} )^{\frac{1}{2}}\leq  |Q_{k,m} |^{-\frac{1}{2}}.$$
\end{enumerate}
\end{defn}

We have obtained in \cite{XX18} the following smooth atomic decomposition:

\begin{thm}
\label{thm: atomic decop T_L}
Let $\alpha\in\mathbb{R}$ and $K$, $L$ be two integers fixed as in Definition \ref{def:smooth T atom}. Then any
$f\in F_{1}^{\alpha, c}  (\mathbb{R}^{d},\mathcal{M} )$ can be represented as
\begin{equation}\label{eq:decomp T-L}
f=\sum_{j=1}^{\infty}\big(\mu_j b_j+\lambda_{j}g_{j}\big),
\end{equation}
where the $b_j$'s are $(\alpha,1)$-atoms, the $g_{j}$'s are $(\alpha,Q)$-atoms, and $\mu_j$, $\lambda_{j}$ are complex numbers
with
\begin{equation}\label{eq:mu+lambda}
\sum_{j=1}^{\infty}  (  |\mu_j |+  |\lambda_{j} | )<\infty.
\end{equation}
 Moreover, the infimum of \eqref{eq:mu+lambda} with respect to all admissible representations
yields an equivalent norm in $F_{1}^{\alpha,c}   (\mathbb{R}^{d},\mathcal{M} )$.
\end{thm}

It is worthwhile to point out that the above $K$ and $L$ can be arbitrarily large, depending on the resolution of the unit used in the proof of Theorem \ref{thm: atomic decop T_L} given in \cite{XX18}. In other words, the orders of the smoothness and moment cancellation of the atoms are at our disposal, so that we can require good enough conditions on the atoms. This will be a very important technique in the proofs of our main results.

\section{Pseudo-differential operators: definitions and basic properties}\label{section-pdo-def}

We introduce the definitions and some basic properties of pseudo-differential operators in this section. The symbols of pseudo-differential operators considered here are $B(X)$-valued, where $X$ is a Banach space and $B(X)$ denotes the space of all bounded linear operators on $X$. However, in the later sections, we will only consider those symbols with values in $\M$.

 Let $n\in\mathbb{R}$ and $0\leq \delta, \rho \leq1$.
Then $S_{\rho,\delta}^{n}$ denotes the collection of all  infinitely differentiable
functions $\sigma$ defined on $\R\times \R$  and with values in $B(X)$, such that for each pair of multi-indices of nonnegative integers
$\gamma$, $\beta$, the inequality
\begin{equation*} 
\Vert D_{s}^{\gamma}D_{\xi}^{\beta}\sigma(s,\xi)\Vert _{B(X)}\leq C_{\gamma,\beta}(1+\vert\xi\vert)^{n+\delta |\gamma |_1-\rho |\beta |_1}
\end{equation*}
holds for some constant $C_{\gamma, \beta}$ depending on $\gamma,\beta$ and $\sigma$. Here again $\gamma=  (\gamma_{1},\cdots,\gamma_{d} )\in\mathbb{N}_{0}^{d}$,  $|\gamma|_1=\gamma_1+\cdots +\gamma_d $ and $D_s^{\gamma}=\frac{\partial^{\gamma_{1}}}{\partial s_{1}^{\gamma_{1}}}\cdots\frac{\partial^{\gamma_{d}}}{\partial s_{d}^{\gamma_{d}}}$.

\begin{defn}

Let $\sigma\in S_{\rho,\delta}^{n}$. For any function
$f\in \mathcal{S}(\R; X)$,
 the pseudo-differential operator $T_\sigma$ is a mapping $f\mapsto T_\sigma f$
given by
\begin{equation}\label{eq: def of pd}
T_\sigma f(s)=\int_{\mathbb{R}^{d}}\sigma(s,\xi)\widehat{f}(\xi)e^{2\pi {\rm i}s\cdot \xi}d\xi.
\end{equation}
We call $\sigma$ the symbol of $T_\sigma $. 
 \end{defn}

 \begin{prop}
Let $0\leq \delta, \rho \leq 1$ and $n\in \mathbb{R}$. For any $\sigma\in S_{\rho,\delta}^n$, $T_\sigma $ is continuous on $\mathcal{S}(\R; X)$. 
 \end{prop}
 \begin{proof}
 By integration by parts, for any $s\in \R$ and $\gamma \in \mathbb{N}_0^d$, we have
 \begin{equation*}
\begin{split}
\| (2\pi {\rm i}s)^{\gamma}T_{\sigma} f\|_{X} & =\Big\|
(2\pi {\rm i}s)^{\gamma}\int_{\R}  \sigma(s,\xi)\widehat{f}(\xi)e^{2\pi {\rm i} s\cdot \xi} d\xi \Big\|_{X}\\
& = \Big\| \int _{\R} \sigma(s,\xi)\widehat{f}(\xi)D_\xi^{\gamma}(e^{2\pi {\rm i} s \cdot \xi} )d\xi \Big\|_{X}\\
&= \Big\| \int_{\R} D_\xi^{\gamma}[\sigma(s,\xi)\widehat{f}(\xi) ]e^{2\pi {\rm i} s \cdot \xi} d\xi \Big\|_{X}<\infty.
\end{split}
\end{equation*} 
Thus, $T _{\sigma} f $ is rapidly decreasing. A similar argument works for the partial derivatives of $T _{\sigma}f$, then we easily check that $T _{\sigma} f$ maps $\mathcal{S}(\R; X)$ continuously to itself. 
\end{proof}

Another way to write \eqref{eq: def of pd} is as a double integral:
\begin{equation}\label{eq: repeated integral}
T_{\sigma}f (s)=\int_{\R}\int_{\R} \sigma(s,\xi)f(t)e^{2\pi {\rm i} (s-t)\cdot \xi}dt d\xi.
\end{equation}
However, the above $\xi$-integral does not necessarily converge absolutely, even for $f\in \mathcal{S}(\R; X)$.
To overcome this difficulty, we will  approximate $\sigma$ by symbols with compact support. To this end, let us fix a compactly supported infinitely differentiable function $\eta$ defined on $\R\times \R$ such that $\eta = 1$ near the origin. Set
\begin{equation}\label{eq: compact symbol}
\sigma_j(s,\xi)=\sigma(s,\xi)\eta(2^{-j}s,2^{-j}\xi)\quad\text{ with }j\in \mathbb{N}.
\end{equation}
 Note that $\sigma_j$ converges pointwise to $\sigma$ and $\sigma_j\in S_{\rho,\delta}^{n}$ uniformly in $j$. Thus, for any $f\in \mathcal{S}(\R; X)$,  $T_{\sigma_j}f$ converges to $T_\sigma f$ in $\mathcal{S}(\R; X)$ as $j\rightarrow \infty$. Since the $\sigma_j$'s have compact supports, formula \eqref{eq: repeated integral} works for $T _{\sigma_j}f(s)$. 
 Then we can define the integral \eqref{eq: repeated integral} as follows:
\begin{equation}\label{eq: converge symbol}
T_\sigma  f(s)=\lim_{j\rightarrow \infty} \int_{\R}\int_{\R} {\sigma_j}(s,\xi)f(t)e^{2\pi {\rm i} (s-t)\cdot \xi}dt d\xi. 
\end{equation}

\begin{prop}\label{prop: cont of adjoint}
Let  $0\leq \delta<1, 0\leq \rho \leq 1$ and $n\in \mathbb{R}$. For any $\sigma\in S_{\rho,\delta}^n$, the adjoint of $T_\sigma^c$ is continuous on $ \mathcal{S}(\R; X^*)$. 
\end{prop}
\begin{proof}
For any $f\in \mathcal{S}(\R; X)$ and $g\in \mathcal{S}(\R; X^*)$, by the duality relation 
$$
\langle T_\sigma  f,g\rangle =\langle f,(T_\sigma )^* g\rangle, 
$$ 
we check that 
\begin{equation}\label{eq: adjoint operator}
(T_{\sigma})^* g(s)=\lim_{j\rightarrow \infty} \int_{\R}\int_{\R} \sigma_j^*(t,\xi)g(t)e^{2\pi {\rm i} (s-t)\cdot \xi}dt d\xi.
\end{equation}
By integration by parts,  it is clear that $(T _{\sigma})^*$ is continuous on $ \mathcal{S}(\R; X^*)$.
\end{proof}

Since $\mathcal{S}' (\R; X^{**})=(\mathcal{S}(\R; X^{*}))^*$ (see \cite[Section~51]{Trves} for more details of this duality), in the usual way, we extend $T _{\sigma}$  to an operator on $\mathcal{S}' (\R; X^{**})$.

\begin{defn}
Let $f\in \mathcal{S}' (\R; X^{**})$. We define $T_{\sigma}f$ by the formula 
$$
\langle T_\sigma  f,g\rangle=\langle f,(T_\sigma )^*g\rangle ,  \quad \forall\, g\in  \mathcal{S}(\R; X^*).
$$
\end{defn}

By Proposition \ref{prop: cont of adjoint}, $(T_\sigma)^* g \in \mathcal{S}(\R; X^{*})$ whenever $g\in  \mathcal{S}(\R; X^*)$. So the bracket on the right hand side of the above definition is well defined. Therefore, $T_\sigma f$ is well defined, and takes value in $\mathcal{S}' (\R; X^{**})$ as well.

\begin{prop}
Let  $0\leq \delta<1, 0\leq  \rho \leq 1$  and $n\in \mathbb{R}$.  For any $\sigma\in S_{\rho,\delta}^n$,  $T_\sigma$ is continuous on $ \mathcal{S}' (\R; X^{**})$. 
\end{prop}

\begin{proof}
For any $f\in \mathcal{S}' (\R; X^{**})$, we take a sequence $(f_j)$  such that $f_j\rightarrow f$ in $\mathcal{S}' (\R; X^{**})$. Then we have
$$
\langle T_\sigma  f_j ,g\rangle=\langle f_j, (T_\sigma )^* g\rangle \longrightarrow \langle f,(T_\sigma )^* g \rangle =\langle T_\sigma f , g\rangle \quad \forall \, g\in  \mathcal{S}(\R; X^*).
$$
Thus, $T_\sigma  f_j $ converges to $T_\sigma   f $ in $ \mathcal{S}' (\R; X^{**})$. So $T_\sigma $ is continuous on $ \mathcal{S}' (\R; X^{**})$. 
\end{proof}

The pseudo-differential operator defined above has a parallel description
in terms of a distribution kernel:
$$
T_\sigma f(s)=\int_{\R}K(s,s-t)f(t)dt,
$$
where $K $ is the inverse Fourier transform of $\sigma $
with respect to the variable $\xi$, i.e.  
\begin{equation}\label{kernel-symbol}
K(s,t)=\int_{\R} \sigma (s,\xi)e^{2\pi {\rm i}t \cdot \xi} d\xi.
\end{equation}

In the sequel, we will focus on the symbols in the class $S_{1,\delta}^n$ with $0\leq\delta\leq 1$ and $n\in \mathbb{R}$. Similarly to the classical case (see \cite{CM1978}, \cite{Hormander1967}, \cite{Stein1993}
and \cite{Torres1991}), we prove that for any operator-valued symbol $\sigma\in S_{1,\delta}^n$, the corresponding kernel $K$ satisfies the following estimates:

\begin{lem}\label{Lem T K}
Let  $\sigma\in S_{1,\delta}^{n}$ and $0\leq\delta\leq1$. Then the kernel $K(s,t)$ in \eqref{kernel-symbol} satisfies
\begin{equation}
\|  D_{s}^{\gamma}D_{t}^{\beta}K(s,t)\| _{B(X)}  \leq  C_{\gamma,\beta}|  t|  ^{-  |\gamma |_1 -  |\beta |_1-d-n},\quad \forall\, t\in\mathbb{R}^{d}\setminus  \{ 0 \},\label{eq:kernel}
\end{equation}
\begin{equation}
\|  D_{s}^{\gamma}D_{t}^{\beta}K(s,t)\| _{B(X)} \leq  C_{\gamma,\beta,N}|  t| ^{-N},\quad \forall\,  N>0\,\, \text{ if }   |t |>1.\label{eq:kernel'}
\end{equation}
\end{lem}

\begin{proof}
This lemma can be deduced easily from the corresponding scalar-valued results, which can be found in many classical works on pseudo-differential operators, for instance, \cite[Lemma 5.1.6]{Torres-1990}. Given $x\in X$ and $x^*\in X^*$ with norms equal to one, it is clear that $\langle x^*, \sigma(s,t)x \rangle$ is a scalar-valued symbol in $S_{1,\delta}^n$, with distribution kernel  $\langle x^*, K(s,t)x \rangle$. Thus, we have 
$$
\langle  x^*, D_{s}^{\gamma}D_{t}^{\beta}K(s,t)x\rangle = D_{s}^{\gamma}D_{t}^{\beta}[\langle x^*, K(s,t)x \rangle ] \leq  C_{\gamma,\beta}|  t|  ^{-  |\gamma |_1 -  |\beta |_1-d-n},\quad \forall\, t\in\mathbb{R}^{d}\setminus  \{ 0 \}
$$
and
$$
\langle  x^*, D_{s}^{\gamma}D_{t}^{\beta}K(s,t)x\rangle = D_{s}^{\gamma}D_{t}^{\beta}[\langle x^*, K(s,t)x \rangle ] \leq  C_{\gamma,\beta,N}|  t| ^{-N},\quad \forall\, N>0 \text{ if }   |t |>1.
$$
Then, taking the supremum over $x$ and $x^*$ in the above two inequalities, we get the desired assertion. 
\end{proof}

In the classical case, the proof the above lemma makes use of the decomposition of the symbol $\sigma$ into dyadic pieces. Let $(\widehat{\varphi}_k)_{k\geq 0}$ be the resolution of the unit satisfying \eqref{eq:resolution of unity}. Set 
\begin{equation}\label{eq: sigma_k}
\sigma_k(s,\xi)=\sigma(s,\xi)\widehat{\varphi}_k(\xi), \quad \forall\, (s,\xi)\in \R\times \R. 
\end{equation}
By a similar argument as in the above proof, we also have the following estimates of the corresponding kernels of these pieces $\sigma_k$'s. 
 
 \begin{lem}\label{lem: dyadic kernel estimate}
Let $\sigma\in S_{1,\delta}^{n}$ and $ \sigma_k$ be as in \eqref{eq: sigma_k} and $K_k(s,t)=\int_{\R} \sigma _k(s,\xi)e^{2\pi {\rm i}t\cdot \xi} d\xi$. Then 
$$
\|D_{s}^{\gamma}D_{t}^{\beta}K_k (s,t )\| _{B(X)}\lesssim  |t |^{-2M} 2^{k(|\beta |_1+|\gamma |_1+d-2M+n)}, \quad \forall\, M\in \mathbb{N}_0.
$$
 \end{lem}

Now we study the composition of pseudo-differential operators.The following proposition gives a rule of the composition of two pseudo-differential operators. In particular,  it shows that the symbol class $S_{1,\delta}^0$ is closed under product. Different from the proof of Lemma \ref{Lem T K}, we can not  reduce that of the following proposition to the scalar-valued case. So we need to perform an argument which is similar to the classical case. We refer the reader to \cite[Theorem 2.5.1]{R-T} or \cite[p. 237]{Stein1993}, where the case $\delta=0$ is dealt with in the classical setting. However, for the case $0<\delta<1$, the remainder of the Taylor expansion of $\sigma_1$ is much harder to handle, which requires a subtler expansion of $\sigma_1$.

\begin{prop}\label{prop: composition}
Let $0\leq \delta<1$ and $\sigma_1$, $\sigma_2$ be two symbols in $S_{1,\delta}^{n_1}$ and $S_{1,\delta}^{n_2}$ respectively. There exists a symbol $\sigma_3$ in $S_{1,\delta}^{n_1+n_2}$ such that 
$$
T_{\sigma_3} =T_{\sigma_1}   T_{\sigma_2} .
$$
Moreover, 
\begin{equation}\label{eq: asymptotic formula}
\sigma_3-\sum_{|\gamma|_1<N_0} \frac{(2\pi {\rm i})^{-|\gamma|_1}}{\gamma !}D_\xi^\gamma \sigma_1 D_s^\gamma \sigma_2 \in S_{1,\delta}^{n_1+n_2-(1-\delta)N_0}, \quad \forall\, N_0\geq 0.
\end{equation}
\end{prop}

\begin{proof}
Firstly, we assume that $\sigma_1$ and $\sigma_2$ have compact supports, so we can use \eqref{eq: repeated integral} as an alternate definition of $T_{\sigma_1} $ and $T_{\sigma_2} $. 
In this way,  $T_{\sigma_1}   T_{\sigma_2} $ can be written as follows:
$$
T _{\sigma_1}(T  _{\sigma_2}f)(s)=\int_{\R}\int_{\R}\sigma_3(s,\xi)f(r)e^{2\pi {\rm i} (s-r)\cdot \xi}dr d\xi,
$$
where 
\begin{equation}\label{eq: sgma3}
\begin{split}
\sigma_3(s,\xi) &  =\int_{\R}\int_{\R}\sigma_1(s,\eta)\sigma_2(t,\xi)e^{2\pi {\rm i}(s-t)\cdot (\eta-\xi)}dtd\eta\\
&  = \int_{\R}\sigma_1(s,\xi+\eta)\widehat{\sigma}_2(\eta,\xi) e^{-2\pi {\rm i}s\cdot\eta} d\eta
\end{split}
\end{equation}
with $\widehat{\sigma}_2$ the Fourier transform of $\sigma_2$ with respect to the first variable. 
We expand $\sigma_1(s,\xi+\eta)$ by the Taylor formula:
$$
\sigma_1(s,\xi+\eta)= \sum_{|\gamma|_1<N_0} \frac{1}{\gamma !}D_\xi ^\gamma \sigma_1(s,\xi) \eta^\gamma + \sum_{N_0\leq |\gamma|_1<N}  \frac{1}{\gamma !}D_\xi ^\gamma \sigma_1(s,\xi) \eta^\gamma +R_N(s,\xi, \eta), 
$$
with the remainder 
$$R_N(s,\xi, \eta)=\sum_{|\gamma|_1=N}\frac{1}{\gamma !}\int_{0}^1 D_\xi ^\gamma \sigma_1(s,\xi +\theta \eta)(1-\theta)^N \eta^\gamma d\theta . $$ 
Now we replace $\sigma_1(s,\xi+\eta)$ in \eqref{eq: sgma3} by the above Taylor polynomial and remainder. Notice that 
$$
\frac{1}{\gamma !}\int_{\R}D_\xi ^\gamma \sigma_1(s,\xi) \eta^\gamma  \widehat{\sigma}_2(\eta,\xi)e^{-2\pi {\rm i}s\cdot\eta} d\eta= \frac{(2\pi {\rm i})^{-|\gamma|_1}}{\gamma !}D_\xi ^\gamma \sigma_1(s,\xi) D_s^\gamma \sigma_2(s,\xi).
$$
Thus, 
\begin{equation}\label{sigma3-expand}
\begin{split}
\sigma_3(s,\xi) &= \big( \sum_{|\gamma|_1<N_0}+  \sum_{N_0\leq |\gamma|_1<N}  \big)\frac{(2\pi {\rm i})^{-|\gamma|_1}}{\gamma !}D_\xi ^\gamma \sigma_1(s,\xi) D_s^\gamma \sigma_2(s,\xi)\\
&\;\;\;\; +\int_{\R}R_N(s,\xi, \eta)\widehat{\sigma}_2(\eta,\xi) e^{-2\pi {\rm i}s\cdot\eta} d\eta.
\end{split}
\end{equation}
For every $\gamma$, the term $D_\xi ^\gamma \sigma_1(s,\xi) D_s^\gamma \sigma_2(s,\xi)$ is a symbol in $S_{1,\delta}^{n_1+n_2-(1-\delta)|\gamma |_1}$. Indeed, it is clear that 
$$
\| D_\xi ^\gamma \sigma_1(s,\xi) D_s^\gamma \sigma_2(s,\xi)\|_{B(X)}\lesssim (1+|\xi |)^{n_1-|\gamma |_1} (1+|\xi |)^{n_2+\delta |\gamma |_1}=(1+|\xi |)^{n_1+n_2-(1-\delta)|\gamma |_1}.
$$
Moreover, for any $\beta_1, \beta_2 \in \mathbb{N}_0^d$, we have $D_s^{\beta_1} \sigma_1 \in S_{1,\delta}^{n_1+\delta |\beta_1|_1} , D_s^{\beta_2} \sigma_2 \in S_{1,\delta}^{n_2+\delta |\beta_1|_1}$ and $D_\xi^{\beta_1} \sigma_1\in S_{1,\delta}^{n_1- |\beta_2|_1} , D_\xi^{\beta_2} \sigma_2 \in S_{1,\delta}^{n_2- |\beta_2|_1}$. Thus, we get
\begin{equation*}
\begin{split}
\big\| D_s^\beta [D_\xi ^\gamma \sigma_1(s,\xi) D_s^\gamma \sigma_2(s,\xi)] \big\|_{B(X)} & \lesssim \sum_{\beta_1+\beta_2=\beta} \big\|D_s^{\beta_1}D_\xi^\gamma \sigma_1(s,\xi)D_s^{\beta_2} D_s^\gamma \sigma_2(s,\xi) \big\|_{B(X)}\\
&  \lesssim (1+|\xi |)^{n_1+n_2-(1-\delta)|\gamma |_1+\delta |\beta|_1}, 
\end{split}
\end{equation*}
and 
\begin{equation*}
\begin{split}
\big\| D_\xi^\beta [D_\xi ^\gamma \sigma_1(s,\xi) D_s^\gamma \sigma_2(s,\xi)] \big\|_{B(X)} & \lesssim \sum_{\beta_1+\beta_2=\beta} \big\| D_\xi^{\gamma+\beta_1} \sigma_1(s,\xi) D_s^\gamma D_\xi^{\beta_2} \sigma_2(s,\xi) \big\|_{B(X)}\\
&  \lesssim (1+|\xi |)^{n_1+n_2-(1-\delta)|\gamma |_1-|\beta|_1}. 
\end{split}
\end{equation*}
By the above estimates, we see that when $N_0\leq |\gamma|_1<N$, $D_\xi ^\gamma \sigma_1(s,\xi) D_s^\gamma \sigma_2(s,\xi)\in S_{1,\delta}^{n_1+n_2-(1-\delta)N_0}$.

Now we have to treat the last term in \eqref{sigma3-expand}. For the remainder $R_N(s,\xi, \eta)$, we easily check that for any $|\gamma |_1=N$ and $0\leq \theta \leq 1$,  
\begin{equation}\label{eq: estimate sigma 1}
\| D_\xi^\gamma \sigma_1(s,\xi+\theta \eta) \|_{B(X)}\leq C_N (1+|\xi|)^{n_1-N}, \quad \text{if }|\xi|\geq 2|\eta|,
\end{equation}
and 
\begin{equation}\label{eq: estimate sigma 2}
\| D_\xi^\gamma \sigma_1(s,\xi+\theta \eta) \|_{B(X)}\leq C_N', \quad \forall\, \eta, \xi \in \R.
\end{equation}
For $\widehat{\sigma}_2$, by integration by parts, we see that for any $\beta\in \mathbb{N}_0^d$ such that $|\beta|_1=\widetilde{N}$, we have
\begin{equation*}
\begin{split}
(-2\pi {\rm i}\eta)^{\beta}\widehat{\sigma}_2(\eta,\xi) & =\int_{\R}(-2\pi {\rm i} \eta)^{\beta} e^{-2\pi {\rm i}t\cdot\eta}\sigma_2(t,\xi) dt\\
& =\int_{\R} D_t^\beta (e^{-2\pi {\rm i}t\cdot\eta})\sigma_2(t,\xi)dt\\
&= (-1)^\beta\int_{\R}  e^{-2\pi {\rm i}t\cdot\eta}D_t^\beta \sigma_2(t,\xi)dt. 
\end{split}
\end{equation*}
Denote the compact $t$-support of $\sigma_2(t,\xi)$ by $\Omega$. Then the above calculation immediately implies that 
\begin{equation}\label{eq: estimate hat sigma}
\| \widehat{\sigma}_2(\eta,\xi)  \|_{B(X)}\lesssim |\Omega |(1+|\eta |)^{-\widetilde{N}}(1+|\xi |)^{n_2+\delta \widetilde{N}}.
\end{equation}
We keep the constant $|\Omega|$ in this inequality for the moment, and will see in the next step that our final result does not depend on the volume of this support. 
Take $\widetilde{N}$ large enough so that 
$$\widetilde{N}>\max\big\{\frac{d}{1-\widetilde{\delta}},
\frac{(1-\delta)N_0}{\widetilde{\delta}-\delta}, \frac{d-n_1+(1-\delta)N_0}{1-2\delta}\big\},$$ 
and take $N=\widetilde{\delta}\widetilde{N}$ with $0\leq \delta<\widetilde{\delta}<1$.
Continuing the estimate of the last term in \eqref{sigma3-expand}, inequalities \eqref{eq: estimate sigma 1} and \eqref{eq: estimate hat sigma} give 
\begin{equation*}
\begin{split}
& \Big\| \int_{|\eta |\leq \frac{|\xi |}{2}}\int_{0}^1 D_\xi ^\gamma \sigma_1(s,\xi +\theta \eta)(1-\theta)^N \eta^\gamma \widehat{\sigma}_2(\eta,\xi) e^{-2\pi {\rm i}s\cdot\eta} d\theta d\eta  \Big\|_{B(X)} \\
& \lesssim \int_{\R} |\eta |^N(1+|\eta|)^{-\widetilde{N}}d\eta \cdot (1+|\xi|)^{n_1+n_2-N+\delta \widetilde{N}}\\
& \leq \int_{\R} (1+|\eta|)^{(\widetilde{\delta}-1)\widetilde{N}}d\eta\cdot (1+|\xi|)^{n_1+n_2+( \delta-\widetilde{\delta}) \widetilde{N}}\\
& \lesssim (1+|\xi|)^{n_1+n_2+( \delta-\widetilde{\delta}) \widetilde{N}}.
\end{split}
\end{equation*}
Moreover, since $\widetilde{N}\geq 
\frac{(1-\delta)N_0}{\widetilde{\delta}-\delta}$, we have 
$$
 (1+|\xi|)^{n_1+n_2+( \delta-\widetilde{\delta}) \widetilde{N}}\leq  (1+|\xi|)^{n_1+n_2-(1-\delta)N_0}. 
$$
According to \eqref{eq: estimate sigma 2} and \eqref{eq: estimate hat sigma}, we get 
\begin{equation*}
\begin{split}
& \Big\| \int_{|\eta |> \frac{|\xi |}{2}}\int_{0}^1 D_\xi ^\gamma \sigma_1(s,\xi +\theta \eta)(1-\theta)^N \eta^\gamma \widehat{\sigma}_2(\eta,\xi)e^{-2\pi {\rm i}s\cdot\eta}  d\theta d\eta \Big\|_{B(X)} \\
& \lesssim \int_{|\eta |> \frac{|\xi |}{2}} |\eta |^N(1+|\eta|)^{-\widetilde{N}}d\eta \cdot (1+|\xi|)^{n_2+\delta \widetilde{N}}\\
& \lesssim  (1+|\xi|)^{n_2+N+d-(1-\delta) \widetilde{N}}\leq (1+|\xi|)^{n_1+n_2-(1-\delta)N_0}.
\end{split}
\end{equation*}
Therefore, $R_N(s,\xi, \eta) \in S_{1,\delta}^{n_1+n_2-(1-\delta)N_0}$. Combining the estimates above, we see that, if we set $R_{N_0}(s,\xi, \eta)=\sum_{N_0\leq |\gamma|_1<N}  \frac{1}{\gamma !}D_\xi ^\gamma \sigma_1(s,\xi) \eta^\gamma +R_N(s,\xi, \eta)$, then $R_{N_0}(s,\xi, \eta) \in S_{1,\delta}^{n_1+n_2-(1-\delta)N_0}$. This proves the assertion \eqref{eq: asymptotic formula} when $\sigma_2$ has compact support with respect to the first variable.

Noticing that the above proof depends on the constant $|\Omega|$ in \eqref{eq: estimate hat sigma}, we now make use of the resolution of the unit in \eqref{eq: unit resolution} to deal with general symbol $\sigma_2$ with arbitrary $s$-support. For each $k\in \mathbb{Z}^d$, denote $\sigma_{2,k}(s,\xi) = \mathcal{X}_k(s)\sigma_2(s,\xi)$ and 
$$\sigma_{3,k}(s,\xi)  =\int_{\R}\int_{\R}\sigma_1(s,\eta)\sigma_{2,k}(t,\xi)e^{2\pi {\rm i}(s-t)\cdot (\eta-\xi)}dtd\eta\,.$$
It has already been established that
\begin{equation}\label{eq: asymptotic formula-k}
\sigma_{3,k}-\sum_{|\gamma|_1<N_0} \frac{(2\pi {\rm i})^{-|\gamma|_1}}{\gamma !}D_\xi^\gamma \sigma_1 D_s^\gamma \sigma_{2,k} \in S_{1,\delta}^{n_1+n_2-(1-\delta)N_0}, \quad \forall \,N_0>0, k\in \mathbb{Z}^d,
\end{equation}
with relevant constants uniform in $k$. 
Observe that if two symbols $b_1, b_2$ in some $S^n_{1,\delta}$ have disjoint $s$-supports, with
\begin{equation*} 
\Vert D_{s}^{\gamma}D_{\xi}^{\beta}b_i(s,\xi)\Vert _{B(X)}\leq C_{i,\gamma,\beta}(1+\vert\xi\vert)^{n+\delta |\gamma |_1- |\beta |_1}, \quad  i=1,2,
\end{equation*}
 then $b_1 +b_2 \in S^n_{1,\delta}$ with 
  \begin{equation*} 
\Vert D_{s}^{\gamma}D_{\xi}^{\beta}\big(b_1(s,\xi)+b_2(s,\xi)\big)\Vert _{B(X)}\leq \max\{C_{1,\gamma,\beta}, C_{2,\gamma,\beta}\}(1+\vert\xi\vert)^{n+\delta |\gamma |_1- |\beta |_1}.
\end{equation*}
 For our use, we construct a partition of $\mathbb{Z}^d$ with subsets $U_1, U_2,\cdots,U_{2^d}$ such that for any $k_1, k_2$ in each $U_j$, the supports $\supp \mathcal{X}_{k_1}$ and $\supp \mathcal{X}_{k_2}$ are disjoint. More precisely, let $\pi$: $\mathbb{Z}\longrightarrow \mathbb{Z}/ 2\mathbb{Z}$ be the canonical projection sending even integer to 0 and odd integer to 1. Let $\pi^d$: $\mathbb{Z}^d \longrightarrow (\mathbb{Z}/ 2\mathbb{Z})^d$ be the $d$-fold product of $\pi$. Then $(U_j)_{j\in (\mathbb{Z}/ 2\mathbb{Z})^d}=\big(( \pi^d)^{-1}(j)\big)_{j\in (\mathbb{Z}/ 2\mathbb{Z})^d}$ gives the desired partition of $\mathbb{Z}^d$. 
Summing over \eqref{eq: asymptotic formula-k} in each $U_j$, we get a symbol still in $S_{1,\delta}^{n_1+n_2-(1-\delta)N_0}$, that is,
\begin{equation*}
\sum_{k\in U_j}\sigma_{3,k}-\sum_{k\in U_j}\sum_{|\gamma|_1<N_0} \frac{(2\pi {\rm i})^{-|\gamma|_1}}{\gamma !}D_\xi^\gamma \sigma_1 D_s^\gamma \sigma_{2,k} \in S_{1,\delta}^{n_1+n_2-(1-\delta)N_0} .
\end{equation*}  
Taking the finite sum over $\{U_j\}_{1\leq j \leq 2^d}$, we get the asymptotic formula \eqref{eq: asymptotic formula} in this case.

Finally, let us get rid of the additional assumption that $\sigma_1$ and $\sigma_2$ have compact supports.  We define $\sigma_{3}^j$ as follows:
$$
T_{\sigma_{3}^j} =T_{\sigma_{1}^j}  T_{\sigma_{2}^j}  .
$$
where $ \sigma_{1}^j(s,\xi)=\sigma_1(s,\xi)\eta(2^{-j}s,2^{-j}\xi)$ and $ \sigma_{2}^j(s,\xi)=\sigma_2(s,\xi)\eta(2^{-j}s,2^{-j}\xi)$ with $\eta$ given in \eqref{eq: compact symbol}. 
Notice that the $ \sigma_{1}^j$'s and the $ \sigma_{2}^j$'s are in the class $S^{n_1}_{1,\delta}$ and $S^{n_2}_{1,\delta}$ respectively with symbolic constants uniform in $j$. Therefore, the above arguments ensure that $\sigma_{3}^j$ belongs to $S^{n_1+n_2}_{1,\delta}$ and satisfies \eqref{eq: asymptotic formula} uniformly in $j$. Passing to the limit, we get that $\sigma_3\in S^{n_1+n_2}_{1,\delta} $ and satisfies \eqref{eq: asymptotic formula}. Furthermore, by \eqref{eq: converge symbol}, we get 
$$
T_{\sigma_3} =T_{\sigma_1}  T_{\sigma_2} .
$$
The proof is complete.
\end{proof}

We end this section with the asymptotic formula for the adjoint of a pseudo-differential operator with symbol in the class $S_{1,\delta}^n$ when $0\leq \delta <1$.

\begin{prop}\label{prop: adjoint}

Let $0\leq \delta <1$, $n\in \mathbb{R}$ and $\sigma$ be a symbol in $S_{1,\delta}^n$. There exists a symbol $\widetilde{\sigma}\in S_{1,\delta}^n$ such that $T_{\widetilde{\sigma}} =(T_\sigma )^*$. Moreover, 
$$
\widetilde{\sigma} -\sum_{|\gamma|_1<N_0} \frac{(2\pi {\rm i})^{-|\gamma|_1}}{\gamma !}D_\xi^\gamma D_s^\gamma \sigma^* \in S_{1,\delta}^{n-(1-\delta)N_0}, \quad \forall\, N_0\geq 0.
$$

\end{prop}

  \begin{proof}
The proof is similar to that of Proposition  \ref{prop: composition}. By \eqref{eq: adjoint operator}, we get the formal expression of $\widetilde{\sigma}$ that 
\begin{equation*}
\begin{split}
\widetilde{\sigma}(s,\xi) & = \int_{\R}\int_{\R} \sigma^*(t,\eta)e^{2\pi {\rm i}(s-t)\cdot (\eta -\xi)} dtd\eta \\
& = \int_{\R}\widehat{{\sigma}}^*(\eta, \xi+\eta)e^{2\pi {\rm i}s\cdot \eta} d\eta, 
\end{split}
\end{equation*}
where $\widehat{{\sigma}}^*$ is the Fourier transform of $\sigma^*$ with respect to the first variable. By the same argument used in the proof of the previous proposition, we may focus on the symbol with compact $t$-support. Taking the Taylor expression of $\widehat{{\sigma}}^*(\eta, \xi+\eta)$, we get
$$
\widehat{{\sigma}}^*(\eta,\xi+\eta)= \sum_{|\gamma|_1<N_0} \frac{1}{\gamma !}D_\xi ^\gamma \widehat{{\sigma}}^*(\eta,\xi) \eta^\gamma + \sum_{N_0\leq |\gamma|_1<N}  \frac{1}{\gamma !}D_\xi ^\gamma \widehat{{\sigma}}^*(\eta,\xi) \eta^\gamma +R_N(\xi, \eta).
$$
As before, we can show that   
$$
\frac{1}{\gamma !}\int_{\R}D_\xi ^\gamma \widehat{{\sigma}}^*(\eta,\xi) \eta^\gamma e^{2\pi {\rm i}s\cdot\eta} d\eta= \frac{(2\pi {\rm i})^{-|\gamma|_1}}{\gamma !}D_\xi ^\gamma  D_s^\gamma \widehat{{\sigma}}^*(s,\xi)\in S_{1,\delta}^{n-(1-\delta)|\gamma|_1}.
$$
On the other hand, we can also show that
$$
\big\| \int_{\R} R_N(\xi,\eta) e^{2\pi {\rm i}s\cdot\eta} d\eta \big\|_{B(X)} \lesssim (1+|\xi |)^{n-(1-\delta)N_0}
$$
by splitting the integral over $\eta$ into two parts. Moreover, repeating the above procedure to its derivatives, we have $ \int_{\R} R_N(\xi,\eta) e^{2\pi {\rm i}s\cdot\eta} d\eta \in S_{1,\delta}^{n-(1-\delta)N_0}$. Thus, the proposition is proved.
\end{proof}

\section{The action of pseudo-differential operators on (sub)atoms}\label{section-pdo-atom}

In order to study the boundedness of pseudo-differential operators on the Triebel-Lizorkin spaces, we will use the atomic decomposition stated in Theorem \ref{thm: atomic decop T_L}. In other words, we will focus on the images of the atoms under the action of  pseudo-differential operators instead of those of general functions in the Triebel-Lizorkin spaces.

In the sequel, we will only consider pseudo-differential operators whose symbols take values in some von Neumann algebra $\M$. If we take $X=L_1(\M)+\M$, then $\M$ admits an isometric embedding into $B(X)$ by left multiplication.
In this way, these $\M$-valued symbols can be seen as a special case of the $B(X)$-valued symbols defined in the previous section. On the other hand, if we embed $\M$ into $B(X)$ by right multiplication, we get another kind of $\M$-valued symbol actions. Accordingly, we define
$$
T_\sigma^c f(s)=\int_{\mathbb{R}^{d}}\sigma(s,\xi)\widehat{f}(\xi)e^{2\pi {\rm i}s\cdot \xi}d\xi $$
and 
$$T^r_\sigma f(s)=\int_{\mathbb{R}^{d}}\widehat{f}(\xi)\sigma(s,\xi)e^{2\pi {\rm i}s\cdot \xi}d\xi.$$
All the conclusions proved in the previous section still hold for both $T_\sigma^c $ and $T_\sigma^r $ in parallel. In the following sections, we mainly focus on the operators $T_\sigma^c $.

The first  lemma in this section  concerns the image of an $(\alpha,Q_{\mu, l})$-subatom under the action of pseudo-differential operators.

 \begin{lem}
\label{lem:moledule}Let $\alpha\in \mathbb{R}$, $\sigma\in S_{1,\delta}^{0}$
and $T^c_{\sigma}$ be the corresponding pseudo-differential operator. In addition, we assume that  $K>\frac{d}{2}$. Then for any  $(\alpha,Q_{\mu, l})$-subatom $a_{\mu, l}$, we have
\begin{equation}\label{eq:moledule}
\tau  \big(\int _{\R}  (1+2^{\mu}  |s-2^{-\mu}l | )^{d+M}  |D^{\gamma}T^c_{\sigma}a_{\mu, l}(s) |^{2}ds \big)^{\frac{1}{2}}\lesssim  |Q_{\mu, l} |^{\frac{\alpha}{d}-\frac{  |\gamma |_1}{d}},\quad |\gamma|_1< K-\frac{d}{2},
\end{equation}
where $M\in \mathbb{R}$ such that $M<2L+2$ and the relevant constant depends on $M$, $K$, $L$, $\gamma$ and $d$.
\end{lem}
\begin{proof}
We split the integral on the left hand side of \eqref{eq:moledule} into $\int_{4Q_{\mu, l}}$ and $\int_{(4Q_{\mu, l})^c}$. To estimate the term with $\int_{4Q_{\mu, l}}$, we begin with a technical modification of $a_{\mu, l}$. For every $a_{\mu, l}$, we define
$$
a=|Q_{\mu, l} |^{-\frac{\alpha}{d}+\frac{1}{2}}a_{\mu,l}(2^{-\mu}(\cdot+ l)).
$$
It is easy to see that $ a$ is an $(\alpha,Q_{0,0})$-subatom.
By translation, we may assume that $l=0$. Then by the Cauchy-Schwarz inequality \eqref{op-CS} and the Plancherel formula \eqref{op-Plancherel}, for any $s\in \R$, we have
\begin{equation*}
\begin{split}
 |T^c_{\sigma}a_{\mu, l}  (s ) |^{2}  &=   2^{-2\mu d}  |Q_{\mu, l} |^{2  (\frac{\alpha}{d}-\frac{1}{2} )}  \Big|\int_{\R} \sigma(s,\xi)\widehat{a}(2^{-\mu}\xi)e^{2\pi {\rm i}s\cdot \xi}d\xi \Big|^{2}\\
&   =   |Q_{\mu, l} |^{2  (\frac{\alpha}{d}-\frac{1}{2} )}  \Big|\int_{\R} \sigma(s,2^{\mu}\xi)\widehat{a}(\xi)e^{2\pi {\rm i}s\cdot 2^{\mu}\xi}d\xi \Big|^{2}\\
 &  \leq  |Q_{\mu, l} |^{2  (\frac{\alpha}{d}-\frac{1}{2} )}\int _{\R}  \| \sigma(s,2^{\mu}\xi) \|_\M ^2  (1+  |\xi |^2 )^{-K}d\xi \\
 & \;\;\;\; \int_{\R}   (1+  |\xi |^2 )^{K}  \| \sigma(s,2^{\mu}\xi) \|_\M ^{-2}\,\widehat{a}^*(\xi) |\sigma(s,2^{\mu}\xi)|^2  \widehat{a}(\xi)  d\xi\\
 &  \lesssim  |Q_{\mu, l} |^{2  (\frac{\alpha}{d}-\frac{1}{2} )} \int_{\R}   (1+  |\xi |^2 )^{-K}d\xi\cdot\int _{\R}  (1+  |\xi |^2 )^{K}  |\widehat{a}(\xi) |^{2}d\xi\\
&  \lesssim   |Q_{\mu, l} |^{2  (\frac{\alpha}{d}-\frac{1}{2} )}\int _{\R}  |J^{K}a(t) |^{2}dt,
\end{split}
\end{equation*}
where $J^K$ is the Bessel potential of order $K$.  Combining the second assumption on $a_{\mu,l}$ in Definition \ref{def:smooth T atom} and the above estimate, we obtain
\begin{equation*}
\begin{split}
\tau  \big(\int_{4Q_{\mu, l}}  |T^c_{\sigma}a_{\mu, l}  (s ) |^{2}  (1+2^{\mu}  |s | )^{d+M}ds \big)^{\frac{1}{2}}  &\lesssim\tau   \big(\int_{4Q_{\mu, l}}  |T^c_{\sigma}a_{\mu, l}  (s ) |^{2}ds  \big)^{\frac{1}{2}} \\
 &\lesssim |Q_{\mu, l} |^{\frac{\alpha}{d}}\tau  (\int _{\R}  |J^{K}a(t) |^{2}dt )^{\frac{1}{2}}\\
 &  \lesssim  |Q_{\mu, l} |^{\frac{\alpha}{d}}\sum_{|\gamma|_1\leq K}\tau   \big(\int _{\R}  |D^{\gamma}a(t) |^{2}dt  \big)^{\frac{1}{2}}\lesssim |Q_{\mu, l} |^{\frac{\alpha}{d}}.
\end{split}
\end{equation*}
If $s\in  (4Q_{\mu, l} )^{c}$, since $a_{\mu, l}$ has the
moment cancellations of order less than or equal to $L$, we can subtract a Taylor polynomial
of degree $L$ from the kernel associated to $T^c_{\sigma}$. Then, applying the estimate \eqref{eq:kernel} and the Cauchy-Schwarz inequality \eqref{op-CS}, we get
\begin{equation}\label{eq: s off origin}
\begin{split}
&  |T^c_{\sigma}a_{\mu, l}  (s ) |^{2}  =    \Big|\int_{\R}  K  (s,s-t )a_{\mu, l}  (t )dt \Big|^{2}\\
 & =    \Big|\int_{\R}   [K  (s,s-t )-K(s,s) ]a_{\mu, l}  (t )dt \Big|^{2}\\
 & =   \Big|\int _{\R}  [\sum_{|\beta |_1=L+1}\frac{L+1}{\beta !}t^\beta \int_0^1(1-\theta)^\beta D^\beta K(s,s-\theta t)d\theta  ]a_{\mu, l}  (t )dt \Big|^{2}\\
 & \lesssim \sum_{|\beta|_1=L+1} \int_{2Q_{\mu, l}}  \Big\|  \int_0^1(1-\theta)^\beta D^\beta K(s,s-\theta t)d\theta \Big\|_\M ^2  |t |^{2L+2}dt\\
 & \;\;\;\; \cdot\int _{\R}   \Big\|  \int_0^1(1-\theta)^\beta D^\beta K(s,s-\theta t)d\theta \Big\|_\M ^{-2}\, \Big| \int_0^1(1-\theta)^\beta D^\beta K(s,s-\theta t)d\theta \, a_{\mu, l}(t) \Big|^{2}dt\\
 & \leq  \sum_{|\beta|_1=L+1} \int_{2Q_{\mu, l}}\sup_{0\leq \theta\leq 1}  \|  D^{\beta}K(s,s-\theta t) \| _\M^2\,{  |t |^{2L+2}}dt\cdot\int _{\R}  |a_{\mu, l}(t) |^{2}dt\\
 & \lesssim    |s |^{-2d-2L-2}\int_{2Q_{\mu, l}}  |t |^{2L+2}dt\cdot\int _{\R}  |a_{\mu, l}(t) |^{2}dt\\
 & \lesssim  2^{-\mu(2L+2+d)}  |s |^{-2d-2L-2}\int _{\R}  |a_{\mu, l}(t) |^{2}dt.
 \end{split}
\end{equation}
This estimate implies
\begin{equation*}
\begin{split}
 &  \tau  \big(\int_{  (4Q_{\mu, l} )^{c}}  |T^c_{\sigma}a_{\mu, l}  (s ) |^{2}  (1+2^{\mu}  |s | )^{d+M}ds \big)^{\frac{1}{2}}\\
 & \lesssim  2^{-\mu(L+1-\frac{M}{2})}  \big(\int_{  (4Q_{\mu, l} )^{c}}  |s |^{-d-2L-2+M}ds \big)^{\frac{1}{2}}\cdot\tau  \big(\int _{\R}  |a_{\mu, l}(t) |^{2}dt \big)^{\frac{1}{2}}\\
 & \lesssim  2^{-\mu(L+1-\frac{M}{2})}2^{\mu(L+1-\frac{M}{2})} |Q_{\mu, l} |^{\frac{\alpha}{d}}=  |Q_{\mu, l} |^{\frac{\alpha}{d}}.
\end{split}
\end{equation*}
If we take $M=-d$ in the above inequality, we have $T_\sigma^c a_{\mu, l}\in L_1\big(\M;L_2^c(\R)\big)$. By approximation, we can assume that $\sigma(s,\xi)$ has compact $\xi$-support, so that
 $$
 T_\sigma^c a_{\mu, l}(s)=\int _{\R}\sigma (s,\xi)\widehat{a_{\mu, l}}(\xi)e^{2\pi {\rm i}s\cdot \xi}d\xi
$$
 is uniformly convergent. Moreover,  one can differentiate the integrand  and obtain always uniformly  convergent integrals.
 Then, for any  $ |\gamma|_1< K-\frac{d}{2}$, we have 
\begin{equation}\label{eq: s near origin}
\begin{split}
 & \tau  \big(\int_{4Q_{\mu, l}}  |D^{\gamma} T^c_{\sigma}a_{\mu, l}  (s ) |^{2}  (1+2^{\mu}  |s | )^{d+M}ds \big)^{\frac{1}{2}} \\
&\lesssim  |Q_{\mu, l} |^{\frac{\alpha}{d}}\tau  (\int _{\R}  |J^{K}a(t) |^{2}dt )^{\frac{1}{2}} \int_{\R}   (1+  |\xi | )^{2 |\gamma|_1-2K}d\xi\\
 & \lesssim   |Q_{\mu, l} |^{\frac{\alpha}{d}}\sum_{|\gamma|_1\leq K}\tau  \big(\int _{\R}  |D^{\gamma}a(t) |^{2}dt \big)^{\frac{1}{2}}\lesssim |Q_{\mu, l} |^{\frac{\alpha}{d}}. 
\end{split}
\end{equation}
 By a similar argument to that of  \eqref{eq: s off origin}, we have, for any $\gamma \in \mathbb{N}_0^d$ and $s\in (4Q_{\mu,l})^c$, 
\begin{equation}\label{eq: derivative s off origin}
 |D^\gamma T^c_{\sigma}a_{\mu, l}  (s ) |^{2}\lesssim 2^{-\mu(2L+2+d)}  |s |^{-2d-2L-2-2|\gamma|_1}\int _{\R}  |a_{\mu, l}(t) |^{2}dt.
\end{equation}
Therefore, we deduce that 
\begin{equation*}
\begin{split}
 &  \tau  \big(\int_{  (4Q_{\mu, l} )^{c}}  |D^{\gamma} T^c_{\sigma}a_{\mu, l}  (s ) |^{2}  (1+2^{\mu}  |s | )^{d+M}ds \big)^{\frac{1}{2}}\\
 & \lesssim  2^{-\mu(L+1-\frac{M}{2})}  \big(\int_{  (4Q_{\mu, l} )^{c}}  |s |^{-d-2L-2+M-2|\gamma |_1}ds \big)^{\frac{1}{2}}\cdot\tau  \big(\int _{\R}  |a_{\mu, l}(t) |^{2}dt \big)^{\frac{1}{2}}\\
 & \lesssim  2^{-\mu(L+1-\frac{M}{2})}2^{\mu(L+1-\frac{M}{2}+|\gamma |_1)} |Q_{\mu, l} |^{\frac{\alpha}{d}}=  |Q_{\mu, l} |^{\frac{\alpha}{d}-\frac{|\gamma |_1}{d}}.
\end{split}
\end{equation*}
Combining the estimates above, we get \eqref{eq:moledule}.
\end{proof}

On the other hand, we have the following lemma concerning the image of $(\alpha,1)$-atoms under the action of pseudo-differential operators.

 \begin{lem}
\label{lem:moledule 2}Let $\alpha\in \mathbb{R}$,  $\sigma \in S_{1,\delta}^{0}$
and $T^c_{\sigma}$ be the corresponding pseudo-differential operator. Let $K>\frac{d}{2}$ and  $b$ be an  $(\alpha,1)$-atom based on the cube $Q_{0,m}$. Then for any $M\in \mathbb{R}$, we have
\begin{equation}\label{eq:moledule 2}
\tau  \big(\int _{\R}  (1+  |s-m | )^{d+M}  |D^{\gamma}T^c_{\sigma}b(s) |^{2}ds \big)^{\frac{1}{2}}\lesssim1,\quad |\gamma|_1 < K-\frac{d}{2},
\end{equation}
where the relevant constant depends on $M$, $K$, $\gamma$ and $d$.
\end{lem}
\begin{proof}
The proof of this lemma is similar to that of the previous one. The only difference is that for an $(\alpha,1)$-atom, we do not necessarily have the moment cancellation; in this case, we need to use the extra decay of the kernel proved in Lemma \ref{Lem T K} for $|  t|  >1$.

If $s\in 4Q_{0,m}$, we follow the estimate for subatoms in the previous lemma. Applying the size estimate of $b$, we get
$$
\tau  \big(\int_{4Q_{0,m}}  (1+  |s-m | )^{d+M}  | T^c_{\sigma}b(s) |^{2}ds \big)^{\frac{1}{2}}\lesssim \tau   ( \int_{\R} |  J^K b(t)| ^2dt )^\frac{1}{2}\lesssim  1.
$$

If $s\in (4Q_{0,m})^c$ and $t\in 2Q_{0,m}$, we have $|  s-t|  \geq 1$. Then \eqref{eq:kernel'} gives
\begin{equation*}
\begin{split}
|  T^c_{\sigma} b(s)| ^2 &=   \Big| \int_{\R}  K(s,s-t)b(t)dt  \Big|^2\\
&\leq  \int _{2Q_{0,m}}\|  K(s,s-t)\|_\M ^2 dt \int _{2Q_{0,m}}|  b(t)|  ^2dt\\
&\lesssim  |  s-m |  ^{-2N}\int_{2Q_{0,m}}  |  b(t)|  ^2dt,
\end{split}
\end{equation*}
where the positive integer $N$ can be arbitrarily large. Thus
\begin{equation*}
\begin{split}
& \tau  \big(\int_{(4Q_{0,m})^c}  (1+  |s-m | )^{d+M}  |T^c_{\sigma}b(s) |^{2}ds \big)^{\frac{1}{2}}\\
& \lesssim   \big(\int_{(4Q_{0,m})^c}|  s-m |  ^{d+M-2N}ds \big)^\frac{1}{2} \tau   \big( \int_{2Q_{0,m}}  |  b(t)|  ^2dt \big)^\frac{1}{2}\lesssim 1.
\end{split}
\end{equation*}
Then, the estimates obtained above imply that
$$
\tau  \big(\int_{\R}  (1+  |s-m | )^{d+M}  |T^c_{\sigma}b(s) |^{2}ds \big)^{\frac{1}{2}}\lesssim 1.
$$
Similarly, we treat $D^{\gamma}T^c_{\sigma}b(s)$ as
 $$
\tau  \big(\int_{\R}  (1+  |s-m | )^{d+M}  |D^{\gamma}T^c_{\sigma}b(s) |^{2}ds \big)^{\frac{1}{2}}\lesssim 1, \quad |\gamma|_1 < K-\frac{d}{2}.
$$
Therefore, \eqref{eq:moledule 2} is proved.
\end{proof}

The following lemma shows that, if the symbol $\sigma$ satisfies some support condition, we can even estimate the $F_1^{\alpha,c}$-norm of the image of $(\alpha, Q_{\mu, l})$-subatoms under $T_\sigma^c$.

\begin{lem}\label{lem:moledule norm}
Let $\sigma\in S_{1,\delta}^{0}$
and $T^c_{\sigma}$ be the corresponding pseudo-differential operator. Assume that $\alpha\in \mathbb{R}$, $K\in \mathbb{N}$ satisfy $K>\frac{d}{2}$ and $K>\alpha+d$. If the $s$-support of $\sigma$ is in $( 2^{-\mu}l+4Q_{0,0})^c$, then for any $(\alpha,Q_{\mu, l})$-subatom $a_{\mu, l}$,  we have
$$
\| T_\sigma ^c a_{\mu ,l} \|_{F_1^{\alpha,c}}\lesssim 2^{-\mu(\frac{d}{2}+\iota)},
$$
where $\iota$ is a positive real number.
\end{lem}

\begin{proof}
Without loss of generality, we still assume $l=0$. We need to use the characterization of $F_1^{\alpha,c}$-norm by the following convolution kernels. Let $\kappa$ be a radial, real and infinitely differentiable function on $\R$ which is supported in the unit cube centered at the origin, and assume that $\widehat{\kappa}(0)>0$. We take $\widehat{\Phi}=|\cdot |^N\widehat{\kappa}$ with $N\in \mathbb{N}_0$ such that $\alpha+\frac{d}{2}-1<N<K-\frac{d}{2}$, and another test function $\Phi_0\in \mathcal{S}$ with $\supp \Phi_0\subset Q_{0,0}$. Let $\Phi_\e$ be the inverse Fourier transform of $\Phi(\e \xi)$. To simplify the notation, we denote $T_\sigma ^c a_{\mu ,l}$ by $\eta_{\mu,l}$. Then Theorem 4.2 in \cite{XX18} gives
 $$
  \| \eta_{\mu,l} \|  _{F_{1}^{\alpha,c}}\approx \|  \Phi _0* \eta_{\mu,l}\| _{1}+\Big\| \big(\int _0^1\e^{-2\alpha}|  \Phi_\e* \eta_{\mu,l} | ^2\frac{d\e}{\e} \big)^\frac{1}{2}\Big\| _{1}.
$$ 
 We notice that  $\Phi$ satisfies the moment cancellation up to order $N$. It follows that
\begin{equation}\label{eq: Taylor m}
\begin{split}
 \Phi_{\varepsilon}* \eta_{\mu,l}(s) & = \int_{\R} \Phi_\e  (t )[ \eta_{\mu,l}(s-t)- \eta_{\mu,l}(s)]dt\\
 &  = \int_{\R}\Phi_\e (t )\sum_{  |\gamma |_1=N+1}\frac{N+1}{\gamma !}(-t)^{\gamma}\int_0^1(1-\theta)^{N} D^{\gamma} \eta_{\mu,l}(s- \theta t)d\theta\, dt.
 \end{split}
\end{equation}
Applying the Cauchy-Schwarz inequality, we have 
\begin{equation}\label{eq: m pointwise 2}
\begin{split}
& \Big| \int_{|t|>\frac{  |s |}{2}}\Phi_\e (t)\sum_{  |\gamma |_1=N+1}\frac{N+1}{\gamma !}t^{\gamma}\int_0^1(1-\theta)^{N} D^{\gamma} \eta_{\mu,l}(s-\theta t)d\theta\, dt\Big|^2\\
&\lesssim   \sum_{  |\gamma |_1=N+1}\int_{\R}\int_0^1(1-\theta)^{2N} |D^{\gamma} \eta_{\mu,l}(s-\theta t)|^2 d\theta (1+|t|)^{-d-1}dt\\
&\;\;\;\; \cdot \int_{\R}|\Phi_\e (t)|^2 |t |^{2N+2}(1+|t|)^{d+1} dt.
\end{split}
\end{equation}
By \eqref{eq: derivative s off origin}, if $s-\theta t\in (4 Q_{0,0})^c$, we have 
$$
|D^{\gamma} \eta_{\mu,l}(s-\theta t)|^2\lesssim  2^{-\mu(2L+2+d)}  |s-\theta t |^{-2d-2L-2-2|\gamma|_1}\int _{\R}  |a_{\mu, l}(r) |^{2}dr.
$$
Therefore, using  the Cauchy-Schwarz inequality again, we have 
 \begin{equation}\label{eq: m pointwise 3}
 \begin{split}
&\Big\| \big(\int _0^1\e^{-2\alpha}|  \Phi_\e* \eta_{\mu,l} | ^2\frac{d\e}{\e} \big)^\frac{1}{2}\Big\| _{1} \\
& \lesssim  2^{-\mu(L+1+\frac{d}{2})} \big( \int_0^1 \e^{2N-d+2-2\alpha}\frac{d\e}{\e}\big)^\frac{1}{2} \int_{(2Q_{0,0})^c}  |s' |^{-d-L-N-2}ds' \int_{\R}(1+|t|)^{-d-1}dt\\
&\;\;\;\; \cdot \int_{\R}|\Phi(t')|^2|t'|^{2N+2}(1+|t'|)^{d+1}dt' \cdot \tau \big(\int _{\R}  |a_{\mu, l}(r) |^{2}dr \big)^\frac{1}{2}\\
& \lesssim  2^{-\mu(L+1+\frac{d}{2})} \tau \big(\int _{\R}  |a_{\mu, l}(t) |^{2}dt \big)^\frac{1}{2}\\
& \lesssim  2^{-\mu(L+1+\frac{d}{2}+\alpha)}  .
\end{split}
\end{equation}

It remains to estimate the $L_1$-norm of  $\Phi_0*\eta_{\mu,l}$, where $\Phi_0$ does not have the moment cancellation. Since $\supp \eta_{\mu ,l} \subset (4Q_{0,0})^c$ and by the support assumption of $\Phi_0$, we have $\supp \Phi _0*\eta_{\mu,l} \subset\{s\in\R: |s|\geq \frac{1}{2}\}$. By Lemma \ref{lem:moledule 2} and the fact that  $| \Phi_0 (s)|\lesssim (1+|s|)^{-d-R}$ for any $R\in \mathbb{N}$, we have 
\begin{equation*}
\begin{split}
& | \Phi _0*\eta_{\mu,l}(s) |^2= \Big|\int _{\R}\Phi _0(s-t)\eta_{\mu,l}(t)dt \Big|^2\\
&\leq   \Big|\int_{|  t| \geq \max\{\frac{|  s| }{2},1\}}\Phi _0(s-t)\eta_{\mu,l}(t)dt \Big|^2+\Big| \int_{1\leq |  t| <\frac{|  s| }{2}}\Phi _0(s-t)\eta_{\mu,l}(t)dt \Big|^2\\
&\leq    \int _{|  t| \geq\max\{\frac{|  s| }{2},1\}} (1+2^{\mu }  |  t| )^{-2d-2R}| \Phi _0(s-t)|  ^2dt \cdot\int_{\R}  (1+2^{\mu } |  t| )^{2d+2R}|  \eta_{\mu,l}(t)|  ^2dt \\
& \;\;\;\; +  \int _{1\leq |  t| <\frac{|  s| }{2}} (1+2^{\mu } |  t| )^{-2d-2R}| \Phi _0(s-t)|  ^2dt \cdot \int_{\R}  (1+2^{\mu } |  t| )^{2d+2R}|  \eta_{\mu,l}(t)|  ^2dt  \\
& \lesssim   \int_{\R} |  \Phi _0(t)|  ^2dt \, (1+2^{\mu } |  s| )^{-2d-2R}  \int_{\R}  (1+2^{\mu } |  t| )^{2d+2R}|  \eta_{\mu,l}(t)|  ^2dt \\
&\;\;\;\; +  \int _{|t|\geq 1} (1+2^{\mu } |  t| )^{-2d-2R}dt \,(1+ |  s| )^{-2d-2R}  \int_{\R}  (1+2^{\mu } |  t| )^{2d+2R}|  \eta_{\mu,l}(t)|  ^2dt \\
&\lesssim \Big( (1+2^{\mu } |  s| )^{-2d-2R}  +2^{-2\mu (d+R)}(1+ |  s|) ^{-2d-2R}\Big)   \int_{\R}  (1+ 2^{\mu } |  t| )^{2d+2R}|  \eta_{\mu,l}(t)|  ^2dt .
\end{split}
\end{equation*}
Then we use \eqref{eq:moledule} to  get,  for any $ R\in \mathbb{N}$, 
\begin{equation*}
\begin{split}
 \|  \Phi _0*\eta_{\mu,l}\|  _{1} & \lesssim  \Big( \int_{|s|\geq \frac{1}{2}} (1+2^{\mu } |  s| )^{-d-R} ds + 2^{-\mu (d+R)} \int_{|s|\geq \frac{1}{2}} (1+ |  s| )^{-d-R} ds \Big)\\
& \;\;\;\;\cdot \tau   (\int_{\R}  (1+ 2^\mu |  t| )^{2d+2R}|  \eta_{\mu,l}(t)|  ^2dt )^\frac{1}{2}\\
& \lesssim 2^{-\mu (d+R+\alpha)}.
\end{split}
\end{equation*}
Combining the estimates above, we see that, there exists $\iota>0$ such that 
$$
\|  T^c_{\sigma}a_{\mu,l} \| _{F_1^{\alpha,c}} =\|  \eta_{\mu,l} \| _{F_1^{\alpha,c}}\lesssim 2^{-\mu(\frac{ d}{2}+\iota) },
$$
which completes the proof.
\end{proof}

Since every $(\alpha, Q_{k, m})$-atom is a linear combination of subatoms, the above lemma helps us to estimate the image of $(\alpha, Q_{k, m})$-atoms under $T_\sigma^c$.

\begin{cor}\label{cor: moledule norm}
Let $\sigma\in S_{1,\delta}^{0}$
and $T^c_{\sigma}$ be the corresponding pseudo-differential operator. Assume that $\alpha\in \mathbb{R}$, $K\in \mathbb{N}$ satisfy $K>\frac{d}{2}$ and $K>\alpha+d$.  If the $s$-support of $\sigma$ is in $( 2^{-k}m+6Q_{0,0})^c$, then for any $(\alpha,Q_{k, m})$-atom $g$,we have
$$
\| T_\sigma ^c g \|_{F_1^{\alpha,c}}\lesssim 1.
$$
\end{cor}

\begin{proof}
Every $(\alpha,Q_{k, m})$-atom $g$ admits the form 
\[
g=\sum_{(\mu,l)\leq (k,m)}d_{\mu,l}a_{\mu,l}  \quad \text{with} \; \sum_{(\mu,l)\leq (k,m)}  |d_{\mu, l} |^{2} \leq  |Q_{k,m} |^{-1}=2^{kd}.
\]
By the support assumption of $\sigma$, $\sigma(s,\xi)=0$ if $s\in 2^{-\mu}l+4Q_{0,0} \subset 2^{-k}m+6Q_{0,0}$. Then, we can apply the previous lemma to every $a_{\mu,l}$ with $(\mu,l)\leq (k,m)$. The result is
\[
\|  T^c_{\sigma}a_{\mu,l} \| _{F_1^{\alpha,c}} \lesssim 2^{-\mu(\frac{ d}{2}+\iota) } \quad \text{with } \iota>0.
\]
Applying the Cauchy-Schwarz inequality, we get
\begin{equation}\label{e: estimate of H}
\begin{split}
  \|  T_\sigma ^c g \|  _{F_{1}^{\alpha,c}} & \leq  \sum_{(\mu,l)\leq (k,m)} |d_{\mu, l} | \cdot \|  T_\sigma^ca_{\mu,l} \|  _{F_{1}^{\alpha,c}}\\
  & \leq  \sum_{(\mu,l)\leq (k,m)} |d_{\mu, l} |  \cdot 2^{-\mu (\frac d 2 +\iota)}\\
 & \lesssim  (\sum_{(\mu,l)\leq (k,m)}  |d_{\mu, l} |^{2} )^{\frac{1}{2}} (\sum_{(\mu,l)\leq (k,m)}   2^{-\mu(d+2\iota)} )^{\frac{1}{2}}\\
 & \leq    (\sum_{(\mu,l)\leq (k,m)}  |d_{\mu, l} |^{2} )^{\frac{1}{2}}(\sum_{\mu\geq k} \frac{|2Q_{k,m}|}{|Q_{\mu, l}|} \cdot  2^{-\mu(d+2\iota)})^\frac{1}{2}  \\
 &\lesssim |Q_{k,m}|^{-\frac 1 2 }  \cdot 2^{-\frac{kd}{2}} = 1.
 \end{split}
\end{equation}
Thus, the assertion is proved.
\end{proof}

Likewise, we can estimate the image of $(\alpha,1) $-atoms under the pseudo-differential operator $T_\sigma^c$.

\begin{lem}\label{lem:moledule norm b}
Let $\sigma\in S_{1,\delta}^{0}$
and $T^c_{\sigma}$ be the corresponding pseudo-differential operator. Assume that $\alpha\in \mathbb{R}$, $K\in \mathbb{N}$ satisfy $K>\frac{d}{2}$ and $K>\alpha+d$. If the $s$-support of $\sigma$ is in 
$(k+4Q_{0,0})^c$ for some $k\in \mathbb{Z}^d$, then for any  
$(\alpha,1)$-atom $b$ such that $\supp b\subset 2Q_{0,k}$, we have
$$
\| T_\sigma ^c b \|_{F_1^{\alpha,c}}\lesssim 1.
$$
\end{lem}

\begin{proof}
The proof of this lemma is very similar to that of  Lemma \ref{lem:moledule norm}; it suffices to apply (the proof of) Lemma \ref{lem:moledule 2} instead of Lemma \ref{lem:moledule}.
\end{proof}

\begin{cor}\label{lem:moledule m norm}
Let $\sigma\in S_{1,\delta}^{0}$
and $T^c_{\sigma}$ be the corresponding pseudo-differential operator. Given $\alpha\in \mathbb{R}$, $K\in \mathbb{N}$ such that $K>\frac{d}{2}$ and $K>\alpha+d$, then for any  
$(\alpha,1)$-atom $b$,  we have
$$
\| T_\sigma ^c b \|_{F_1^{\alpha,c}}\lesssim 1.
$$
\end{cor}

\begin{proof}
Let $(\mathcal{X}_j)_{j\in \mathbb{Z}^d}$ be the smooth resolution of the unit in \eqref{eq: unit resolution}. We decompose $ T_\sigma ^c b$ as
$$
 T_\sigma ^c b=\sum_{j\in\mathbb{Z}^d}\mathcal{X}_j T_\sigma ^c b = \sum_{j\in\mathbb{Z}^d} T_{\sigma_j}^c b,
$$
where all $\sigma_j=\mathcal{X}_j(s)\sigma(s,\xi)$ belong to $ S_{1,\delta}^0$ uniformly in $j$.
Suppose that $b$ is supported in $2Q_{0,k}$ with $k\in \mathbb{Z}^d$. We split the above summation into two parts:
\begin{equation}\label{eq: split sum}
 T_\sigma ^c b=\sum_{j\in k+6Q_{0,0}}\mathcal{X}_j T_\sigma ^c b+\sum_{j\notin k+ 6Q_{0,0}}\mathcal{X}_j T_\sigma ^c b.
\end{equation}
Applying Lemma \ref{lem:moledule 2} with  $M=-d$ to the symbol $\mathcal{X}_j(s)\sigma(s,\xi)$, we get, for any $j\in \mathbb{Z}^d$,
$$
\tau \big( \int_{j+2Q_{0,0}} |D^\gamma(\mathcal{X}_j \,  T_\sigma ^c b(s)) |^2 ds  \big)^\frac{1}{2}\lesssim 1, \quad \forall\, |\gamma |_1\leq [\alpha]+1.
$$
Thus, $\mathcal{X}_j \,  T_\sigma ^c b$ is a bounded multiple of  an $(\alpha,1)$-atom. So the first term on the right hand side of \eqref{eq: split sum} is a finite sum of $(\alpha,1)$-atoms, and thus has bounded $F_1^{\alpha,c}$-norm. Now we deal with the second term. Note that the $s$-support of the symbol $\sum_{j\notin k+6Q_{0,0}}\mathcal{X}_j (s)\sigma(s,\xi)$ is in $(k+4Q_{0,0})^c$. Then, it suffices to apply Lemma \ref{lem:moledule norm b} to this symbol, so that
\[
\big\| \sum_{j\notin k+6Q_{0,0}}\mathcal{X}_j T_\sigma ^c b \|_{F_1^{\alpha,c}}\lesssim 1.
\]
The proof is complete.
\end{proof}

 \section{ Regular symbols}\label{section-regular}

In this section, we study the continuity of the pseudo-differential operators with regular symbols in $S_{1,\delta}^0$ ($0\leq \delta <1$) on Triebel-Lizorkin spaces. We start by presenting an $L_2$-theorem. It is a noncommutative analogue of the corresponding  classical theorem, which can be found in many works, for instance, \cite{{Stein1993}, {Tay}, {R-T}}. Then, we will use the atomic decompositions introduced in section \ref{prelimi} to deduce the $F_{p}^{\alpha,c}$-boundedness. Different from the pseudo-differential operators with the forbidden symbols in $S_{1,1}^0$, which will be treated in the next section, our proof stays at the level of atoms; in other words, we do not need the subtler decomposition that every $(\alpha,Q)$-atom can be written as a linear combination of subatoms.

 \begin{thm}\label{thm:bdd}
Let $0\leq \delta<1$, $\sigma\in S_{1,\delta}^{0}$ and $\alpha\in \mathbb{R}$. Then $T^c_{\sigma}$ is bounded on $F_{p}^{\alpha,c}(\mathbb{R}^d,\mathcal{M})$ for any $1\leq p \leq \infty$.
 \end{thm}

In order to fully understand the image of an $(\alpha,Q)$-atom under the action of a pseudo-differential operator, we  need to study its $L_1(\M;L_2^c(\R))$-boundedness. 
 We will work on the exotic class $S_{\delta,\delta}^0$ with $0\leq\delta <1$, since we have the inclusion $S_{1,\delta}^0\subset S_{\delta,\delta}^0$. The Cotlar-Stein almost orthogonality lemma plays a crucial role in our proof. 
Namely, given a family of operators $(T_j)_j\subset B(H)$ with $H$ a Hilbert space, and a  positive sequence $\big(c(j) \big)_j$ such that $\sum_j c(j)=C<\infty$, if the $T_j$'s satisfy:
$$
\|T_k^*T_j\|_{B(H)} \leq | c(k-j) |^2
$$
and 
$$
\|T_k T_j^*\|_{B(H)}\leq  | c(k-j) |^2,
$$ 
 then we have 
$$
\big\| \sum_j T_j \big\| _{B(H)}\leq C.
$$

We begin with a simpler case where $\delta=\rho=0$. The following lemma is modelled after \cite[Proposition~VII.2.4]{Stein1993}; we include a proof for the sake of completeness.

\begin{lem}\label{lem: L2 bdd}
Assume $\sigma\in S_{0,0}^0$. Then $T_\sigma^c$ is bounded on $L_2(\N)$.
\end{lem}

\begin{proof}
By the Plancherel formula, it is enough to prove the $L_2(\N)$-boundedness of the following operator:
\[
S^c_\sigma (f)(s)=\int_{\R} \sigma (s,\xi) f(\xi) e^{2\pi{\rm i}s\cdot \xi}d\xi. 
\]
Let us make use of the resolution of the unit $(\mathcal{X}_k)_{k\in\mathbb{Z}^d}$ introduced in \eqref{eq: unit resolution} to decompose $S^c_\sigma$ into almost orthogonal  pieces. Denote ${\rm k}=(k,k')\in \mathbb{Z}^d \times \mathbb{Z}^d$, and set
\[
\sigma_{\rm k}(s,\xi)=\mathcal{X}_k(s)\sigma(s,\xi)\mathcal{X}_{k'}(\xi),
\]
Then, the series $\sum_{{\rm k}\in \mathbb{Z}^d \times \mathbb{Z}^d} S^c_{\sigma_{\rm k}}$ converges in the strong operator topology and 
\[
S^c_\sigma = \sum_{{\rm k}\in \mathbb{Z}^d \times \mathbb{Z}^d} S^c_{\sigma_{\rm k}}.
\]
We claim that $(S^c_{\sigma_{\rm k}})_{\rm k}$ satisfies the almost-orthogonality estimates, i.e., for any $N\in \mathbb{N}$, 
$$
\| (S^c_{\sigma_{\rm k}})^*S^c_{\sigma_{\rm j}}\|_{B(L_2(\N))} \leq C_N(1+ | {\rm k}-{\rm j} | )^{-2N},
$$
and 
$$
\|S^c_{\sigma_{\rm k}} (S^c_{\sigma_{\rm j}})^*\|_{B(L_2(\N))}\leq C_N(1+ | {\rm k}-{\rm j} | )^{-2N},
$$ 
where the constant $C_N$ is independent of ${\rm k}=(k,k')$ and ${\rm j}=(j,j')$.
Armed with this claim, we can then apply the Cotlar-Stein almost orthogonality lemma stated previously to the operators $(S^c_{\sigma_{\rm k}})_{\rm k}$ with $c({\rm j})=(1+ |{\rm j}|)^{-N}$, $N>2d$. Then, we will have
\[
\|S^c_\sigma\|_{B(L_2(\N))}= \| \sum_{{\rm k}\in \mathbb{Z}^d \times \mathbb{Z}^d} S^c_{\sigma_{\rm k}}\|_{B(L_2(\N))}\leq C.
\]

Now we prove the claim. Note that for any $f\in L_2(\N)$,
\begin{equation*}
(S^c_{\sigma_{\rm k}})^*S^c_{\sigma_{\rm j}}(f)(\xi)=\int_{\R} \sigma_{{\rm k},{\rm j}}(\xi,\eta) f(\eta)d\eta,
\end{equation*}
where 
\begin{equation}\label{eq: orthogonal formula}
\sigma_{{\rm k},{\rm j}}(\xi,\eta) =\int_{\R} \sigma^*_{{\rm k}}(s,\xi)\sigma_{{\rm j}}(s,\eta)e^{2\pi {\rm i}s\cdot(\eta-\xi)} ds.
\end{equation}
By the definition of $\sigma_{\rm k}$, we see that if $k-j \notin 2Q_{0,0}$ (recalling that $Q_{0,0}$ is the unit cube centered at the origin), $\sigma_{\rm k}$ and $\sigma_{\rm j}$ have disjoint $s$-support, so 
\[
\sigma_{\rm k}^*\sigma_{\rm j}=0.
\]
When $k-j \in 2Q_{0,0}$, using the identity 
\[
(1-\Delta _s)^N e^{2\pi {\rm i}s\cdot(\eta-\xi)}=(1+4\pi^2|\eta-\xi |^2)^Ne^{2\pi {\rm i}s\cdot(\eta-\xi)},
\]
we integrate \eqref{eq: orthogonal formula} by parts, which gives
\[
\| \sigma_{{\rm k},{\rm j}}(\xi,\eta)\|_{\M}\leq C_N \mathcal{X}_{k'}(\xi)\mathcal{X}_{j'}(\eta)(1+|\xi-\eta|)^{-2N}.
\]
Whence,
\begin{equation}\label{eq: symbol integral}
\max\Big\{\int_{\R}\| \sigma_{{\rm k},{\rm j}}(\xi,\eta)\|_{\M} d\xi, \int_{\R}\| \sigma_{{\rm k},{\rm j}}(\xi,\eta)\|_{\M} d\eta\Big\} \leq C_N' (1+|{\rm k}-{\rm j}|)^{-2N}. 
\end{equation}
For any $f\in L_2(\N)$, there exists $g\in L_2(\N)$ with norm one such that 
$$
\big\| (S^c_{\sigma_{\rm k}})^*S^c_{\sigma_{\rm j}} f  \big\|_{L_2(\N)}= \Big|\tau \int_{\R} \int_{\R} \sigma_{{\rm k},{\rm j}}(\xi,\eta) f(\eta)\,d\eta\, g^*(\xi)\,d\xi\Big| .
$$
Applying the H\"older inequality and \eqref{eq: symbol integral}, we get 
\begin{equation*}
\begin{split}
& \Big|\tau \int_{\R} \int_{\R} \sigma_{{\rm k},{\rm j}}(\xi,\eta) f(\eta)\,d\eta \,g^*(\xi)\,d\xi\Big| \\
& \leq  \big( \tau \int_{\R} \int_{\R} \| \sigma_{{\rm k},{\rm j}}(\xi,\eta)\|_{\M}| f(\eta)|^2d\eta d\xi \big)^\frac{1}{2} \big( \tau \int_{\R} \int_{\R} \| \sigma_{{\rm k},{\rm j}}(\xi,\eta)\|_\M |g(\xi)|^2d\xi d\eta \big) ^\frac{1}{2} \\
& \leq C_N' (1+|{\rm k}-{\rm j}|)^{-2N} \| f\| _{L_2(\N)}.
\end{split}
\end{equation*}
Thus, $\| (S^c_{\sigma_{\rm k}})^*S^c_{\sigma_{\rm j}}\|_{B(L_2(\N))} \leq C_{N}'(1+ | {\rm k}-{\rm j} | )^{-2N}$.
On the other hand, a similar argument also shows that 
\[
\| S^c_{\sigma_{\rm k}}(S^c_{\sigma_{\rm j}})^*\|_{B(L_2(\N))} \leq C_{N}'(1+ | {\rm k}-{\rm j} | )^{-2N},
\]
which proves the claim.
\end{proof}

 A weak form of Cotlar-Stein's almost orthogonality lemma also plays a crucial role. As before, we suppose that $\sum_j c(j)=C<\infty$. This time we assume that the $T_j$'s satisfy:
\begin{equation}\label{CS-uniform}
\sup_{j}\| T_j\|_{B(H)}\leq C
\end{equation}
and the following conditions hold for $j\neq k$:
\begin{equation}\label{CS-orth}
\|T_jT_k^*\|_{B(H)}=0 \quad \text{and}\quad \|T_j^*T_k\|_{B(H)}\leq c(j)c(k).
\end{equation}
Then we have 
$$
\big\| \sum_j T_j \big\| _{B(H)}\leq \sqrt{2}C.
$$

\begin{lem}\label{lem: delta L2 bdd}
Let $\sigma\in S_{\delta,\delta}^0$  with $0\leq \delta<1$. Then $T_\sigma^c$ is bounded on $L_2(\N)$.
\end{lem}

\begin{proof}
To prove this lemma, we apply Cotlar's lemma as stated above. Let $(\widehat{\varphi}_j)_{j\geq 0}$ be the resolution of the unit defined in \eqref{eq:resolution of unity}. We can decompose $T_\sigma^c$ as follows:
$$
T_\sigma^c =\sum_{j=0}^\infty T_{\sigma_j}^c=\sum_{j \text{ even }} T_{\sigma_j}^c+\sum_{j \text{ odd }} T_{\sigma_j}^c,
$$
where $\sigma_j(s,\xi)= \widehat{\varphi}_j (\xi)\sigma (s,\xi) $. Note  that the symbols in either odd or even summand have disjoint $\xi$-supports. We will only treat the odd part, since the other part can be dealt with in a similar way. It is clear that $T_{\sigma_j}^c(T_{\sigma_k}^c)^*=0$ if $j\neq k$, since $T_{\sigma_j}^c(T_{\sigma_k}^c)^*=T_{\sigma}^c M_{\widehat{\varphi}_j }M_{\overline{\widehat{\varphi}}_k}(T_{\sigma}^c)^*$ and 
$\widehat{\varphi}_j$, $\widehat{\varphi}_k$  have disjoint supports. Now let us estimate the second inequality in \eqref{CS-orth}, i.e. the norm of $(T_{\sigma_k}^c)^*T_{\sigma_j}^c$. Since
$$
(T_{\sigma_k}^c)^*(f)(s)=\int_{\R}\int_{\R}\sigma_k^*(t,\xi)f(t) e^{2\pi{\rm i}\xi\cdot (s-t)}dtd\xi, 
$$
and 
$$
T_{\sigma_j}^c(f)(t)=\int_{\R}\int_{\R}\sigma_j(t,\eta)f(r) e^{2\pi{\rm i}\eta\cdot (t-r)}drd\eta.
$$
Then we have
$$
(T_{\sigma_k}^c)^*T_{\sigma_j}^c(f)(s)=\int_{\R}K(s,r)f(r)dr,
$$
with 
$$
K(s,r)=\int_{\R}\int_{\R}\int_{\R}\sigma_k^*(t,\xi)\sigma_j(t,\eta)e^{2\pi{\rm i}[\eta\cdot (t-r)+\xi \cdot (s-t)]}dt d\eta d\xi.
$$
Writing 
$$
e^{2\pi{\rm i}(\eta-\xi)\cdot t}= \frac{(1-\Delta_t)^N}{(1+4\pi^2|\xi-\eta|^2)^N}  e^{2\pi{\rm i}(\eta-\xi)\cdot t},
$$
$$
e^{2\pi{\rm i}(t-r)\cdot \eta}= \frac{(1-\Delta_\eta)^N}{(1+4\pi^2|t-r|^2)^N}  e^{2\pi{\rm i}(t-r)\cdot \eta},
$$
and 
$$
e^{2\pi{\rm i}(s-t)\cdot \xi}= \frac{(1-\Delta_\xi)^N}{(1+4\pi^2|s-t|^2)^N} e^{2\pi{\rm i}(s-t)\cdot \xi},
$$
we use the integration by parts with respect to the variables $t$, $\xi$ and $\eta$. By standard calculation (see \cite[Theorem 2, p. 286]{Stein1993} for more details), we get
$$
\| K(s,r) \|_\M\lesssim 4^{\max (k,j)((\delta-1)N+d)}\int Q(s-t)Q(t-r)dt,
$$
where $Q(t)=(1+|t|)^{-2N}$, if $k\neq j$. Denote $K_0(s,r)=\int Q(s-t)Q(t-r)dt$, then 
\begin{equation}\label{eq: kernel}
\int_{\R}K_0(s,r)ds=\int_{\R}K_0(s,r)dr=\big(\int_{\R}(1+|t|)^{-2N}dt\big)^2<\infty. 
\end{equation}
For any $f\in L_2(\N)$, there exists $g\in L_2(\N)$ with norm one such that 
$$
\big\| (T_{\sigma_k}^c)^*T_{\sigma_j}^c f  \big\|_{L_2(\N)}=\Big|\tau \int_{\R} \int_{\R} K(s,r)f(r)g(s)drds\Big|.
$$
Applying the H\"older inequality and \eqref{eq: kernel}, we get 
\begin{equation*}
\begin{split}
& \Big|\tau \int_{\R} \int_{\R} K(s,r)f(r)g(s)drds\Big| \\
& \leq  \big( \tau \int_{\R} \int_{\R} \| K(s,r)\|_\M |f(r)|^2dsdr\big)^\frac{1}{2} \big( \tau \int_{\R} \int_{\R} \| K(s,r)\|_\M |g(s)|^2dsdr\big) ^\frac{1}{2} \\
& \lesssim 4^{\max (k,j)((\delta-1)N+d)}\| f\| _{L_2(\N)},
\end{split}
\end{equation*}
which implies that 
$$
\big\| (T_{\sigma_k}^c)^*T_{\sigma_j}^c\big\|_{B(L_2(\N))}\lesssim  c(j)c(k), \quad j\neq k,
$$
with $c(j)=2^{j((\delta-1)N+d)}$. If we take $N>\frac{d}{1-\delta}$, the sequence $\big(c(j)\big)_j$ is summable. 

In order to  apply Cotlar-Stein's lemma, it remains to show that $T_{\sigma_j}^c$'s satisfy \eqref{CS-uniform}. To this end, we do some technical modifications. Set 
$$
\widetilde{\sigma}_j=\sigma_j (2^{-j\delta}\cdot, 2^{j\delta}\cdot).
$$
We can easily check that the $\widetilde{\sigma}_j$'s belong to $S_{0,0}^0$, uniformly in $j$. Then, by Lemma \ref{lem: L2 bdd}, the $T_{\widetilde{\sigma}_j}^c$'s are bounded on  $L_2(\N)$ uniformly in $j$. If $\Lambda_j$ denotes the dilation operator given by 
$$
\Lambda_j (f)=f(2^{j\delta}\cdot),
$$ 
then, we can easily verify that
$$
T_{\sigma_j}^c=\Lambda_j T_{\widetilde{\sigma}_j}^c\Lambda_j^{-1}.
$$
Thus, $$
\|T_{\sigma_j}^c\|_{B(L_2(\N))} \leq \| T_{\widetilde{\sigma}_j}^c \| _{B(L_2(\N))}<\infty.
$$
Therefore, $(T_{\sigma_j}^c)_{j\geq 0}$ satisfy the assumptions of Cotlar's lemma. So we get 
$$
\|T_{\sigma}^c\|_{B(L_2(\N))}=\|\sum_{j=0}^\infty T_{\sigma_j}^c\|_{B(L_2(\N))} < \infty.
$$
Thus, $T_{\sigma}^c$ is bounded on $L_2(\N)$.
\end{proof}

Let $0\leq \delta <1$. Since $S_{1,\delta}^0\subset S_{\delta,\delta}^0$, the following corollary is obvious.
\begin{cor}\label{cor: L_2 bdd}
Let $\sigma\in S_{1,\delta}^0$  with $0\leq \delta<1$. Then $T_\sigma^c$ is bounded on $L_2(\N)$.
\end{cor}

Furthermore, we can deduce the boundedness of $T_\sigma^c$ on $L_1\big(\M;L_2^c(\R)\big)$ from the above corollary. 

\begin{lem}\label{lem:L_1 bdd}
Let $\sigma\in S_{1,\delta}^0$  with $0\leq \delta<1$. Then $T_\sigma^c$ is bounded on $L_1\big(\M;L_2^c(\R)\big)$.
\end{lem}

\begin{proof}
Since $0\leq \delta<1$, Proposition \ref{prop: adjoint} tells us that the adjoint $(T_\sigma^c)^*$ of $T_\sigma^c$ is still in the class $S_{1,\delta}^0$. Thus, by duality,  it is enough to prove the boundedness of $(T_\sigma^c)^*$ on $L_\infty\big(\M;L_2^c(\R)\big)$. Indeed, there exists $u\in L_2(\M)$ with norm one such that 
\begin{equation*}
\begin{split}
\Big\| \big( \int_{\R}|(T_\sigma^c)^*(f)(s)|^2ds \big)^\frac{1}{2} \Big\|_{\M} & =\Big( \int_{\R}\langle |(T_\sigma^c)^*(f)(s)|^2 u,u \rangle_{L_2(\M)}ds \Big)^\frac{1}{2}\\
&= \Big( \int_{\R} \|(T_\sigma^c)^*(fu)(s)\|^2_{L_2(\M)}ds \Big)^\frac{1}{2}.
\end{split}
\end{equation*}
Then, applying Corollary \ref{cor: L_2 bdd} to $(T_\sigma^c)^*$, we get
$$
\Big( \int_{\R} \|(T_\sigma^c)^*(fu)(s)\|^2_{L_2(\M)}ds \Big)^\frac{1}{2}\lesssim  \Big( \int_{\R} \| f(s)u\|^2_{L_2(\M)}ds \Big)^\frac{1}{2}\leq \Big\| \big( \int_{\R}|f(s)|^2ds \big)^\frac{1}{2} \Big\|_{\M}.
$$
Thus, we conclude the boundedness of $T_\sigma^c$ on $L_1(\M;L_2^c(\R))$. 
\end{proof}

Now we are ready to prove the main theorem in this section.

\begin{proof}[Proof of Theorem \ref{thm:bdd}]
\emph{Step 1.}
 We begin with the special case $p=1$ and $\alpha =0$. Since $F_1^{0,c}(\R,\M)=\h_1^c(\R,\M)$ with equivalent norms,  the assertion is equivalent to saying that when $\sigma\in S_{1,\delta}^0$ with $0\leq \delta<1$, $T_\sigma^c$ is bounded on $\h_1^c(\R,\M)$. By the atomic decomposition stated in Theorem \ref{thm: atomic decop T_L}, it suffices to prove that, for any atom $b$ based on a cube with side length 1 and any atom $g$  based on a cube with side length less than 1, we have
\[
\| T_\sigma^c b \|_{\h_1^c}\lesssim 1 \quad \text{and}\quad \| T_\sigma^c g \|_{\h_1^c}\lesssim 1.
\] 
Corollary \ref{lem:moledule m norm} tells us that 
\[
\| T_\sigma^c b \|_{\h_1^c}\lesssim 1.
\]
Thus, it remains to consider the atom $g$  based on cube $Q$ with $|Q|<1$. Without loss of generality, we may assume that $Q$ is centered at the origin.
Let $(\mathcal{X}_j)_{j\in\mathbb{Z}^d}$ be the resolution of the unit defined in \eqref{eq: unit resolution} and $\mathcal{X}^Q_j= \mathcal{X}_j(l(Q)^{-1} \cdot)$ for $j\in\mathbb{Z}^d$. Then, we have $\supp \mathcal{X}^Q_j \subset l(Q)j+2Q$. Now, set $h_1=\sum_{j\in 6Q_{0,0}}\mathcal{X}^Q_j $ and $h_2=\sum_{j\notin 6Q_{0,0}}\mathcal{X}^Q_j $. By the support assumption of $\mathcal{X}^Q_j $, it is obvious that $\supp h_1\subset 8Q$, $\supp h_2\subset (4Q)^c$.  Moreover, 
$$
h_1(s)+h_2(s)=1 \quad, \forall\, s\in \R.
$$
Now we decompose $\sigma$  into two parts:
$$
\sigma (s,\xi)=h_1(s) \sigma (s,\xi)+h_2(s)\sigma (s,\xi) \stackrel{\mathrm{def}}{=}  \sigma ^1(s,\xi)+\sigma^2 (s,\xi).
$$
Note that $\sigma ^1$ and $\sigma ^2$ are still in the class $S_{1,\delta}^0$ and
$$
T_\sigma^c g=T_{\sigma^1}^c g+T_{\sigma^2}^c g.
$$
Firstly, we deal with  the symbol $\sigma ^1$ which has compact $s$-support. We consider the  adjoint operator $(T^{c}_{\sigma})^*$ of $T_\sigma^c$. 
Since $\delta<1$, by Proposition \ref{prop: adjoint}, there exists $\widetilde {\sigma}\in S_{1,\delta}^0$ such that 
$$
 (T^{c}_{\sigma})^*=T^{c}_{\widetilde{\sigma}}.
 $$
 If we take $\zeta_j(s)= \mathcal{X}^Q_j(s)\widetilde{\sigma}(s,0)^*$ for $j\in 6Q_{0,0}$, then $\zeta_j$ is an $\mathcal{M}$-valued infinitely differentiable function with all derivatives belonging to $L_\infty(\mathcal{N})$. Denote by $m_{\zeta_j }^c$ the pointwise multiplication $g\mapsto \zeta_j g$. Then, we have 
\begin{equation*}
\supp m_{\zeta_j }^c  g\subset l(Q)j+2Q.
\end{equation*}
and
\begin{equation}\label{eq: Mg size condition}
 \tau  \big(\int _{\R} | m_{\zeta_j}^c g(s) |^2 ds \big)^\frac{1}{2}  \lesssim |Q|^{-\frac{1}{2}}.
\end{equation}
This indicates that, except for the vanishing mean property, each $m_{\zeta_j}^c g$ coincides with a bounded multiple of an $\h_1^c$-atom defined in Definition \ref{def: atom h1}. Now let us set $\sigma^1_j(s,\xi)=\mathcal{X}^Q_j(s)\sigma(s,\xi)$ for $j\in 6Q_{0,0}$ and set $T_j^c=T^c_{\sigma^1_j}- m^c_{\zeta_j}$. It is clear that $\supp T_j^c g\subset l(Q)j+2Q$.  Since $(m_{\zeta_j}^c)^*=m_{{\zeta_j}^*}^c$ and $(T_{\sigma_j^1}^c)^*x=\widetilde{\sigma^1_j}(s,0)x={\zeta_j}^*x$ for every $x\in \M$, we have 
$$
\tau \big(\int _{l(Q)j+2Q} T_j^c g(s)ds\, \cdot x\big)=\langle T_j^c g,x \rangle=\langle g, (T_j^c)^*x\rangle=\langle g, (T^c_{\sigma_j^1}-m_{\zeta_j}^c)^*x\rangle=0.
$$ 
Hence, $T_j^c g$ has vanishing mean.
Moreover, applying  Lemma \ref{lem:L_1 bdd} and \eqref{eq: Mg size condition}, we get   
\begin{equation*}
\begin{split}
\tau \big( \int_{l(Q)j+2Q}|T_j^c g(s)|^2ds\big)^\frac{1}{2}
& \leq \tau \big( \int_{l(Q)j+2Q}|T_{\sigma_j^1}^c g(s)|^2ds\big)^\frac{1}{2}+ \tau\big( \int_{l(Q)j+2Q}|m_{\zeta_j}^c g(s)|^2ds\big)^\frac{1}{2}\\
&\lesssim \tau \big( \int_{2Q}|g(s)|^2ds\big)^\frac{1}{2} + |Q|^{-\frac{1}{2}} \lesssim |Q|^{-\frac{1}{2}}.
\end{split}
\end{equation*}
Combining the above estimates, we see that $T_j^c$ maps $\h_1^c$-atoms to $\h_1^c$-atoms. Thus, $T_j^c$ is bounded on $\h_1^c(\R,\M)$, and so are $T_{\sigma^1_j}^c$ and $T_{\sigma^1}^c$.

\emph{Step 2.}  Now let us consider $T^c_{\sigma^2}$. By Theorem \ref{thm: atomic decop T_L}, we may assume that $g$ has moment cancellations up to order $L> \frac d 2 -1$. Note that  $\supp  T^c_{\sigma^2}g \subset (4Q)^c$. And if $s\in (4Q)^c$, following the argument in \eqref{eq: s off origin} with $g$ in place of $a_{\mu,l}$, we get
$$
|T^c_{\sigma^2} g(s)|^2\lesssim l(Q)^{2L+2+d}  |s |^{-2d-2L-2}\int _{2Q}  | g(t) |^{2}dt.
$$
Then for $M<2L+2$,
\begin{equation}\label{eq: g image off origin}
\begin{split}
 &  \tau  \big(\int_{ (4Q)^{c}}  |T^c_{\sigma^2} g (s ) |^{2}  (1+l(Q)^{-1} |s | )^{d+M}ds \big)^{\frac{1}{2}}\\
 & \lesssim  l(Q)^{L+1-\frac{M}{2}}  \big(\int_{ (4Q)^{c}}  |s |^{-d-2L-2+M}ds \big)^{\frac{1}{2}}\cdot\tau  (\int _{2Q}  |g(t) |^{2}dt )^{\frac{1}{2}}\\
 & \lesssim  l(Q)^{L+1-\frac{M}{2}}l(Q)^{-L-1+\frac{M}{2}} |Q|^{-\frac{1}{2}}=   |Q |^{-\frac{1}{2}}.
\end{split}
\end{equation}
Moreover, we claim that $T^c_{\sigma^2}g$  can be decomposed as follows:
$$
T^c_{\sigma^2}g=\sum_{m\in \mathbb{Z}^d} \nu_m H_m,
$$ 
where $\sum_m |\nu_m |\lesssim1$ and the $H_m$'s are $\h_1^c$-atoms. Then, by  \eqref{thm:atomic h1}, we will get $\| T^c_{\sigma^2}g \|_{\h_1^c}\lesssim  1$.  Now let us prove the claim. Since $L>\frac d 2 -1$, we can choose $M$ such that $M>d$ and $M<2L+2$. Take $\nu_m =|Q|^{-\frac{1}{2}} (1+l(Q)^{-1} |m|)^{-\frac{d+M}{2}}$ and $H_m=\nu_m^{-1}\mathcal{X}_{m}T^c_{\sigma^2}g$, where $(\mathcal{X}_m)_{m\in\mathbb{Z}^d}$ denotes again the smooth resolution of the unit \eqref{eq: unit resolution}, i.e.
$$
1=\sum_{m\in\mathbb{Z}^d}\mathcal{X}_m(s), \quad \forall \, s\in \mathbb{R}^d.
$$
Applying \eqref{eq: g image off origin}, we have 
\begin{equation*}
\begin{split}
& \tau \big(\int_{2Q_{0,m}} |H_m(s)|^2 ds\big)^\frac{1}{2} \\
&  \lesssim \nu_m^{-1} (1+l(Q)^{-1} |m|)^{-\frac{d+M}{2}} \tau  \big(\int_{ (4Q)^{c}}  |T^c_{\sigma^2}g (s ) |^{2}  (1+l(Q)^{-1}  |s | )^{d+M}ds\big)^\frac{1}{2}\lesssim 1.
\end{split}
\end{equation*}
And the normalizing constants $\nu_m$ satisfy 
\begin{equation*}
\begin{split}
\sum_m |\nu_m | &=|Q |^{-\frac{1}{2}}\sum_m (1+l(Q)^{-1} |m|)^{-\frac{d+M}{2}}\\
& \leq  |Q|^{-\frac{1}{2}}\int_{\R} (1+l(Q)^{-1} |s|)^{-\frac{d+M}{2}}ds\lesssim 1.
\end{split}
\end{equation*}
Combining the estimates of $T_{\sigma_1}^c g$ and $T_{\sigma_2}^c g$, we conclude that $\| T_{\sigma}^c g \|_{\h_1^c}\lesssim 1$. Thus, $T_\sigma^c$ is bounded on $\h_1^c(\R,\M)$.

\emph{Step 3.} For the case where $p=1$ and $\alpha\neq 0$, we use the lifting property of Triebel-Lizorkin spaces (see \cite[Proposition 3.4]{XX18}.
By the property of the composition of pseudo-differential operators stated in Proposition \ref{prop: composition}, we see that 
$$
T_{\sigma^\alpha}^c
=J^{\alpha} T_\sigma^c J^{-\alpha} 
$$
is still a pseudo-differential operator with symbol $\sigma^\alpha$ in $S_{1,\delta}^0$. Then for any $f \in F_1^{\alpha,c}(\R,\M)$, we have
$$\|T_\sigma^c f \|_{F_1^{\alpha,c}} =  \| J^{-\alpha}     T_{\sigma^\alpha} ^c   J^\alpha f \|_{F_1^{\alpha,c }} \approx \|   T_{\sigma^\alpha} ^c   J^\alpha f \|_{\h_1^{ c }   }   \lesssim \|    J^\alpha f \|_{\h_1^{ c }   }  \approx \|  f \|_{F_1^{\alpha,c}} .
 $$
Hence, $T_\sigma^c$ is bounded on $F_1^{\alpha,c}(\R,\M)$.
 
\emph{Step 4.} Finally, we deal with the case $1<p\leq \infty$.  By the previous steps, $(T_\sigma^c)^*=T_{\widetilde{\sigma}}^c$ is bounded on
 $F_1^{-\alpha,c}(\R,\M)$ with $\alpha\in \mathbb{R}$, then it is clear that $T_\sigma^c$ is bounded on $F_\infty^{\alpha,c}(\R,\M)$.
 Given $1< p<\infty$ and  $\alpha \in \mathbb{R}$, by interpolation
$$\big( F_\infty^{\alpha,c}(\R,\M), F_1^{\alpha,c}(\R,\M)\big)_{\frac{1}{p}}=F_p^{\alpha,c}(\R,\M),$$
we get the boundedness of $T_\sigma^c$ on $F_p^{\alpha,c}(\R,\M)$.
\end{proof}

 \begin{rmk}
A special case of Theorem \ref{thm:bdd} is that if the symbol is scalar-valued, then 
$$\int_{\mathbb{R}^{d}}\sigma(s,\xi)\widehat{f}(\xi)e^{2\pi {\rm i}s\cdot\xi}d\xi = \int_{\mathbb{R}^{d}}\widehat{f}(\xi)\sigma(s,\xi)e^{2\pi {\rm i}s\cdot\xi}d\xi.$$ 
In this case, $T^c_{\sigma}$ is also bounded on $\h_{p}^r(\mathbb{R}^d,\mathcal{M})$ for any  $1\leq p \leq \infty$. By \eqref{equi-Lp-hp}, we deduce that $T^c_{\sigma}$ is bounded on $L_p(\N)$.

\end{rmk}

 \begin{cor}\label{cor:bdd}
Let $n,\alpha \in \mathbb{R}$, $0\leq \delta<1$ and $\sigma\in S_{1,\delta}^{n}$.  Then $T^c_{\sigma}$ is bounded from $F_{p}^{\alpha,c}(\mathbb{R}^d,\mathcal{M})$ to $F_{p}^{\alpha-n,c}(\mathbb{R}^d,\mathcal{M})$ for any $1\leq p \leq \infty$.
 \end{cor}

 \begin{proof}
 Recall that the Bessel potential of order $n$ maps $F_{p}^{\alpha,c}$ isomorphically onto $F_{p}^{\alpha-n,c}$. If $\sigma \in S_{1,\delta}^{n}$, by Proposition \ref{prop: composition}, we see that
$$
\sigma(s,\xi )  (1+|\xi |^2 )^{-\frac{n}{2}}\in S_{1,\delta}^{0},
$$
and its corresponding pseudo-differential operator is $T^c_{\sigma}   J^{-n}$. Since $T_\sigma^c = T_\sigma^c   J^{-n}   J^n$, the assertion follows obviously from Theorem \ref{thm:bdd}.
 \end{proof}

\section{Forbidden symbols} \label{section-forbidden}

The purpose of this section is to extend the boundedness results obtained in the previous one to the  pseudo-differential operators with forbidden symbols, i.e. the symbols in the class $S_{1,1}^n$. There are two main differences between these operators and those with symbols in $S_{1,\delta}^n$ with $0\leq \delta<1$. The first one is that  when $\sigma \in S_{1,1}^0$, $T_\sigma^c$ is not necessarily bounded on $L_2(\N)$. The second one is that $S_{1,1}^0$ is not closed under the products and adjoints. Fortunately, if the function spaces have a positive degree of smoothness, the operators with symbols in $S_{1,1}^0$ will be bounded on them.

 In the classical theory, the regularity of operators with forbidden symbols on Sobolev spaces $H_p^\alpha(\R)$, Besov spaces $B_{pq}^\alpha(\R)$ and Triebel Lizorkin spaces $F_{pq}^\alpha(\R)$ with $\alpha>0$ has been widely investigated, see \cite{{Meyer-1980}, {Meyer-1981}, {Bourdaud-1988}, {Runst-1985}, {Torres-1990}}. 
Our first result in this section concerns the regularity of pseudo-differential operators with forbidden
symbols on operator-valued Sobolev spaces. Let us now give some background on these function spaces.

 For $\alpha \in \mathbb{R}, 1\leq p \leq \infty$ and a Banach space $X$, the potential Sobolev space $H_p^\alpha(\R;X)$ is the space of all distributions in $S'(\R;X)$ which have finite Sobolev norm  $\| f\|_{H_p^\alpha}=\| J^\alpha f \|_{L_p(\R;X)}$. It is well known that the potential Sobolev spaces are closely related to Besov spaces. We still use the resolution of the unit $(\varphi_k)_{k\geq 0}$ introduced in \eqref{eq:resolution of unity} to define Besov spaces. Given $\alpha \in \R$ and $1\leq p,q \leq \infty$, the Besov space $B_{p,q}^\alpha (\R;X)$ is defined to be the subspace of $S'(\R;X)$ consisting of all $f$ such that 
$$
\|f\|_{B_{p,q}^\alpha}=\Big(\sum_{k\geq 0}2^{qk \alpha}\| \varphi_k*f\|_{L_p(\R;X)}^q \Big)^\frac{1}{q}<\infty.
$$
The above vector-valued Besov spaces $B_{p,q}^\alpha (\R;X)$ have been studied by many authors, see for instance \cite{Amann-97}.

Instead of the above defined Banach-valued spaces, we prefer to study the operator-valued spaces $H_p^\alpha(\R;L_p(\M))$ and $B_{p,q}^\alpha (\R;L_p(\M))$. Obviously, the main difference is that the Banach space $X$ varies for different $p$.
The following inclusions are easy to check for every $1\leq p \leq \infty$,
$$B_{p,1}^\alpha (\R;L_p(\M)) \subset  H_p^\alpha(\R;L_p(\M))  \subset B_{p,\infty}^\alpha (\R;L_p(\M)).$$
Besov spaces are stable under real interpolation. 
 More precisely, if $\alpha_0,\alpha_1 \in \mathbb{R}$, $\alpha_0 \neq \alpha_1$ and $0<\theta<1$, then 
 \begin{equation}\label{inter-Besov}
\big(B_{p,q_0}^{\alpha_0}(\R;L_p(\M)), B_{p,q_1}^{\alpha_1}(\R; L_p(\M))  \big)_{\theta,q}= B_{p,q}^{\alpha}(\R;L_p(\M)),
 \end{equation}
 for $\alpha=(1-\theta)\alpha_0+\theta \alpha_1, \;p,q,q_0,q_1\in [1,\infty]$. This result can be deduced from its Banach-valued counterpart in \cite{Amann-97}, or in the same way as how we did in \cite{XXY17} for noncommutative tori.

The following lemma states the regularity of pseudo-differential operators with forbidden symbols on $H_2^\alpha(\R;L_2(\M))$ for $\alpha>0$.

 \begin{lem}\label{lem: bdd exotic Sobolev}
Let $\sigma\in S_{1,1}^{0}$. Then $T_\sigma^c$ is bounded on $H_2^\alpha(\R; L_2(\M))$ for any $\alpha > 0$. 
 \end{lem}

 \begin{proof}
Let $(\varphi_j)_{j\geq 0}$ be the resolution of the unit satisfying \eqref{eq:resolution of unity}. It is straightforward to show that $H_2^\alpha(\R; L_2(\M))$ admits an equivalent norm:
 \begin{equation}\label{eq: H_2 equi norm}
 \| f\|_{H_2^\alpha(\R; L_2(\M))}\approx \big( \sum_{j\geq 0}2^{2j\alpha} \| \varphi_j*f \|_{L_2(\N)}^2\big)^\frac{1}{2}= \| f\|_{B_{2,2}^\alpha(\R; L_2(\M))}.
 \end{equation}
 Let $\sigma_k$ with $k\in \mathbb{N}_0$ be the dyadic decomposition of $\sigma$ given in \eqref{eq: sigma_k}. By the support assumptions of $\widehat{\varphi}$ and $\widehat{\varphi}_0$, we have 
 $$
 T_{\sigma_k}^c(f)= T_{\sigma_k}^c(f_k),
 $$
 where $f_k=(\varphi_{k-1}+\varphi_{k}+\varphi_{k+1})*f$ for $k\geq 1$, and $f_0=(\varphi_{0}+\varphi_{1})*f$. 
Applying  Lemma \ref{lem: dyadic kernel estimate} to $K_k$ with $M=0$, we get 
 \begin{equation*}
 \int_{|s-t |\leq 2^{-k}} \| D_s^\gamma  K_k(s,s-t)\|_\M dt  
 \lesssim \int_{|s-t |\leq 2^{-k}}   2^{k(|\gamma |_1+d)}dt \approx 2^{k|\gamma|_1}.
 \end{equation*}
 If $d+1$ is even, applying  Lemma \ref{lem: dyadic kernel estimate} again to $K_k$ with $2M=d+1$, we get
   \begin{equation*}
 \int_{|s-t | > 2^{-k}}\| D_s^\gamma  K_k(s,s-t)\|_\M dt  
 \lesssim  \int_{|s-t | > 2^{-k}}    2^{k(|\gamma |_1-1)} |s-t |^{-d-1}dt \approx 2^{k|\gamma|_1}\, ;
 \end{equation*}  
 if $d+2$ is even, letting $2M=d+2$ in Lemma \ref{lem: dyadic kernel estimate}, we get the same estimate. Therefore, summing up the above estimates of $\int_{|s-t |\leq 2^{-k}} $ and $\int_{|s-t | > 2^{-k}}$, we obtain
 $$
 \int_{\R}\| D_s^\gamma  K_k(s,s-t)\|_\M dt \lesssim 2^{k |\gamma |_1}.
 $$
 Since the estimate of $\| D_s^\gamma  K_k(s,s-t)\|_\M$ is symmetric in $s$ and $t$, the same proof also shows that 
 \[
 \int_{\R}\| D_s^\gamma  K_k(s,s-t)\|_\M ds \lesssim 2^{k |\gamma |_1}.
 \]
 For any $f\in H_2^\alpha(\R; L_2(\M))$ and $k\in\mathbb{N}_0$, there exists $g_k\in L_2(\N)$ with norm one  such that $ \| D_s^\gamma T_{\sigma_k}^c(f)\|_{L_2(\N)} = \big| \tau \int_{\R} D_s^\gamma T_{\sigma_k}^c(f)(s)g_k^*(s)ds\big| $. Then,
 \begin{equation}\label{eq: S11 2}
\begin{split}
& \| D_s^\gamma T_{\sigma_k}^c(f)\|^2_{L_2(\N)} \\
& = \Big|\tau \int_{\R} D_s^\gamma T_{\sigma_k}^c(f)(s)g_k^*(s)ds\Big| ^2 \\
&  =\Big|\tau \int_{\R} \int_{\R} D_s^\gamma  K_k(s,s-t)f_k(t)dt  \,g_k^*(s)\, ds\Big| ^2 \\
&  \leq \tau  \int_{\R} \| D_s^\gamma  K_k(s,s-t)\|_\M |g_k(s)|^2dt ds \cdot \tau \int_{\R} \| D_s^\gamma  K_k(s,s-t)\|_\M  |f_k(t)|^2ds dt\\
& \lesssim 2^{2 k |\gamma |_1} \cdot \|f_k \|^2_{L_2(\N)}.
\end{split} 
 \end{equation}  
Taking $\gamma=0$, the above calculation implies that 
 \begin{equation}\label{eq: S11 1}
\|  T_{\sigma}^c(f)\|_{L_2(\N)} \leq \sum_{k\geq 0}\|  T_{\sigma_k}^c(f)\|_{L_2(\N)}\lesssim  \sum_{k\geq 0}\|f_k \|_{L_2(\N)}\lesssim \| f \|_{B_{2,1}^0},
 \end{equation}
which implies the boundedness of $T_{\sigma}^c$ from $B_{2,1}^0(\R;L_2(\M))$ to $L_2(\N)$.

On the other hand, if we take 
$$
a_0=\varphi_0, \quad a_j(\xi)=(1-\varphi_0(\xi))\frac{\xi_j}{|\xi |^2},
$$
then we get 
$$
1=a_0(\xi)+\sum_{j= 1}^d  a_j(\xi)\xi_j, \quad \forall\, \xi\in \R.
$$
This identity implies 
$$
1=(a_0(\xi)+\sum_{j= 1}^d a_j(\xi)\xi_j)^l=\sum_{|\gamma|_1\leq l}\sigma_\gamma (\xi) \xi^\gamma, \quad \forall \, l\in \mathbb{N}_0, \,\forall\, \xi\in \R,
$$
where the $\sigma_\gamma (\xi) $'s are symbols in $S^{-|\gamma|_1}_{1,0}\subset S^{-|\gamma|_1}_{1,1}$. The above identity allows us to decompose the term $\varphi_j* T_{\sigma_k}^c(f)$ in the following way:
\begin{equation}\label{eq: S11 4}
\varphi_j* T_{\sigma_k}^c(f)=\sum_{|\gamma|_1\leq l} T_{\sigma_\gamma}^c(\varphi_j* D_s^\gamma T_{\sigma_k}^c(f))=\sum_{|\gamma|_1\leq l} T_{\sigma_\gamma^j}^c(D_s^\gamma T_{\sigma_k}^c(f)),
\end{equation}
where $\sigma_\gamma^j=\sigma_\gamma\widehat{\varphi}_j$. Note that the symbol $\sigma_\gamma^j\in S_{1,0}^{-|\gamma |_1}$ for any $j$, and if $|\gamma |_1< l$, $\sigma_\gamma^j \neq 0$ if and only if $j=0$ and $j=1$.
If $j\leq k+1$, by the Plancherel formula and \eqref{eq: S11 2}, we have 
\begin{equation*}
2^{j \alpha } \| \varphi_j*T_{\sigma_k}^c(f) \|_{L_2(\N)}\lesssim 2^{j \alpha } \| T_{\sigma_k}^c(f) \|_{L_2(\N)} \lesssim 2^{j \alpha } \| f_k \|_{L_2(\N)}  \lesssim 2^{k \alpha}\| f_k \|_{L_2(\N)}.
\end{equation*}
If $j\geq k+2$, adapting the  proof of \eqref{eq: S11 2}  with  $\sigma_\gamma^j$ in place of $\sigma_k$, we deduce that 
\begin{equation}\label{eq: S11 5}
\| T_{\sigma_\gamma^j}^c(D_s^\gamma T_{\sigma_k}^c(f))\|_{L_2(\N)} \leq C_ {\gamma}2^{-j|\gamma|_1}\| D_s^\gamma T_{\sigma_k}^c(f)\|_{L_2(\N)}. 
\end{equation}
 For any $| \gamma |_1 <l$, by the previous observation, $\sigma_\gamma^j=0$. Therefore, estimates \eqref{eq: S11 2}, \eqref{eq: S11 4} and \eqref{eq: S11 5} imply that 
\begin{equation*}
\begin{split}
\| \varphi_j*T_{\sigma_k}^c(f) \|_{L_2(\N)}& = \| \sum_{|\gamma|_1= l} T_{\sigma_\gamma^j}^c(D_s^\gamma T_{\sigma_k}^c(f))\|_{L_2(\N)}\\
& \lesssim \sum_{|\gamma|_1= l}  2^{-jl} \| D_s^\gamma T_{\sigma_k}^c(f)\|_{L_2(\N)}\\
&\lesssim \sum_{|\gamma|_1= l} 2^{(k-j)l }\| f_k\|_{L_2(\N)}.
\end{split}
\end{equation*}
Thus, if we take $l$ to be the smallest integer larger than $\alpha$, we have 
\[
2^{j \alpha } \| \varphi_j*T_{\sigma_k}^c(f) \|_{L_2(\N)} \lesssim 2^{(j-k)(\alpha-l) }2^{k \alpha}\| f_k\|_{L_2(\N)} \leq 2^{k \alpha }\| f_k\|_{L_2(\N)}.
\]
Combining the above estimate for $j\geq k+2$ and that for $j\leq k+1$, we get
\begin{equation*}
\sup_{j\in \mathbb{N}_0} 2^{j \alpha } \| \varphi_j*T_{\sigma_k}^c(f) \|_{L_2(\N)} \lesssim 2^{k \alpha}\| f_k\|_{L_2(\N)},
\end{equation*} 
whence, 
\[
\|T_{\sigma_k}^c(f) \|_{B_{2,\infty}^\alpha}\lesssim   2^{k \alpha}\| f_k\|_{L_2(\N)}.
\]
Then by the triangle inequality, we have
\begin{equation}\label{eq: S11 6}
\|T_{\sigma}^c(f) \|_{B_{2,\infty}^\alpha}\leq \sum_{k\geq 0}\|T_{\sigma_k}^c(f) \|_{B_{2,\infty}^\alpha}\lesssim  \sum_{k\geq 0} 2^{k \alpha}\| f_k\|_{L_2(\N)}\lesssim\| f\|_{B_{2,1}^\alpha}, 
\end{equation}
which shows that $T_{\sigma}^c$ is bounded from $B_{2,1}^\alpha(\R;L_2(\M))$ to $B_{2,\infty}^\alpha(\R; L_2(\M))$.

Applying \eqref{eq: S11 1}, \eqref{eq: S11 6} and the real interpolation \eqref{inter-Besov} with $p=2$, $q=2$ and $\alpha_0=0$, $\alpha_1=\alpha$, we obtain the following boundedness:
$$
\|T_{\sigma}^c(f) \|_{B_{2,2}^\beta}\lesssim \| f\|_{B_{2,2}^\beta}, \quad \forall\, \beta >0.
$$
Finally, \eqref{eq: H_2 equi norm} together with the above inequality yields the desired assertion.
\end{proof}
 
 \begin{rmk}
 Even though it is not the main subject of this paper, the regularity of pseudo-differential operators on operator-valued Besov spaces is already obtained in the above proof. Let us record it specifically in the below. Let $1\leq p, q \leq \infty$.
\begin{enumerate}[$\rm (i)$]
\item If $\sigma\in S_{1,\delta}^0$ for some $0\leq \delta\leq 1$, then $T_\sigma^c$ is bounded from $B_{p,1}^0(\R;L_p(\M))$ to $L_p(\N)$, and bounded on $B_{p,q}^\alpha(\R;L_p(\M))$ for any $\alpha>0$.
\item If $\sigma\in S_{1,\delta}^0$ with $0\leq \delta<1$, then $T_\sigma^c$ is bounded on $B_{p,q}^\alpha(\R;L_p(\M))$ for any $\alpha\in \mathbb{R}$.
\end{enumerate} 
Indeed, the argument in \eqref{eq: S11 2} still works for all $1\leq p \leq \infty$. Then we get the boundedness of $T_\sigma^c$  from $B_{p,1}^0(\R;L_p(\M))$ to $L_p(\N)$ as in \eqref{eq: S11 1}. Likewise, we can deduce the $p$-version of \eqref{eq: S11 6}, i.e. the boundedness from $B_{p,1}^\alpha$ to $B_{p,\infty}^\alpha$ for $\alpha>0$. Thus, for $\alpha>0$, the boundedness of $T_\sigma^c$ on $B_{p,q}^\alpha(\R;L_p(\M))$ is ensured by interpolation. If $\delta<1$, by Proposition \ref{prop: composition} and the lifting property of Besov spaces, we get the boundedness for general $\alpha\in \mathbb{R}$. Moreover, we note that, different from the Triebel-Lizorkin spaces, the above assertions hold for $T_\sigma^r$ as well.

\end{rmk}

Since for $\sigma \in S_{1,1}^0$, $T_\sigma^c$ is  not necessarily bounded on $L_2(\N)$, we cannot expect its boundedness on $L_1\big(\M;L_2^c(\R)\big)$. However, by Lemma \ref{lem: bdd exotic Sobolev}, we are able to prove its boundedness on $ L_1\big(\M; H_2^\alpha(\R)^c\big)$ when $\alpha>0$. Note that the classical Sobolev space $H_2^\alpha(\R)$ is a Hilbert space with the inner product $\langle  f,g\rangle =\int _{\R}J^\alpha f(s)\overline{J^\alpha g(s)}ds$. By the definition of Hilbert-valued $L_p$-spaces, we see that $f\in  L_1\big(\M;H_2^\alpha(\R)^c\big)$ if and only if $J^\alpha f \in  L_1\big(\M;L^c_2(\R)\big)$.

 \begin{lem}\label{lem: S11 L1 bdd}
Let $\sigma\in S_{1,1}^{0}$. Then $T_\sigma^c$ is bounded on $L_1\big(\M;H_2^\alpha(\R)^c\big)$ for any $\alpha > 0$. 
 \end{lem}
 
 \begin{proof}
 Following the argument for lemma \ref{lem:L_1 bdd} by replacing $(T_\sigma^c)^*$ with $J^\alpha T_\sigma^c$, we see that $T_\sigma^c$ is bounded on $L_\infty\big(\M;H_2^\alpha(\R)^c\big)$. 
Let $f\in L_1\big(\M;H_2^\alpha(\R)^c\big)$. Then $f$ admits the decomposition 
$$f=g h,
$$
where $\| h\|_{L_1(\M)}=\| f\|_{L_1\big(\M;H_2^\alpha(\R)^c\big)}$ and $\| g \|_{L_\infty\big(\M;H_2^\alpha(\R)^c\big)}=1$. In particular, if we assume that  $A=(\int_{\R}|J^\alpha f(s)|^2 ds)^\frac{1}{2}$ is invertible, we could take $g=fA^{-1}$, $h=A$.
From this  decomposition, we establish the $L_1\big(\M;H_2^\alpha(\R)^c\big)$-norm of $T_\sigma^c(f)$ as follows:
\begin{equation*}
\begin{split}
\| T_\sigma^c(f)\| _{L_1\big(\M;H_2^\alpha(\R)^c\big)} & =\| T_\sigma^c(g)h\| _{L_1\big(\M;H_2^\alpha(\R)^c\big)} \\
& \leq \| T_\sigma^c(g)\| _{L_\infty\big(\M;H_2^\alpha(\R)^c\big)} \| h \|_{L_1(\M)}\\
& \lesssim \| g \| _{L_\infty\big(\M;H_2^\alpha(\R)^c\big)} \| h \|_{L_1(\M)}\\
&= \| f\|_{L_1\big(\M;H_2^\alpha(\R)^c\big)},
\end{split}
\end{equation*}
which implies  that $T_\sigma^c$ is bounded on $L_1\big(\M;H_2^\alpha(\R)^c\big)$.
 \end{proof}

Based on the previous lemma and the atomic decomposition in Theorem \ref{thm: atomic decop T_L}, we are able to study the boundedness of pseudo-differential operators with forbidden symbols on the operator-valued Triebel-Lizorkin spaces $F_{1}^{\alpha,c}(\mathbb{R}^d,\mathcal{M})$.

 \begin{thm}\label{thm:bddforbidden class}
Let $\sigma\in S_{1,1}^{0}$ and $\alpha > 0$. Then  $T^c_{\sigma}$ is bounded  on $F_{1}^{\alpha,c}(\mathbb{R}^d,\mathcal{M})$.
 \end{thm}

\begin{proof}
Let $f\in F_{1}^{\alpha,c}(\R,\M)$. We fix $K,L$ to be two integers such that $K>\alpha+d$ and $L>d$. By the atomic decomposition in Theorem \ref{thm: atomic decop T_L}, $f$ can be written as
$$
f=\sum_{j=1}^\infty (\mu_j b_j+\lambda_jg_j),
$$
where the $b_j$'s are $(\alpha,1)$-atoms and the $g_j$'s are $(\alpha,Q)$-atoms, $\mu_j$ and $\lambda_j$ are complex numbers such that
$$
\sum_{j=1}^\infty (|\mu _j|+|\lambda_j|)\approx \|f\|_{F_1^{\alpha,c}}.
$$
In order to prove the assertion, by the above atomic decomposition,  it suffices to prove that 
$$
 \|  T^c_{\sigma} b \|  _{F_{1}^{\alpha,c}}\lesssim 1 \quad \text{and}\quad  \|  T^c_{\sigma} g \|  _{F_{1}^{\alpha,c}}\lesssim 1 ,
$$
for any $(\alpha,1)$-atom $b$ and $(\alpha,Q)$-atom $g$. We have shown in Corollary \ref{lem:moledule m norm}  that 
\begin{equation}\label{eq: norm of b}
  \|  T^c_{\sigma} b \|  _{F_{1}^{\alpha,c}}\lesssim 1.
\end{equation}
Thus it remains to consider $T^c_\sigma g$. This is the main part of the proof which will be divided into several steps for clarity.

\emph{Step 1.} 
By translation, we may assume that the supporting cube $Q$ of the atom $g$ is centered at the origin. We begin with a split of the symbol $\sigma$: 
Let $h_1$, $h_2$ be two nonnegative  infinitely differentiable functions on $\mathbb{R}^d$ such that $\supp h_1\subset (Q )^c$, $\supp h_2\subset 2Q $ and 
$$
1=h_1(\xi)+h_2(\xi), \quad \forall \, \xi \in \mathbb{R}^d.
$$
For any $(s,\xi)\in \mathbb{R}^d\times \R$, we write 
$$\sigma(s,\xi)=h_1(\xi)\sigma(s,\xi)+h_2(\xi)\sigma(s,\xi) \stackrel{\mathrm{def}}{=}\sigma_1(s,\xi)+\sigma_2(s,\xi).$$ 
It is clear that $\sigma_1$ and $\sigma_2$ are still two symbols in $S_{1,1}^0$, and
\begin{equation}\label{split-Tg}
\|  T^c_{\sigma}g\|  _{F_{1}^{\alpha,c}}\leq \|  T^c_{\sigma_1}g\|  _{F_{1}^{\alpha,c}}+\|  T^c_{\sigma_2}g\|  _{F_{1}^{\alpha,c}}.
\end{equation}

First, we consider the case where the cube $Q$ is of side length one, i.e. $Q=Q_{0,0}$, and deal with the term $\|  T^c_{\sigma_1}g\|  _{F_{1}^{\alpha,c}}$ in the above split.
Let $(\mathcal{X}_j)_{j\in\mathbb{Z}^d}$ be the resolution of the unit defined in \eqref{eq: unit resolution} and $\widetilde{\mathcal{X}}_j= \mathcal{X}_j(2 \cdot)$ for $j\in\mathbb{Z}^d$.  We write   
\begin{equation}
\begin{split}
T^c_{\sigma_1}g & = \sum_{j\in 8Q_{0,0}}T^c_{\sigma_1^j} g+ \sum_{j\notin 8Q_{0,0}}T^c_{\sigma_1^j} g \\
&  \stackrel{\mathrm{def}}{=} G_1+H_1,\label{eq:G H}
\end{split}
\end{equation} 
where $\sigma_1^j(s,\xi)=\sigma_1(s,\xi)\widetilde{\mathcal{X}}_j(s)$.

 We claim that for every $j\in\mathbb{Z}^d $, $T^c_{\sigma_1^j} g$ is the bounded multiple of an $(\alpha,Q_{0,\frac j 2})$-atom (with the convention $Q_{0,\frac j 2}= \frac{ j}{2} +Q_{0,0}$). No loss of generality, we prove the claim just for $j=0$. Applying Lemma \ref{lem: S11 L1 bdd} to the symbol $\sigma_1^0$, we get 
$$
\tau  \big(\int_{\R}  |  J^\alpha T^c_{\sigma_1^0} g(s) |^{2}ds \big)^{\frac{1}{2}}\lesssim \tau \big(\int_{\R} | J^\alpha  g(s)|^2ds\big)^\frac{1}{2} \lesssim |Q_{0,0}|^{-\frac{1}{2}}.
$$
Thus, in order to prove the claim, it remains to show that $T^c_{\sigma_1^0} g$ can be written as the linear combination of subatoms and the coefficients satisfy a certain condition. By Definition \ref{def:smooth T atom}, $g$ admits the following representation:
\begin{equation}
g=\sum_{(\mu,l)\leq (0,0)}d_{\mu, l}a_{\mu, l},\label{eq:g decomp}
\end{equation}
 where the $a_{\mu, l}$'s are $(\alpha ,Q_{\mu, l})$-subatoms and the coefficients  $d_{\mu, l}$'s are
complex numbers satisfying $ \sum_{(\mu,l)\leq (0,0)}  |d_{\mu, l} |^{2} \leq1$.
Then we have
$$
T^c_{\sigma_1^0}g=\sum_{  (\mu,l )\leq (0,0)}d_{\mu, l}T^c_{\sigma_1^0}a_{\mu, l}.
$$
Given $\mu\in \mathbb{N}_0$, let $(\mathcal{X}_{\mu, m})_{m\in \mathbb{Z}^{d} }$ be a sequence of infinitely differentiable functions on $\R$ such that 
\begin{equation} \label{eq: resolution of unit length mu}
1=\sum_{m\in\mathbb{Z}^{d}}\mathcal{X}_{\mu, m}(s),\quad\forall\, s\in\mathbb{R}^{d},
\end{equation}
and each $\mathcal{X}_{\mu, 0}$ is nonnegative,
supported in $2Q_{\mu, 0}$ and $\mathcal{X}_{\mu, m}(s)=\mathcal{X}_{\mu,0}(s-2^{-\mu}m)$. It is the $2^{-\mu}$-dilated version of the resolution of the unit in \eqref{eq: unit resolution}.
We decompose $T^c_{\sigma_1^0}g$ in the following way:
\begin{equation} \label{eq:decomp of image of g}
T^c_{\sigma_1^0}g =  \sum_{\mu=0}^{\infty}{\sum_{m}} \mathcal{X}_{\mu, m}\sum_{l}d_{\mu, l}T^c_{\sigma_1^0}a_{\mu, l}.
\end{equation}
Observe that the only $m$'s that contribute to the above sum $\sum_m$ are those $m\in\mathbb{Z}^{d}$ such
that $2Q_{\mu, m}\cap Q_{0,0}\neq \emptyset$, so $Q_{\mu, m}\subset 2Q_{0,0}$.
Thus, we obtain  the decomposition
\begin{equation}\label{eq:decomp G}
T^c_{\sigma_1^0}g  =\sum_{(\mu,m )\leq (0,0)}D_{\mu, m}G_{\mu, m} ,
\end{equation}
where
\begin{equation*}
\begin{split}
D_{\mu, m}  &=  \big(\sum_{l}  |d_{\mu, l} |^{2}  (1+  | m-l | )^{-  (d+1 )} \big)^{\frac{1}{2}},\\
G_{\mu, m}  &=  \frac{1}{D_{\mu, m}}\mathcal{X}_{\mu, m}\sum_{l}d_{\mu, l}T^c_{\sigma_1^0}a_{\mu, l}.
\end{split}
\end{equation*}
It is evident that
$$
\big(\sum_{(\mu,m)\leq (0,0)}  |D_{\mu, m} |^{2}\big)^{\frac{1}{2}}\lesssim \big(\sum_{(\mu,l)\leq (0,0)}  |d_{\mu, l} |^{2}\big)^{\frac{1}{2}}\leq1.
$$
Now we show that the $G_{\mu, m}$'s are bounded multiple of $(\alpha,Q_{\mu, m})$-subatoms. Firstly, we have $\supp G_{\mu, m} \subset \supp \mathcal{X}_{\mu, m} \subset 2Q_{\mu, m}$. Secondly, by the Cauchy-Schwarz inequality,
\begin{equation}\label{eq:estimate atom}
\begin{split}
&  \tau  \big(\int_{2Q_{\mu, m}}  |\sum_{l}d_{\mu, l}T^c_{\sigma_1^0}a_{\mu, l}(s) |^{2}ds \big)^{\frac{1}{2}} \\
&  \lesssim \big(\sum_{l}  |d_{\mu, l} |^{2}  (1+  | m-l  | )^{-  (d+1 )} \big)^{\frac{1}{2}}\\
&  \;\;\;\; \cdot \sum_{l}  (1+ | m-l | )^{\frac{1-M}{2}}\tau \big(\int_{2Q_{\mu, m}}(1+2^\mu (s-2^{-\mu}l))^{d+M} |T^c_{\sigma_1^0}a_{\mu, l}(s) |^{2}ds \big)^{\frac{1}{2}}.
\end{split}
\end{equation}
If we take $M=2L+1$, since $L>d$, we have $\frac{1-M}{2}<-d$. Applying Lemma \ref{lem:moledule}, we get
\begin{equation*}
\tau  (\int_{\R}  |G_{\mu, m}(s) |^{2}ds )^{\frac{1}{2}}  \lesssim  \sum_{l}  (1+ | m-l | )^{\frac{1-M}{2}} |Q_{\mu,l}|^\frac{\alpha}{d}\lesssim  |Q_{\mu,m}|^\frac{\alpha}{d}.
\end{equation*}
Similarly, the derivative estimates in Lemma \ref{lem:moledule} ensure that
$$
\tau  (\int  |D^{\gamma}G_{\mu, m}(s) |^{2}ds )^{\frac{1}{2}}\lesssim  |Q_{\mu, m} |^{\frac{\alpha}{d}-\frac{  |\gamma |_1}{d}},   \quad \forall\, |\gamma |_1\leq [\alpha]+1.
$$
Since $\alpha>0$, no moment cancellation for subatoms is required. Thus,
we have proved that the $G_{\mu, m}$'s are bounded multiple of $(\alpha,Q_{\mu, m})$-subatoms, then the claim is proved. Therefore, $G_1$ in \eqref{eq:G H} is the finite sum of $(\alpha, Q_{0,j})$-atoms, which yields $\|  G_1 \|  _{F_{1}^{\alpha,c}}\lesssim 1$ by Theorem \ref{thm: atomic decop T_L}.

The term $H_1$ in \eqref{eq:G H} is much easier to handle.  Observe that $H_1$ corresponds to the symbol $\sigma(s,\xi)\sum_{j\notin 8Q_{0,0}}\widetilde{\mathcal{X}}(s)$, whose $s$-support is in $(6Q_{0,0})^c$. Thus, we apply Corollary  \ref{cor: moledule norm} directly to get that 
$$
\| H_1 \|_{F_1^{\alpha,c}}\lesssim 1. 
$$

\emph{Step 2.} Let us  consider now the case where the supporting cube $Q$ of $g$ has side length less than one.  As above, we may still assume that $Q$ is centered at the origin. Let $g$ be an $(\alpha,Q_{k,0})$-atom  with $k\in\mathbb{N}$.
Then $g$ is given by
$$
g=\sum_{(\mu,l)\leq (k,0)}d_{\mu, l}a_{\mu, l} \quad \text{with} \; \sum_{(\mu,l)}  |d_{\mu, l} |^{2} \leq  |Q_{k,0} |^{-1}=2^{kd}.
$$
We normalize $g$ as
\begin{equation*}
\begin{split}
h  & =2^{k(\alpha-d)}g(2^{-k}\cdot)\\
 &  =  \sum_{(\mu,l)\leq (k,0)}2^{-\frac{kd}{2}}d_{\mu, l}2^{k(\alpha-\frac{d}{2})}a_{\mu, l}(2^{-k}\cdot)\\
&  =  \sum_{(\mu,l)\leq (k,0)}\widetilde{d}_{\mu, l}\widetilde{a}_{\mu, l},
\end{split}
\end{equation*}
where $\widetilde{a}_{\mu, l}=2^{k(\alpha-\frac{d}{2})}a_{\mu, l}(2^{-k}\cdot)$
and $\widetilde{d}_{\mu, l}=2^{-\frac{kd}{2}}d_{\mu, l}$. Then it is easy to see that  each $\widetilde{a}_{\mu, l}$
is an $(\alpha,Q_{\mu-k,l})$-subatom and $h$ is an $(\alpha,Q_{0,0})$-atom.
Define $\sigma_{1,k}  (s,\xi )=\sigma_1(2^{-k}s,2^{k}\xi)$, then we have
\begin{equation}\label{dilation-g-h}
\begin{split}
T^c_{\sigma_1}g(s)  &= \int_{\R} \sigma_1(s,\xi)\widehat{g}(\xi)e^{2\pi {\rm i} s \cdot \xi}d\xi\\
&=   2^{-k\alpha}\int_{\R} \sigma_1(s,\xi)\widehat{h}(2^{-k}\xi)e^{2\pi {\rm i} s\cdot  \xi}d\xi\\
&=  2^{k(d-\alpha)}\int_{\R} \sigma_{1,k}  (2^{k}s,\xi )\widehat{h}(\xi)e^{2\pi {\rm i} \,2^{k}s\cdot \xi}d\xi\\
& =   2^{k(d-\alpha)}T^c_{\sigma_{1,k} }h(2^{k}s).
\end{split}
\end{equation}
 Since the $\xi$-support of $\sigma_1$ is away from the origin, we have
 $$
 \|D_s^\gamma D_\xi^\beta \sigma_{1,k} (s,\xi )\|_{\M} \leq C_{\gamma,\beta}|\xi |^{|\gamma|_1 -|\beta|_1}
 \approx C_{\gamma,\beta}(1+|\xi |)^{|\gamma|_1-|\beta|_1}, \quad \forall\, k\in \mathbb{N}.
 $$
 Thus, $\sigma_{1,k} $ is still a symbol in the class $S_{1,1}^{0}$.
Then, applying the result for $(\alpha,Q_{0,0})$-atoms obtained in Step 1 to
the symbol $\sigma_{1,k} $, we  get $\|  T^c_{\sigma_{1,k} }h \| _{F_1^{\alpha,c}}\lesssim 1$. 
In order to return back to the $F_1^{\alpha,c}$-norm of $T_{\sigma_1}^c g$, by \eqref{dilation-g-h}, we need a dilation argument. Since $\alpha>0$, we can invoke the characterization of $F_1^{\alpha,c}$-norm in \cite[Corollary~3.10]{XX18}:
 $$\| f\|_{F_{1}^{\alpha,c}} \approx \|  f\| _{1}+\Big\|  (\int _0^\infty\e^{-2\alpha}|  \varphi_\e*f| ^2\frac{d\e}{\e})^\frac{1}{2}\Big\| _{1},$$
 where $\varphi_\e=\F^{-1}(\varphi(\e\cdot))$. For $\lambda>0$, we have $\|  f(\lambda\cdot)\| _{1}=\lambda^{-d}\|  f\| _{1}$, and 
 $$\Big\|  (\int _0^\infty\e^{-2\alpha}|  \varphi_\e*f(\lambda\cdot)| ^2\frac{d\e}{\e})^\frac{1}{2}\Big\| _{1}=\lambda^{\alpha -d}\,\Big\|  (\int _0^\infty\e^{-2\alpha}|  \varphi_\e*f| ^2\frac{d\e}{\e})^\frac{1}{2}\Big\| _{1}$$ 
since  $( \varphi _\e*f(\lambda\cdot))(s)= \varphi _{\lambda\e}*f(\lambda s)$. Taking $\lambda= 2^k$, we deduce
\begin{equation*}
\begin{split}
\| T^c_{\sigma_1}g \|_{F_{1}^{\alpha,c}} &\approx \|  T^c_{\sigma_1}g \| _{1}+\Big\|  (\int _0^\infty\e^{-2\alpha}|  \varphi_\e*T^c_{\sigma_1}g | ^2\frac{d\e}{\e})^\frac{1}{2}\Big\| _{1} \\
&=2^{k(d-\alpha)} \Big(\|  T^c_{\sigma_{1,k} }h(2^{k}\cdot) \| _{1}+\Big\|  (\int _0^\infty\e^{-2\alpha}|  \varphi_\e*T^c_{\sigma_{1,k} }h(2^{k}\cdot) | ^2\frac{d\e}{\e})^\frac{1}{2}\Big\| _{1} \Big)\\
&=2^{k(d-\alpha)} \Big(  2^{-kd}\|  T^c_{\sigma_{1,k} }h \| _{1}+2^{k(\alpha-d)}\Big\|  (\int _0^\infty\e^{-2\alpha}|  \varphi_\e*T^c_{\sigma_{1,k} }h | ^2\frac{d\e}{\e})^\frac{1}{2}\Big\| _{1} \Big)\\
&\lesssim \|  T^c_{\sigma_{1,k} }h \| _{F_1^{\alpha,c}}.
\end{split}
\end{equation*}
This ensures
$$
\|  T^c_{\sigma_1}g \| _{F_1^{\alpha,c}} \lesssim \|  T^c_{\sigma_{1,k} }h \| _{F_1^{\alpha,c}}\lesssim 1.
$$

\emph{Step 3.} It remains to deal with the term with symbol $\sigma _2$ in \eqref{split-Tg}. Note that $\sigma_2= h_2(\xi) \sigma(s,\xi)$ with $\sigma\in S_{1,1}^0$ and $\supp  h_2 \in 2 Q $. Then for $\delta<1$, say $\delta =\frac{9}{10}$, we have $\sigma_2\in S_{1,\delta}^0$. Indeed, by definition, we have, for every $s\in \mathbb{R}$,
\begin{equation*}
\begin{split}
\| D_s^{\gamma } D_{\xi}^\beta  \sigma_2(s,\xi)\|_\M & \lesssim \sum_{\beta_1+ \beta_2  =\beta}  \|D_s^\gamma  D_\xi ^{\beta_1}  \sigma(s, \xi )  \cdot D^{\beta_2} h_2(\xi) \|_\M  \\
& \leq   \sum_{\beta_1+ \beta_2  =\beta}  C_{\gamma, \beta_1} (1+|\xi | )^{|\gamma   |_1  -  |\beta_1 | _1 } \cdot |D^{\beta_2} h_2(\xi) | .
\end{split}
\end{equation*}
But since $h_2$ is an infinitely differentiable function with support $2Q $, it is clear that for $\xi  \in 2Q $,
$$ (1+|\xi | )^{|\gamma   |_1  -  |\beta_1 | _1 } \leq C_\gamma  (1+|\xi | )^{\frac{9}{10}|\gamma   |_1  -  |\beta_1 | _1 },\quad\mbox{and} \quad     | D^{\beta_2} h_2(\xi) | \leq C_{\beta_2} (1+|\xi| ) ^{-|\beta_2|_1} .$$
Putting these two inequalities into the estimate of $\| D_s^{\gamma } D_{\xi}^\beta  \sigma_2(s,\xi)\|_\M $, we obtain
$$\| D_s^{\gamma } D_{\xi}^\beta  \sigma_2(s,\xi)\|_\M  \leq C_{\gamma, \beta} (1+|\xi|) ^{\frac{9}{10}   |\gamma|_1    - |\beta|_1},$$
which yields $\sigma_2\in S_{1,\frac{9}{10}}^0$. Therefore, it follows from Theorem \ref{thm:bdd} that $\|T_{\sigma_2}^c g\|_{F_1^{\alpha,c }}\lesssim \| g\|_{F_1^{\alpha,c }}$ for $g\in F_1^{\alpha,c }(\R,\M)$. Combining this with the estimates in the first two steps, we complete the proof of the theorem.
\end{proof}

If $\sigma\in S_{1,1}^0$, it is not true in general that $(T_\sigma^c)^*$ corresponds to a symbol in the class $S_{1,1}^0$. However, if we assume additionally this last condition, duality and interpolation arguments will give the following boundedness of $T_\sigma^c$:
\begin{thm}\label{thm: bdd forbidden p>1}
Let $1< p<\infty$ and $\sigma\in S_{1,1}^{0}$, $\alpha \in \mathbb{R}$. If  $(T_\sigma^c)^*$ admits a symbol in the class $S_{1,1}^0$, then  $T^c_{\sigma}$ is bounded  on $F_{p}^{\alpha,c}(\mathbb{R}^d,\mathcal{M})$.
\end{thm}

A similar argument as in the proof of Corollary \ref{cor:bdd} gives the following results concerning the symbols in $S_{1,1}^{n}$ with $n\in \mathbb{R}$.

\begin{cor}
Let $n\in \mathbb{R}$, $\sigma\in S_{1,1}^{n}$ and $\alpha>0$. If $\alpha > n$, then $T^c_{\sigma}$ is bounded  from $F_{1}^{\alpha,c}(\R,\M)$ to $F_{1}^{\alpha-n,c}(\R,\M)$.
\end{cor}

\begin{cor}
Let $n,\alpha$ and $ \sigma$ be the same as above, and $1< p<\infty$.  If  $(T_\sigma^c)^*$ admits a symbol in the class $S_{1,1}^n$, then $T^c_{\sigma}$ is bounded from $F_{p}^{\alpha,c}(\R,\M)$ to $F_{p}^{\alpha-n,c}(\R,\M)$.
\end{cor}

\section{Applications}\label{section-app}

The main target of this section is to apply the results obtained previously to pseudo-differential operators over quantum tori. Our strategy is to use the transference method introduced in \cite{N-R} via operator-valued pseudo-differential operators on the usual torus $\T$. Let us begin with the latter case by a periodization argument.

\subsection{Applications to tori}
In this subsection, $\M$ still denotes a von Neumann algebra with a normal semifinite faithful trace $\tau$, but $\N=L_\infty(\T)\overline{\otimes}\M$.

We identify $\T$ with the unit cube $\I=[0, 1)^d$ via $(e^{2 \pi \mathrm{i} s_1},\cdots , e^{2 \pi \mathrm{i} s_d})\leftrightarrow(s_1, \cdots, s_d)$. Under this identification, the addition in $\I$ is the usual addition modulo $1$ coordinatewise; an interval of $\I$ is either a subinterval of $\mathbb{I}$ or a union  $[b, 1] \cup [0, a]$ with $0<a<b<1$, the latter union being the interval $[b-1, a]$ of $\mathbb{I}$ (modulo $1$). So the cubes of $\I$ are exactly those of $\T$.  Accordingly, functions on $\T$ and $\I$ are identified too.

Recall that $ \varphi$ is a Schwartz function satisfying \eqref{condition-phi}. Then for every $m\in \mathbb{Z}^d \setminus \{0\}$,
$$\sum_{j\in \mathbb{Z}}\varphi(2^{-j} m) =\sum_{j\geq 0}\varphi(2^{-j} m)=1.$$
This tells us that in the torus case $\{\varphi(2^{-j} \cdot)\}_{j\geq 0}$ gives a resolvent of the unit. According to this, we make a slight change of the notation that we used before:
 $$\varphi^{(j)} = \varphi(2^{-j}\cdot),\;\forall\, j\geq 0.$$
Let $\varphi_j=\mathcal{F}^{-1}(\varphi^{(j)})$ for any $j\geq 0$. Now we periodize $\varphi_j $ as
$$
\widetilde{\varphi}_j(z)=\sum_{m\in \mathbb{Z}^d}\varphi_j (s+m) \quad \text{with}\quad z=(e^{2\pi {\rm i} s_1},\ldots,e^{2\pi {\rm i} s_d}).
$$
Then, we can easily see that  $\widetilde{\varphi}_j$ admits the following Fourier series:
\begin{equation}\label{eq:same fourier trans}
\widetilde{\varphi}_j(z)=\sum_{m\in \mathbb{Z}^d}\varphi(2^{-j}m)z^m.
\end{equation}
Thus, for any $f\in \mathcal{S}'(\R;L_1(\M)+\M)$, whenever it exists,
$$
\widetilde{\varphi}_j * f(z)=\int_{\T} \widetilde{\varphi}_j (zw^{-1})f(w)dw =\sum_{m\in \mathbb{Z}^d}\varphi(2^{-j}m) \widehat{f}(m) z^m  \quad z\in \T.
$$
The following definition was given in \cite[Section~4.5]{XXY17}.

\begin{defn}
Let $1\leq p<\infty$ and $\alpha\in \R$. The column operator-valued Triebel-Lizorkin space $F_p^{\alpha,c}(\T, \M)$ is defined to be
$$
F_p^{\alpha,c}(\T, \M)=\lbrace f\in \mathcal{S}'(\T;L_1(\M)):\|  f\|  _{F_p^{\alpha,c}}<\infty\rbrace,
$$
where $$
\|  f\|  _{F_p^{\alpha,c}}=\|  \widehat{f}(0)\| _{L_p(\M)}+\big\| (\sum_{j\geq 0}2^{2j\alpha}|  \widetilde{\varphi}_j * f | ^2) ^\frac{1}{2}\big\| _{L_p(\N)}.
$$
\end{defn}
The row and mixture spaces $F_p^{\alpha,r}(\T, \M)$  and $F_p^{\alpha}(\T, \M)$, and the corresponding spaces for $p=\infty$ are defined similarly to the Euclidean case.

By the discussion before \eqref{eq:same fourier trans},  if we identify a function $f$ on $\T$ as a 1-periodic function $f_{\rm{pe}}$ on $\R$, then the convolution $\widetilde{\varphi}_j*f$ on $\T$ coincides with the convolution $\varphi_j*f_{\rm{pe}}$ on $\R$. More precisely:
$$
\widetilde{\varphi}_j*f(z)=\varphi_j*f_{\rm{pe}}(s) \quad \mbox{with}\quad z=(e^{2\pi {\rm i} s_1},\cdots, e^{2\pi {\rm i} s_d}).
$$
By the almost orthogonality of the Littlewood-Paley decomposition given in \eqref{eq:resolution of unity},  we get the following easy equivalent norm of $F_p^{\alpha,c}(\I,\M)$:
\begin{equation*}
\|  f_{\rm{pe}}\| _{F_p^{\alpha,c}(\I,\M)}
\approx  \|{\phi}_0*f_{\rm{pe}}\|_p+ \big\|  (\sum_{j\geq 0}2^{2j\alpha}|  \varphi_{j}*f(z)| ^2)^\frac{1}{2}\big\| _p,
\end{equation*}
where $\widehat{\phi}_0 (\xi) =1-\sum _{j\geq 0} \varphi(2^{-j}\xi) $.
Since $\widehat{\phi}_0$ is supported in $\{\xi: |  \xi |  \leq 1\}$ and $\widehat{\phi}_0(\xi)=1$ if $|  \xi |  \leq \frac{1}{2}$, it then follows that
$$
 \|\check{\phi}_0*f_{\rm{pe}}\|_p=\|  \widehat{f}(0) \| _p\,.
$$
Hence, combining the estimates above, we have
\begin{equation}\label{eq:equi}
\|  f\| _{F_p^{\alpha,c}(\T,\M)}\approx \|  f_{\rm{pe}}\| _{F_p^{\alpha,c}(\I,\M)}
\end{equation}
Thus $F_p^{\alpha,c}(\T, \M)$ embeds into $F_p^{\alpha,c}(\R, \M)$ isomorphically. The equivalence \eqref{eq:equi} allows us to reduce the treatment of $\T$ to that of $\R$; and by periodicity, all the functions considered now are restricted on $\I$.

We are not going to state the properties of $F^{\alpha,c}_p (\T,\M)$ specifically, and refer the reader to \cite[Section~4.5]{XXY17} for similar results on quantum torus.

\medskip

Let us turn to the study of toroidal symbols. In the discrete case, the derivatives degenerate into discrete difference operators.
Let $\sigma:$ $\mathbb{Z}^{d}\rightarrow\mathcal{M}$.
For $1\leq  j\leq d$, let $e_j$ be the $j$-th canonical basis of $\R$.
We define the forward and backward partial difference operators $\Delta_{m_{j}}$
and $\overline{\Delta}_{m_{j}}$: 
$$
\Delta_{m_{j}}\sigma(m):=\sigma(m+e_{j})-\sigma(m),\quad\overline{\Delta}_{m_{j}}\sigma(m):=\sigma(m)-\sigma(m-e_{j}),
$$
and for any $\beta\in\mathbb{N}_{0}^{d}$,
$$
\Delta_{m}^{\beta}:=\Delta_{m_{1}}^{\beta_{1}}\cdots\Delta_{m_{d}}^{\beta_{d}},\quad\overline{\Delta}_{m}^{\beta}:=\overline{\Delta}_{m_{1}}^{\beta_{1}}\cdots\overline{\Delta}_{m_{d}}^{\beta_{d}}.
$$

\begin{defn}
Let $0\leq\delta,\rho\leq1$ and  $\gamma,\beta\in\mathbb{N}_0^{d}$. Then the toroidal symbol class $S^n_{\rho,\delta}(\T \times \mathbb{Z}^d)$ consists of those $\M$-valued functions $\sigma(s,m)$
which are smooth in $s$ for all $m\in \mathbb{Z}^d$, and satisfy
$$
  \|  D_{s}^{\gamma}\Delta_{m}^{\beta}\sigma(s,m) \|  _{\mathcal{M}}\leq C_{\alpha,\beta,m}  (1+  |m | )^{n-\rho  |\beta |_1+\delta  |\gamma |_1}.
$$
\end{defn}

\begin{defn}
Let $\sigma \in S^n_{\rho,\delta}(\T \times \mathbb{Z}^d)$. For any $f\in\mathcal{S}^{\prime}  (\mathbb{T}^{d};L_{1}  (\mathcal{M} ) )$,
we define the corresponding toroidal pseudo-differential operator
as follows:
$$
T^c_{\sigma}f(s)=\sum_{m \in \mathbb{Z}^d} \sigma (s,m)\widehat{f}(m)e^{2\pi {\rm i}s\cdot m}.
$$
\end{defn}
When studying the toroidal pseudo-differential operators $T_\sigma^c$ on $\T$, especially its action on operator-valued Triebel-Lizorkin spaces on $\T$, a very useful tool is to extend the toroidal symbol to a symbol defined on $\T\times \R$, which reduces the torus case to the Euclidean one. This allows us to use the arguments in the last section. The extension of scalar-valued toroidal symbol has been well studied in \cite{R-T}. With some minor modifications, the arguments used in \cite{R-T} can be adjusted to our operator-valued setting.

The following lemma is taken from \cite{R-T}. Denote by $\delta_0(\xi)$ the Kronecker delta function at $0$, i.e., $\delta_0(0)=1$ and $\delta_0(\xi)=0$ if $\xi \neq 0$.
\begin{lem}\label{lem:Ruzhansky}
For each $\beta \in\mathbb{N}_0^d$, there exists a function $\phi_\beta\in \mathcal{S}(\R)$ and a function $\zeta \in \mathcal{S}(\R)$ such that
\begin{gather*}
\sum_{k\in\Z}\zeta (s+k)\equiv 1, \\
\widehat{\zeta }\mid _{\Z} (\xi)=\delta_0(\xi)\quad \mbox{and}\quad D _{\xi}^\beta (\widehat{\zeta })(\xi)=\overline{\Delta}_\xi^\beta \phi_\beta(\xi),
\end{gather*}
for any $\xi \in \R$.
\end{lem}
Now let us give the operator-valued analogue of Theorem 4.5.3 in \cite{R-T}.
\begin{lem}\label{lem: extension}
Let  $0\leq \rho,\delta \leq 1$ and  $n\in \mathbb{R}$. A symbol ${\sigma}\in S^n_{\rho,\delta}(\T \times \Z)$ is a toroidal symbol  if and only if there exists an Euclidean symbol $\widetilde{\sigma} \in S^n_{\rho,\delta}(\T \times \R)$ such that ${\sigma}=\widetilde{\sigma}\mid_{\T\times \Z}$.
\end{lem}

\begin{proof}
We first prove the ``if'' part. Let $\widetilde{\sigma} \in S^n_{\rho,\delta}(\T \times \R)$. If $|  \beta |  _1 =1$, then by the mean value theorem for vector-valued functions, we have
$$
\| \Delta_{m}^{\beta}D_{s}^{\gamma}\sigma(s,m)\|_\M  \leq \sup_{0\leq \theta \leq 1} \big\| \partial _{\xi}^\beta D_s^\gamma\widetilde{\sigma}(s,m+\theta \beta)\big\|_\M  .
$$
For a general multi-index $\beta \in \mathbb{N}_0^d$, we use induction. Writing $\beta=\beta'+\delta_j$ and using the induction hypothesis, we get
\begin{equation*}
\begin{split}
\|  \Delta_{m}^{\beta}D_{s}^{\gamma}{\sigma}(s,m) \|_\M   & = \|  \Delta_{m}^{\delta_j}( \Delta_{m}^{\beta'}D_{s}^{\gamma}\widetilde{\sigma}(s,m) )\| \\
&  \leq  \sup_{0\leq \theta \leq 1} \|  \partial_j (\Delta_{m}^{\beta'}D_{s}^{\gamma}\widetilde{\sigma}(s,m+\theta\delta_j))\|_\M  \\
& =    \sup_{0\leq \theta \leq 1}\|  \Delta_{m}^{\beta'} (\partial_j D_{s}^{\gamma}\widetilde{\sigma}(s,m+\theta\delta_j))\| _\M \\
& \leq   \sup_{0\leq \theta' \leq 1} \|  D_{\xi}^{\beta'} \partial_j D_{s}^{\gamma}\widetilde{\sigma}(s,m+\theta'\beta)\|_\M  \\
& = \sup_{0\leq \theta' \leq 1} \|  D_{\xi}^{\beta}  D_{s}^{\gamma}\widetilde{\sigma}(s,m+\theta'\beta)\|_\M  .
\end{split}
\end{equation*}
 Thus we deduce that
\begin{equation*}
\begin{split}
\|  \Delta_{m}^{\beta}D_{s}^{\gamma}\sigma(s,m) \|_\M   & \leq \sup_{0\leq \theta' \leq 1} \|  D_{\xi}^{\beta}  D_{s}^{\gamma}\widetilde{\sigma}(s,m+\theta'\beta)\|_\M  \\
 &  \leq  C'_{\alpha,\beta,m}(1+|  m| )^{n-\rho| \beta| _1+\delta| \gamma| _1}.
\end{split}
\end{equation*}

Now let us show the ``only if'' part. In the proof of Theorem 4.5.3 in \cite{R-T}, the desired Euclidean symbol is constructed with the help of the functions in Lemma \ref{lem:Ruzhansky}. We can transfer directly the arguments in \cite{R-T} to our setting. But we still include a proof for completeness. Let $\zeta \in \mathcal{S}(\R)$ be as in Lemma \ref{lem:Ruzhansky}. Define a function $\widetilde{\sigma}: \T\times \R \rightarrow \M$ by
$$
\widetilde{\sigma}(s,\xi)=\sum_{m\in \Z}\widehat{\zeta }(\xi-m){\sigma}(s,m).
$$
Thus, ${\sigma}=\widetilde{\sigma}\mid_{\T\times \Z}$. Moreover, using summation by parts, we have
\begin{equation*}
\begin{split}
\|   D_{s}^{\gamma}D_{\xi}^{\beta}\widetilde{\sigma}(s,\xi)\| _\M  &  = \big\| \sum_{m\in \Z}D_{\xi}^{\beta}\widehat{\zeta }(\xi-m)D_s^{\beta}{\sigma}(s,m)\big\|_\M  \\
&  =  \big\|  \sum_{m\in \Z}\overline{\Delta}_\xi^\beta\phi_\beta(\xi-m)D_s^{\gamma}\sigma(s,m)\big\|_\M  \\
&=  \| (-1)^{|  \beta | _1}\sum_{m\in \Z} \phi_\beta(\xi-m)\Delta_m^{\beta}D_s^{\gamma}\sigma(s,m) \|_\M  \\
& \lesssim   \sum_{m\in \Z} |  \phi_{\beta}(\xi-m) |  (1+|  m| )^{n-\rho| \beta| _1 +\delta| \beta| _1}\\
 & \lesssim  \sum_{m \in \Z} |  \phi_{\beta}(\xi-m) |  (1+|  \xi -m| )^{n-\rho| \beta| _1 +\delta| \gamma| _1}(1+|  \xi| )^{n-\rho| \beta| _1 +\delta| \gamma| _1}\\
& \lesssim   (1+|  \xi| )^{n-\rho| \beta| _1 +\delta| \gamma| _1},
\end{split}
\end{equation*}
whence, $\widetilde{\sigma}\in S^n_{\rho,\delta}(\T \times \R)$.
\end{proof}

\begin{thm}\label{thm:tori}
Let $\sigma \in S^0_{1,\delta}(\T \times \mathbb{Z}^d)$ and $\alpha\in \mathbb{R}$. Then
\begin{itemize}
\item  If $0\leq \delta <1$, then $T^c_{\sigma}$ is a bounded operator on $F_p^{\alpha,c}(\T,\M)$ for every $1\leq p \leq \infty$.
\item If $\delta=1$ and $\alpha> 0$, then $T^c_{\sigma}$ is a bounded operator on $F_1^{\alpha,c}(\T,\M)$.
\item If $\delta=1$ and $(T_\sigma^c)^*$ admits a symbol in the class $S_{1,1}^0(\T\times \mathbb{Z}^d)$, then $T^c_{\sigma}$ is bounded on $F_{p}^{\alpha,c}(\T,\M)$ for any  $1< p<\infty$.
\end{itemize}
\end{thm}

\begin{proof}
By Lemma \ref{lem: extension}, there exists $\widetilde{\sigma}$ in $S^n_{\rho,\delta}(\T \times \R)$ such that ${\sigma}=\widetilde{\sigma}\mid_{\T\times \Z}$. Let $f\in F_p^{\alpha,c}(\T,\M)$.
By the identification $\T \approx \I$, for any $z\in \T$, there exists $s\in \I$ such that
\begin{equation*}
\begin{split}
T^c_{\sigma} f(z) &=  \sum_{m\in \Z}\sigma(s,m)\widehat{f}(m)e^{2\pi {\rm i}s\cdot m}\\
&=  \int _{\R} \widetilde{\sigma}(s,\xi)\widehat{f}_{\rm{pe}}(\xi)e^{2\pi {\rm i} s\cdot \xi}d\xi = T^c_{\widetilde{\sigma}}{f}_{\rm{pe}}(s).
\end{split}
\end{equation*}
Now we apply Theorems \ref{thm:bdd}, \ref{thm:bddforbidden class} and \ref{thm: bdd forbidden p>1} to the symbol $\widetilde{\sigma}$ and $f_{\rm{pe}}$. Then by the equivalence \eqref{eq:equi}, we get the boundedness of $T_\sigma^c$ on $F_p^{\alpha,c}(\T,\M)$.
\end{proof}

\subsection{Applications to quantum tori}
We now apply the above results to the quantum case. To this end, we first recall the relevant definitions.
Let $d\geq 2$ and $\theta=(\theta_{kj})$ be a real skew symmetric $d\times d$-matrix. The associated $d$-dimensional noncommutative
torus $\mathcal{A}_{\theta}$ is the universal $C^*$-algebra generated by $d$
unitary operators $U_1, \ldots, U_d$ satisfying the following commutation
relation
 $$U_k U_j = e^{2 \pi \mathrm{i} \theta_{k j}} U_j U_k,\quad j,k=1,\ldots, d.$$
We will use standard notation from multiple Fourier series. Let  $U=(U_1,\cdots, U_d)$. For $m=(m_1,\cdots,m_d)\in\mathbb{Z}^d$, define
 $$U^m=U_1^{m_1}\cdots U_d^{m_d}.$$
 A polynomial in $U$ is a finite sum
  $$
   x =\sum_{m \in \Z}\alpha_{m} U^{m}\quad \text{with}\quad
 \alpha_{m} \in \mathbb{C}.
 $$
The involution algebra
$\mathcal{P}_{\theta}$ of all such polynomials is
dense in $\mathcal{A}_{\theta}$. For any polynomial $x$ as above, we define
$$
\tau(x)=\alpha_0.
$$
Then $\tau$ extends to a  faithful  tracial state $\tau$ on $\mathcal{A}_{\theta}$.  Let  $\T_{\theta}$ be the $w^*$-closure of $\mathcal{A}_{\theta}$ in  the GNS representation of $\tau$. This is our $d$-dimensional quantum torus. The state $\tau$ extends to a normal faithful tracial state on $\T_{\theta}$ that will be denoted again by $\tau$.  Note that if $\theta=0$, then  $\T_{\theta}=L_\infty(\T)$ and $\tau$ coincides with the integral on $\T$ against normalized Haar measure $dz$.

Any $x\in L_1(\T_\theta)$ admits a formal
Fourier series:
 $$
 x \sim \sum_{m \in\Z} \widehat{x} ( m ) U^{m}\;\text{ with }\; \widehat{x}( m) = \tau((U^m)^*x).
 $$

In \cite{N-R},  a transference method has been introduced to overcome the full noncommutativity of quantum tori and to use methods of operator-valued harmonic analysis.  Let $\mathcal{N}_{\theta} = L_{\infty} (\T) {\overline{\otimes}}\T_{\theta}$, equipped with the tensor trace $\nu = \int dz {\otimes} \tau$. For each $z
\in \T,$ define $\pi_{z}$ to be the isomorphism of
$\T_{\theta}$ determined by
\begin{equation}\label{pi_z}
 \pi_{z} (U^{m}) = z^{m} U^{m } = z_1^{m_1} \cdots z_d^{m_d}U_1^{m_1} \cdots U_d^{m_d}.
 \end{equation}
This isomorphism preserves the
trace $\tau$. Thus for every $1\leq p < \infty$,
 $$
 \|\pi_{z} (x) \|_p = \|x\|_p,\; \forall\, x\in L_p(\mathbb{T}^d_{\theta}).
 $$

The main points of the transference method are contained in the following lemma from \cite{CXY2013}.
\begin{lem}\label{prop:TransLp}
 \begin{enumerate}[\rm(1)]
\item  Let $1\leq  p \leq \infty$. For any $x \in L_p (\T_{\theta})$, the function
$\widetilde{x}: z \mapsto \pi_z(x)$ is continuous from
$\T$ to $L_p(\T_{\theta})$ $($with respect to the w*-topology for $p=\infty)$.
\item If $x\in L_p (\T_{\theta}),$ then $\widetilde{x} \in L_p (\N_{\theta})$ and $\|  \widetilde{x} \| _p = \|  x\| _p,$ that is, $x \mapsto \widetilde{x}$ is an isometric embedding from $L_p (\T_{\theta})$ into $L_p (\N_{\theta})$.

\item Let $\widetilde{\T_{\theta}} = \{ \widetilde{x}: x \in \T_{\theta}\}.$ Then $\widetilde{\T_{\theta}}$ is a von Neumann subalgebra of $\mathcal{N}_{\theta}$ and the associated conditional expectation is given by
$$
\mathbb{E} (f)(z) = \pi_{z} \Big ( \int_{\T} \pi_{\overline{w}}
 \big [ f( w )\big ] dw \Big ),\quad z\in\T, \; f \in \N_{\theta}.$$
Moreover, $\mathbb{E}$ extends to a contractive projection from $L_p(\N_{\theta})$ onto $L_p(\widetilde{\T_{\theta}})$ for $1\leq p\leq \infty$.
\end{enumerate}
\end{lem}
To avoid complicated notation, we will use the same notation for the derivation for the quantum tori $\T_\theta$ as for functions on $\T$. For every $1\leq j\leq d$, define the derivation to be the  operator $\partial_j$ satisfying:
$$
\partial_j(U_j)=2\pi {\rm i} U_j \quad \mbox{and} \quad \partial_j(U_k)=0 \mbox{ for } k\neq j.
$$
Given $m\in \mathbb{N}_0^d$, the associated partial derivation $D^m$ is $\partial _1^{m_1} \cdots \partial_d^{m_d}$. We keep using the resolvent of unit given by functions in \eqref{eq:same fourier trans}. The Fourier multiplier on $\T_\theta$ with symbol $\varphi(2^{-j}\cdot)$ is then
 $$
\widetilde{\varphi}_j * x =\sum_{m\in \mathbb{Z}^d}\varphi(2^{-j}m)\, \widehat{x}(m) U^m   .
$$

The analogue of Schwartz class on the quantum torus is given by
$$
\mathcal{S}(\T_\theta)=\{ \sum_{m\in \Z}a_mU^m: \{a_m\}_{m\in \Z} \text{ rapidly decreasing}\}.
$$
This is a $w^*$-dense $*$-subalgebra of $\T_\theta$ and contains all polynomials. It is equipped with a structure of Fr\'echet $*$-algebra, and has a locally convex topology induced by a family of semi-norms.
We denote the tempered distribution on $\T_\theta$ by $\mathcal{S}'(\T_\theta)$ which is the space of all continuous linear functional on $\mathcal{S}(\T_\theta)$.
Then by duality, both partial derivations and the Fourier transform extend to $\mathcal{S}'({\T_\theta})$. Triebel-Lizorkin spaces on the quantum torus are defined and  well studied in \cite{XXY17}. Let us recall the definition.

\begin{defn}
Let $1\leq p<\infty$ and $\alpha\in \R$.
The column Triebel-Lizorkin space $F_p^{\alpha,c}(\T_\theta)$ is defined by
$$
F_p^{\alpha,c}(\T_\theta)=\{ x\in \mathcal{S}'(\T_\theta): \|  x \| _{F_p^{\alpha,c}}<\infty\},
$$
where
$$
\|  x \| _{F_p^{\alpha,c}}= |  \widehat{x}(0) |  +\big\|  (\sum_{j\geq 0}2^{2j\alpha}|  \widetilde{\varphi}_j*x|  ^2)^\frac{1}{2} \big\| _p.
$$
The row space $F_p^{\alpha,r}(\T_\theta)$ and mixture space $F_p^{\alpha}(\T_\theta)$, and the case $p=\infty$ are then defined similarly.
\end{defn}

The transference method in Lemma \ref{prop:TransLp} allows us to connect $F_p^{\alpha}(\T_\theta)$ with operator-valued spaces $F_p^{\alpha,c}(\T,\T_\theta)$. The result is

\begin{lem}\label{lem:iden q to t}
For any $1\leq p< \infty$, the map $x\mapsto \widetilde{x}$ extends to an isometric embedding from $F_p^{\alpha,c}(\T_\theta)$ to $F_p^{\alpha,c}(\T,\T_\theta)$ with complemented image.
\end{lem}

Let us introduce toroidal symbol classes and pseudo-differential operators on $\T_{\theta}$. The following definitions were also given in \cite{L-N-P}.

\begin{defn}
Let $0\leq\delta,\rho\leq1$, $n\in \mathbb{R}$ and  $\gamma,\beta\in\mathbb{N}_0^{d}$ be multi-indices.  Then the toroidal symbol class $S^n_{\T_{\theta},\rho,\delta}(\Z)$ consists of those functions $\sigma: \Z \rightarrow \T_{\theta}$ which satisfy
$$
\|D^{\beta} (\Delta_{m}^{\gamma}\sigma(m))\|  \leq C_{\beta,\gamma}  (1+  |m  |  )^{n-\rho  |\gamma  |_1+\delta  |\beta  |_1}, \quad \forall\, m\in \Z.
$$
\end{defn}

\begin{defn}
Let $\sigma \in S^n_{\T_{\theta},\rho,\delta}(\Z)$. For any $x\in \T_{\theta}$,
we define the corresponding toroidal pseudo-differential operator on $\T_{\theta}$
as follows:
$$
T^c_{\sigma}x=\sum_{m \in \Z}\sigma (m)\widehat{x}(m)U^m.
$$
\end{defn}

Now we are ready to prove the mapping property of pseudo-differential operators on quantum torus.

 \begin{thm}
Let $\sigma \in S^0_{\T_{\theta},1,\delta}(\Z)$ and $\alpha\in \mathbb{R}$. Then
\begin{itemize}
\item If $0\leq \delta <1$, then $T^c_{\sigma}$ is a bounded operator on $F_p^{\alpha,c}(\T_\theta)$ for every $1\leq p \leq \infty$.
\item If $\delta=1$ and $\alpha> 0$, then $T^c_{\sigma}$ is a bounded operator on $F_1^{\alpha,c}(\T_\theta)$.
\item If $\delta=1$ and $(T_\sigma^c)^*$ admits a symbol in the class $S^0_{\T_{\theta},1,\delta}(\Z)$, then $T^c_{\sigma}$ is bounded on $F_{p}^{\alpha,c}(\T_\theta)$ for any  $1< p<\infty$.
\end{itemize}
\end{thm}

\begin{proof}
Recall that $\pi_z$ denotes the isomorphism of $\T_\theta$ determined by \eqref{pi_z}.
We claim that, given $m\in \Z$, the function $z\mapsto \pi_z(\sigma(m))$ from $\T$ to $\T_\theta$ satisfies
\begin{equation}\label{eq:q to t}
\|  D_z^\gamma \Delta_m^\beta \pi_z(\sigma(m)) \|  \leq C_{\gamma,\beta}(1+|  m| )^{n+\delta |  \gamma | _1 -\rho |  \beta | _1}.
\end{equation}
Since $\pi_z$ commutes with the derivations on $\T_\theta$, we have $D^\gamma \Delta^\beta \pi_z\sigma(m)=\pi_z(D^\gamma \Delta^\beta \sigma(m))$. Therefore,
\begin{equation*}
\|  D^\gamma \Delta^\beta \pi_z\sigma(m) \|  = \|  \pi_z(D^\gamma \Delta^\beta \sigma(m)) \|  \leq \|  D^\gamma \Delta^\beta \sigma(m) \|  \leq C_{\gamma,\beta}(1+|  m| )^{n+\delta |  \gamma | _1 -\rho |  \beta | _1}.
\end{equation*}
Denote $\widetilde{\sigma}(z,m)=\pi_z(\sigma(m))$ for $(z,m)\in \T\times\Z$ and consider the pseudo-differential operator $T^c_{\widetilde{\sigma}}$. Combining \eqref{eq:q to t} and Theorem \ref{thm:tori}, we obtain the boundedness of $T^c_{\widetilde{\sigma}}$ on $F_p^{\alpha,c}(\T,\T_\theta)$. Moreover, for any polynomial $x$ on $\T_\theta$ and $f(z)=\pi_{z}(x)$, we have
\begin{equation*}
\begin{split}
T^c_{\widetilde{\sigma}}f(z)&=  \sum_{m\in \Z}\widetilde{\sigma}(z,m)\widehat{f}(m)z^m\\
& =  \sum_{m\in \Z}\pi_z(\sigma(m))\widehat{x}(m)U^mz^m\\
& =  \sum_{m\in \Z}\pi_z(\sigma(m)\widehat{x}(m)U^m)=\pi_z(T^c_{\sigma}(x)).
\end{split}
\end{equation*}
Finally, by Lemma \ref{lem:iden q to t} and Theorem \ref{thm:tori}, we have
\begin{equation*}
\begin{split}
\|  T^c_{\sigma}(x)\| _{F_p^{\alpha,c}(\T_\theta)} &=\|  \pi_\cdot (T^c_{\sigma}(x))\| _{F_p^{\alpha,c}(\T,\T_\theta)}=\|  T^c_{\widetilde{\sigma}}f\| _{F_p^{\alpha,c}(\T,\T_\theta)}\\
&\lesssim \|  f\| _{F_p^{\alpha,c}(\T,\T_\theta)}= \|  x\| _{F_p^{\alpha,c}(\T_\theta)}.
\end{split}
\end{equation*}
All the three assertions are proved. 
\end{proof}

Finally, let $S^n_{\rho,\delta}(\Z)$ be the scalar-valued toroidal symbol class, consisting of those functions $\sigma: \Z \rightarrow \mathbb{C}$ which satisfy
$$
|D^{\beta} (\Delta_{m}^{\gamma}\sigma(m))|  \leq C_{\beta,\gamma}  (1+  |m  |  )^{n-\rho  |\gamma  |_1+\delta  |\beta  |_1}, \quad \forall\, m\in \Z.
$$
In this setting, it is evident that $T_\sigma^c$ and $T_\sigma^r$ give the same pseudo-differential operator on $\mathbb T_\theta^d$, denoted by $T_\sigma$ simply. Then, we have the following

 \begin{cor}
Let $\sigma \in S^0_{1,\delta}(\Z)$ and $\alpha\in \mathbb{R}$. Then
\begin{itemize}
\item If $0\leq \delta <1$, then $T_{\sigma}$ is a bounded operator on $F_p^{\alpha,c}(\T_\theta)$, $F_p^{\alpha,r}(\T_\theta)$ and $F_p^{\alpha}(\T_\theta)$ for every $1\leq p \leq \infty$.
\item If $\delta=1$ and $\alpha> 0$, then $T_{\sigma}$ is a bounded operator on $F_1^{\alpha,c}(\T_\theta)$, $F_1^{\alpha,r}(\T_\theta)$ and $F_1^{\alpha}(\T_\theta)$.
\end{itemize}
\end{cor}

\noindent{\bf Acknowledgements.} The authors are greatly indebted to Professor Quanhua Xu for having suggested to them the subject of this paper, for many helpful discussions and very careful reading of this paper. The authors are partially supported by the the National Natural Science Foundation of China (grant no. 11301401).

\end{document}